\documentclass[twoside]{amsart}

\usepackage{a4}

\usepackage[toc,page]{appendix}

\usepackage{amsmath}

\usepackage{amssymb}

\usepackage{amscd}

\usepackage{amsfonts}

\usepackage{amsthm}

\usepackage{array}%
 
\usepackage[english]{babel}

\usepackage{booktabs}%

\usepackage{bbm}

\usepackage{bm}%

\usepackage{blkarray}

\usepackage{bmpsize}

\usepackage{cancel}  

\usepackage[font=footnotesize]{caption}

\usepackage{color}

\usepackage{enumerate}

\usepackage{enumitem}		

\usepackage{extarrows}%

\usepackage{graphicx}

\usepackage[driverfallback=dvipdfm, colorlinks=true, citecolor=magenta]{hyperref}
\usepackage[capitalise]{cleveref} 

\usepackage{indentfirst}	

\usepackage[modulo,right,pagewise]{lineno} 

\usepackage[scr=boondox]{mathalfa}%

\usepackage{mathdots}%

\usepackage{mathrsfs}

\usepackage{mathtools}		
\usepackage{mathabx}		

\usepackage{multirow}%

\usepackage{siunitx}%

\usepackage{thmtools}

\usepackage{tikz}
\usepackage{tikz-cd}
\usetikzlibrary{fadings}
\usetikzlibrary{patterns}
\usetikzlibrary{shadows.blur}
\usetikzlibrary{shapes}

\usepackage{verbatim}

\usepackage{tabularx}%

\newtheorem{theorem}{Theorem}[section]
\newtheorem{mainthm}{Theorem}

\newtheorem{proposition}[theorem]{Proposition}
\newtheorem{corollary}[theorem]{Corollary}
\newtheorem{corollarymain}[mainthm]{Corollary}
\newtheorem{lemma}[theorem]{Lemma}

\theoremstyle{definition}

\newtheorem{definition}[theorem]{Definition}

\newtheorem{example}[theorem]{Example}

\newtheorem{remark}[theorem]{Remark}

\numberwithin{equation}{section}
\numberwithin{theorem}{section}

\newcommand{\wtilde}[1]{\widetilde{#1}}
\newcommand{\mbb}[1]{\mathbb{#1}}
\newcommand{\msf}[1]{\mathsf{#1}}

\newcommand{\Hy}{\mathbb{H}}

\newcommand{\bbP}{\mathbb{P}}
\renewcommand{\P}{\mathbb{P}}

\newcommand{\R}{\mathbb{R}}

\newcommand{\Z}{\mathbb{Z}}

\newcommand{\calA}{\mathcal{A}}

\newcommand{\calB}{\mathcal{B}}

\newcommand{\calC}{\mathcal{C}}

\newcommand{\calD}{\mathcal{D}}
\newcommand{\Dc}{{\mathcal D}}

\newcommand{\calF}{\mathcal{F}}

\newcommand{\calH}{\mathcal{H}}
\newcommand{\Hc}{{\mathcal H}}

\newcommand{\calK}{\mathcal{K}}
\newcommand{\calL}{\mathcal{L}}

\newcommand{\calM}{\mathcal{M}}

\newcommand{\calO}{\mathcal{O}}

\newcommand{\scrD}{\mathscr{D}}

\newcommand{\scrH}{\mathscr{H}}

\newcommand{\frakm}{\mathfrak{m}}
\newcommand{\frakn}{\mathfrak{n}}

\newcommand{\al}{\alpha}
\renewcommand{\b}{\beta}
\newcommand{\G}{\Gamma}

\newcommand{\gam}{\gamma}
\newcommand{\D}{\Delta}
\newcommand{\Del}{\Delta}

\newcommand{\del}{\delta}
\newcommand{\ep}{\epsilon}
\renewcommand{\epsilon}{\varepsilon}

\newcommand{\Thet}{\Theta}

\renewcommand{\k}{\kappa}

\newcommand{\Lam}{\Lambda}

\newcommand{\lam}{\lambda}
\newcommand{\m}{\mu}
\newcommand{\n}{\nu}

\newcommand{\Sig}{\Sigma}

\newcommand{\sig}{\sigma}

\renewcommand{\P}{\Psi}
\newcommand{\p}{\psi}

\DeclareMathOperator{\PSL}{PSL}

\DeclareMathOperator{\Id}{Id}

\DeclareMathOperator{\Dil}{Dil}

\DeclareMathOperator{\supp}{supp}

\DeclareMathOperator{\CAT}{CAT}

\DeclareMathOperator{\Leb}{Leb}

\DeclareMathOperator{\Par}{Par}

\newcommand{\gpr}[2]{(#1 \mid#2)}

\def\({\left(}
\def\){\right)}

\newcommand{\bs}{\backslash}
\newcommand{\Ra}{\Rightarrow}
\newcommand{\ra}{\rightarrow}

\newcommand{\longhookrightarrow}{\lhook\joinrel\relbar\joinrel\rightarrow}

\newcommand{\wh}{\widehat}
\newcommand{\wt}{\widetilde}
\newcommand{\ov}[1]{\overline{#1}}

\newcommand{\vp}{\mathbf{p}}

\newcommand{\vv}{\mathbf{v}}

\newcommand{\vw}{\mathbf{w}}
\newcommand{\vvu}{\mathbf{u}}
\newcommand{\vx}{\mathbf{x}}
\newcommand{\vy}{\mathbf{y}}
\newcommand{\vz}{\mathbf{z}}

\newcommand{\UG}{\msf{U}\G}
\newcommand{\BMS}{\msf{BMS}}

\DeclareMathOperator{\MC}{MC}
\newcommand{\pG}{\partial \G}

\newcommand{\ppG}{\partial^2 \G}

\newcommand{\wtg}{\wt{\msf{g}}}
\newcommand{\vg}{\msf{g}}
\newcommand{\vc}{\msf{c}}
\newcommand{\sfP}{\mathsf{\Psi}}

\begin{document}

\title[A geometric correspondence for reparameterizations]{A geometric correspondence for reparameterizations of geodesic flows}

\author{Stephen Cantrell}
\address{Mathematical Institute, University of St Andrews, North Haugh, St Andrews, KY16 9SS, Scotland}
\email{sjc33@st-andrews.ac.uk}

\author{Dídac Martínez-Granado}
\address{Department of Mathematics, University of Luxembourg, Av. de la Fonte 6, L-4364 Esch-sur-Alzette, Luxembourg}
\email{didac.martinezgranado@uni.lu}

\author{Eduardo Reyes}
\address{Departamento de Matem\'aticas, Pontificia Universidad Cat\'olica de Chile}
\email{eduardoreyes@uc.cl}


\begin{abstract}
For any non-elementary, torsion-free hyperbolic group, we provide a correspondence between the left-invariant Gromov-hyperbolic metrics on the group that are quasi-isometric to a word metric, and continuous reparameterizations of the associated Mineyev's flow space. From this correspondence, we produce the first examples of continuous reparameterizations of geodesic flows on negatively curved manifolds with all periodic orbits having integer lengths. For surface and free groups, this also yields isometric actions on Gromov-hyperbolic spaces on which loxodromic elements are precisely the non-simple elements. Key ingredients in our proof are an analysis of the geometry of Mineyev's flow space (such as the metric-Anosov property recently proven by Dilsavor), and the density of Green metrics in the moduli space of (symmetric) metrics on the group.
\end{abstract}

\maketitle


\section{Introduction}
Let $M$ be a closed, negatively curved Riemannian manifold with fundamental group $\G$. Write $\msf{g}=(\msf{g}_t)_{t\in \R}$ for the geodesic flow on the unit tangent bundle $\msf{T}^1 M$. In this work we construct a correspondence between 
the following objects:
\begin{itemize}
    \item \emph{Reparameterizations} of the geodesic flow, i.e. continuous flows on $\msf{T}^1 M$ whose oriented flow lines coincide with those of $\msf{g}$.
    \item Left-invariant pseudometrics\footnote{Through this paper pseudometrics are not necessarily symmetric. They are nonnegative functions $d:X\times X\ra \R$ satisfying $d(x,x)=0$ and the triangle inequality $d(x,z)\leq d(x,y)+d(y,z)$.}  on $\G$ which are Gromov-hyperbolic and quasi-isometric to a word metric.
\end{itemize}
We consider these objects up to natural equivalence relations (coming from scaling conditions and flow conjugacy) and write $\Par(\msf{T}^1 M,\msf{g})$ for the collection of (classes of) reparameterizations and $\scrH_\G^{++}$ for the collection of (classes of) pseudometrics.
 Our first main result, stated as \Cref{mainthm:dictionary_manifold} below, establishes a bridge between these spaces. Consequently, we uncover new phenomena and examples, both in the setting of reparameterizations and that of pseudometrics. 

As a motivating example, suppose that $M$ is a surface and let $\rho:\G\ra \PSL(n,\R)$ be a \emph{Hitchin} representation (see for instance, the survey \cite{wienhard.ICM}). Then, setting
\begin{equation}\label{eq:dHitchin}
    d_{\rho}(g,h) \coloneq \log \sigma_1(\rho(g^{-1}h)) \quad \text{ for }g,h\in \G,
\end{equation}
where $\sig_1$ denotes the first singular value, defines a pseudometric on $\G$ that is typically asymmetric for $n\geq 3$. By the work of Sambarino \cite{sambarino.quantitative,sambarino.orbital}, there also exists a \emph{H\"older} reparameterization $\sfP^\rho=(\sfP_t^\rho)_{t\in \R}$ of $(\msf{T}^1M,\msf{g})$ that reflects the geometry of $d_{\rho}$. For an arbitrary reparameterization $\sfP$, its \emph{period function} $\msf{L}_\sfP:\G \ra \R$ maps a non-trivial element $g\in \G$ to the time $\sfP$ takes to run through the unique closed orbit in the free homotopy class of $g$. In the current case, the reparameterization $\sfP^\rho$ satisfies
\[
\msf{L}_{\sfP^\rho}[g]=\log \lam_1(\rho(g)) \quad \text{ for all } g\in \G,
\]
where $\lam_1$ denotes the spectral radius.

A similar duality holds for the more general class of \emph{Anosov representations} of hyperbolic groups into real algebraic semi-simple Lie groups \cite{labourie.anosov,guichard-wienhard}. In this case, the work of Bridgeman--Canary--Labourie--Sambarino \cite{BCLS.pressure,BCLS.rootflows} associates H\"older reparameterizations to several pseudometrics induced by these representations, where, for an arbitrary (non-elementary, torsion-free) hyperbolic group $\G$, the geodesic flow is replaced by \emph{Mineyev's flow space} $(\UG,\msf{g})$ \cite{gromov.hypgroups,mineyev.flow}. See also \Cref{sec:mineyevflow}.

The H\"older regularity of such reparameterizations allows for the implementation of tools from thermodynamic formalism. Using such techniques it is possible to deduce results of various flavours for Anosov representations. These include counting asymptotics ~\cite{sambarino.quantitative,sambarino.orbital}, existence of pressure metrics and analyticity of Hausdorff dimensions of limit sets ~\cite{BCLS.pressure,BCLS.rootflows}, entropy rigidity \cite{PS17:Entropy}, and mixing results~\cite{sambarino.orbital,chow-sarkar.local,chow-sarkar.exp,DMS.I}. 

The space $\Par(\msf{T}^1M,\msf{g})$ is defined as the space of continuous reparameterizations of the flow $\msf{g}$, up to weak conjugacy and rescaling (see \Cref{def:PAR} for the precise notions). Since weak conjugacy implies conjugacy among H\"older reparameterizations, the works \cite{BCLS.pressure,CDPW} imply that many subspaces of Anosov representations (such as Hitchin components) embed into $\Par(\msf{T}^1M,\msf{g})$. When $M$ is a closed hyperbolic surface, this allows us to interpret $\Par(\msf{T}^1 M,\msf{g})$ as the  ``highest Teichm\"uller space'', a perspective discussed by Tholozan \cite{tholozan} and based on ideas by Sullivan \cite{Sullivan91:LinkingUniversalities}. 

Other than Tholozan's work in dimension 2, not much is known about the geometry of $\Par(\msf{T}^1M,\msf{g})$ for general $M$, especially beyond the H\"older setting. For instance, it is unclear what functions can be the period function of a continuous reparameterization, whereas Liv\v{s}ic's theorem implies that the period of a H\"older reparameterization $\sfP$ is \emph{non-arithmetic}. That is, no $a\in \R$ satisfies $\msf{L}_{\sfP}[g]\in a\Z$ for all $g\in \G$ (see also \Cref{ex:nonarithmetic}). 

On the metric side, the geometry of the space $\scrH_\G^{++}$ has been studied in recent years \cite{oregonreyes.metric,cantrell-reyes.manhattan,CRS}. 
 To be more precise, let $\G$ be a hyperbolic group with identity element $o$, and let $\calH_\G^{++}$ be the space of left-invariant functions $\G\times \G \ra \R$ that are
\begin{itemize}
    \item left-invariant;
    \item Gromov hyperbolic (in the sense of \Cref{eq:gromovhyp}); and,
    \item quasi-isometric to a word metric for a finite generating set.
\end{itemize}
See \Cref{sec:preliminaries} and \Cref{def:HMP} for the relevant concepts involved. Elements of $\calH_\G^{++}$ are called \emph{(properly positive) hyperbolic metric potentials}, a notion introduced in \cite{CRS}. This space is an asymmetric extension of the space $\calD_\G\subset \calH_\G^{++}$ of \emph{symmetric} pseudometrics, originally introduced by Furman \cite{furman}. The space $\scrH^{++}_\G$ is then the set of equivalence classes $[\p]$ for $\p\in \calH^{++}_\G$, where $\p,\p'\in \calH^{++}_\G$ are equivalent if the function $|\p-c\p'|$ is bounded for some real number $c>0$. We refer to $\scrH^{++}_\G$ as the space of  \emph{(properly positive) metric structures}. We let $\scrD_\G\subset \scrH_\G^{++}$ be the subspace induced by $\calD_\G$.

Any proper and cobounded isometric action of $\G$ on a geodesic, symmetric metric space induces a metric structure in $\scrD_\G$: if $\G$ acts on $X$ this way, then the metric structure induced by the pseudometric $d(g,h)=d_X(gx,hx)$ on $\G$ is independent of the base point $x\in X$. Natural asymmetric examples in $\scrH_\G^{++}$ arise from Anosov representations. For  instance, the pseudometric $d_{\rho}$ from \Cref{eq:dHitchin} belongs to $\calH_\G^{++}$ \cite[Lemma~7.1]{cantrell-tanaka.invariant}, and the same holds for several other pseudometrics related to Anosov representations~\cite[Corollary~4.8]{dey-kapovich},\cite[Corollary~6.8]{CRS},\cite[Proposition~9.7]{KMG26:Horoboundary}. 

Any hyperbolic metric potential $\p\in \calH_\G^{++}$ is within bounded distance from a genuine pseudometric and so has a well-defined \emph{stable translation length} $\ell_\p:\G \ra \R$ given by
\[\ell_\p[g]=\lim_{k\to \infty}\frac{\p(o,g^k)}{k} \quad \text{ for }g\in \G.\]
The homothety class of $\ell_{\p}$ completely determines the point $[\p]\in \scrH_\G^{++}$ (see \Cref{subsec:hypgroups}), and it can be used to define a natural metric $\Del$ on $\scrH^{++}_\G$ (\Cref{eq:formulaDel}), extending the symmetrized Thurston distance on Teichm\"uller space when $\G$ is a hyperbolic surface group. By \cite{BCLS.pressure,CDPW}, we also obtain embeddings of several subspaces of 
Anosov representations into $\scrH_\G^{++}$.

\subsection{A dictionary between metrics and reparameterizations}\label{subsec:introdictionary} Our first theorem gives a map from $\Par(\msf{T}^1 M,\msf{g})$ into $\MC(\scrH_\G^{++})$, the metric completion of $\scrH_\G^{++}$ with respect to $\Del$. This map extends the embeddings of spaces of Anosov representations into both spaces. In the statement below, $\MC(\msf{X})$ denotes the completion of a subspace $\msf{X}\subset \scrH_\G^{++}$ (see also \Cref{def:MCHandD} for the related spaces $\MC(\calD_\G)$, $\MC(\calH_\G^{++})$, and the translation lengths for $\p\in \MC(\calH_\G^{++})$). Also, we let $\Par^s(\msf{T}^1M,\msf{g})$ and $\Par^h(\msf{T}^1M,\msf{g})$ denote the subspaces of $\Par(\msf{T}^1M,\msf{g})$ represented by the reparameterizations that are reversible and H\"older respectively; see \Cref{subsec:reparameterizations} for the details.

\begin{mainthm}[Dictionary metrics-reparameterizations,  manifold case]\label{mainthm:dictionary_manifold}
    Let $M$ be a closed negatively curved Riemannian manifold with fundamental group $\G$, and let $\msf{g}$ be the geodesic flow on its unit tangent bundle $\msf{T}^1M$. Then there exists an injection \begin{equation*}
        \Thet:\Par(\msf{T}^1M,\msf{g}) \ra \MC(\scrH^{++}_\G)
    \end{equation*} satisfying the following:
    \begin{enumerate}
    \item (Duality) If $\sfP$ is a reparameterization of $\vg$, then $\Thet([\sfP])=[\psi]$ for some $\psi\in \MC(\calH^{++}_\G)$ such that 
    \[\msf{L}_{\sfP}[g]=\ell_{\p}[g] \quad \text{ for any non-trivial }g\in \G.\]
    \item (Bijection on symmetric metrics) $\Thet(\Par^s(\msf{T}^1M,\msf{g}))=\MC(\scrD_\G)$.
    \item (H\"older reparameterizations in the interior) $\Thet(\Par^h(\msf{T}^1M,\msf{g}))\subset \scrH_\G^{++}$.
    \end{enumerate}
\end{mainthm}
\begin{table}[ht]
\centering
\[
\begin{array}{ccc}
\Par(\mathsf{T}^1M,\msf{g}) 
& \longhookrightarrow 
& \MC(\scrH^{++}_\Gamma) \\
\scriptstyle \text{reparameterizations} && \scriptstyle \text{pseudometrics} \\[0.4em]
\cup && \cup \\[0.4em]
\Par^s(\mathsf{T}^1M,\msf{g}) 
& \longleftrightarrow 
& \MC(\scrD_\Gamma) \\
\text{\scriptsize
\begin{tabular}{c}
reversible \\
reparameterizations
\end{tabular}}
&& 
\scriptstyle \text{symmetric pseudometrics}
\end{array}
\]
\caption{Inclusions given by \Cref{mainthm:dictionary_manifold}.}
\end{table}

That is, for any reparameterization $\sfP$ of the geodesic flow there is a pseudometric $\p$ on $\G$ whose stable length $\ell_\p$ corresponds to the periods on closed orbits $\msf{L}_\sfP$ of $\sfP$. 
This result is especially notable in that it does not require any a priori regularity of the reparameterization, extending beyond the H\"older framework. In fact, there are at least two known families of metric structures in $\scrH_\G^{++}$ that cannot be encoded by H\"older reparameterizations.

\begin{example}\label{ex:nonarithmetic}
    Suppose $d$ is a left-invariant metric on $\G$ with \emph{arithmetic} translation length: there exists $a\in \R$ such that $\ell_d[g]\in a\Z$ for any $g\in \G$. Then \emph{no} H\"older reparameterization  $\sfP$ of $\msf{g}$ satisfies $\msf{L}_{\sfP}=\ell_d$, as this would contradict Liv\v{s}ic's theorem \cite{livsic} (cf.~\cite[Theorem~1.3]{GMM} and \cite[Theorem~1.4]{cantrell.mixing}). This is the case when $d$ is a word metric for a finite generating set of $\G$ \cite{cannon} or $d$ is induced by the geometric action of $\G$ on a $\CAT(0)$ cube complex with the combinatorial metric \cite{haglund}.
\end{example}

\begin{example}Suppose $M$ is a surface and  $\mu$ is a filling (symmetric) \emph{geodesic current} on $\G$ \cite{bonahon.annals}; see also \Cref{subsec:currents}. By \cite{BIPP, cantrell-reyes.manhattan, derosa-martinezgranado},  there exists a pseudometric $d_\mu\in \calD_\G$ such that $\ell_{d_\mu}[g]=i(\mu,\al_{[g]})$ for all $g\in \G$, where $i$ denotes Bonahon's intersection number and $\al_{[g]}$ is the rational geodesic current associated to $g$. It is not clear for which $\mu$ we can recover $\ell_{d_\mu}$ as the period function of a H\"older reparameterization. On the one hand, there is a dense subset of filling currents whose length spectra arise as periods of H"older reparameterizations~\cite{JyothisMartinezGranado:NegCurvCurrents}; on the other hand, weighted filling closed curves are dense among filling geodesic currents, and each induces an arithmetic translation length function.
\end{example}

Since these pseudometrics induce points in $\scrD_\G$, we can,  as a consequence of \Cref{mainthm:dictionary_manifold}~(2), relate them to continuous reparameterizations of the geodesic flow.

\begin{corollarymain} \label{cor:combinatorial_examples}
   Let $M$ be and $\G$ be as in \Cref{mainthm:dictionary_manifold}, and let $\ell:\G\ra \R$ be either:
   \begin{enumerate}
       \item the stable translation length of a word metric for a finite generating set;
       \item the stable translation length for a geometric action on a $\CAT(0)$ cube complex with the combinatorial metric; or
       \item the intersection pairing with a filling geodesic current when $M$ is a surface.
   \end{enumerate}
  Then there exists a reparameterization $\sfP$ of the geodesic flow on $\msf{T}^1M$ such that
   \[\msf{L}_{\sfP}[g]=\ell[g] \quad \text{ for any non-trivial }g\in \G.\]
\end{corollarymain}

By rescaling the reparameterization in items (1) or (2) above, we obtain continuous reparameterizations of the geodesic flow of a closed negatively curved manifold \emph{all whose periods are integers}. To the best of the authors' knowledge, these are the first examples of reparameterizations with this property. It is worth remarking that all continuous reparameterizations of the geodesic flow are weak mixing by a result of Arnol'd, see \cite[Section 23]{anosov}.

Our dictionary from \Cref{mainthm:dictionary_manifold} and \Cref{cor:combinatorial_examples} also hold for arbitrary non-elementary, torsion-free hyperbolic group when considering reparameterizations of Mineyev's flow space $(\UG,\msf{g})$; see \Cref{thm:dictionary}. This dictionary also allows us to use the dynamical structure of the space of reparameterizations to deduce new properties of the space $\scrH_\G^{++}$.
For instance, we deduce non-completeness of $\scrH_\G^{++}$, justifying the consideration of the metric completion in \Cref{mainthm:dictionary_manifold}.

\begin{mainthm}\label{thmmain.H_Gnotcomplete}
    Let $\G$ be a non-elementary, torsion-free hyperbolic group. Then $(\scrH^{++}_\G,\Del)$ is not complete. 
\end{mainthm}

The proof uses the dictionary. By Ledrappier's correspondence \cite{ledrappier,sambarino.report}, reparameterizations of $(\UG,\msf{g})$ give rise to positive continuous functions $F:\UG\ra \R$ (\Cref{prop:par->C}). The work of Israel \cite{israel,iommi-velozo} then yields a dense subset of such $F$ with infinitely many equilibrium states.  \Cref{thm:dictionary} relates the associated reparameterizations to metric structures in the completion $\MC(\scrH_\G^{++})$ that do not belong to $\scrH_\G^{++}$. This approach is described in \Cref{subsec:noncomplete}.

As another application of the dictionary, we construct isometric actions of hyperbolic surface groups with unusual properties. We say that an isometric action of a hyperbolic group $\Gamma$ on a Gromov hyperbolic (symmetric) metric space $X$ has \emph{bounded backtracking} if for a fixed word metric on $\G$, orbits of geodesics in $\G$ remain quasigeodesics (up to reparameterization) in $X$ with controlled quasi-isometry constants. This notion was first introduced for studying actions on real trees \cite{GJLL}. For surface groups, we can construct the following actions.

\begin{corollarymain}\label{maincor:positiveonnonsimpleintro}
    Let $\G$ be a hyperbolic surface group. Then there exists a Gromov hyperbolic symmetric metric space $X$ and an isometric and cobounded action of $\G$ on $X$ such that:
    \begin{enumerate}
        \item Any element representing a simple closed geodesic is elliptic.
        \item Any element representing a non-simple closed geodesic is loxodromic.
        \item The action has bounded backtracking.
\end{enumerate}
\end{corollarymain}
The action in the proposition above is \emph{not proper}: there exist balls in $X$ intersecting its translates for infinitely many elements of $\G$. We contrast this example with those of Deroin--Tholozan~\cite{deroin-tholozan}, who construct representations of fundamental groups of punctured spheres into $\PSL(2,\R)$ for which no simple element acts loxodromically (yet, which fail to satisfy the bounded backtracking property for the action on $\Hy^2$). The proof of this proposition also makes use of the dictionary: if $\Sig$ is a hyperbolic surface with fundamental group $\G$ and $F:\msf{T}^1 \Sig \ra \R$ is H\"older continuous and nonnegative, we associate a hyperbolic metric potential $\p$ in the \emph{extension} $\calH^{+}_\G$ of $\calH_\G^{++}$, see
\Cref{sec:HMP}. After appropriately choosing $F$, the hyperbolic metric potential $\p$ obtained via~\Cref{thm:dictionary} corresponds to the desired isometric action of $\G$ on $X$.

The proof of this proposition is given in \Cref{subsec:exotic}. There, we also show that $X$ is neither a real tree nor induced by a geodesic current on $\Gamma$ (see \Cref{prop:positiveonnonsimple}). We also exhibit similar exotic isometric actions for free groups, see \Cref{prop:positiveonnonsimple_free}.

\subsection{Related work}\label{subsec:introotherwork}
\Cref{mainthm:dictionary_manifold} is inspired by \emph{Ledrappier's correspondence} \cite{ledrappier}, which for a closed negatively curved manifold $M$ with $\G=\pi_1(M)$ establishes correspondences between
\begin{enumerate}
    \item conjugacy classes of H\"older reparameterizations in $\Par^h(\msf{T}^1M,\vg)$; 
    \item cohomology classes of positive real-valued H\"older continuous potentials $F:\msf{T}^1M\ra \R$; and,
    \item cohomology classes of H\"older cocycles $\msf{c}:\G \times \partial \G\ra \R$ with positive periods and finite entropy.
\end{enumerate}

See \Cref{subsec:metricanosovledrappier} for the detailed definitions. Note that the cocycles in item (3) above are only defined in terms of $\G$, and therefore can be obtained from geometric information that is not necessarily related to $M$. In the setting of Anosov representations, this correspondence was first used by Sambarino
\cite{sambarino.quantitative} to prove counting results for strictly convex (for instance, Hitchin) representations. See also \cite{sambarino.report} for a version of Ledrappier's correspondence for Mineyev's flow space $\UG$.

For compact locally $\CAT(-1)$ spaces, some distance-like functions on $\G$ have been produced from H\"older functions on the geodesic flow space, similar to \Cref{mainthm:dictionary_manifold}.
This is the case of the works \cite{connell-muchnik.gibbs,connell-muchnik.harmonicity} on harmonicity of Patterson--Sullivan measures, and \cite{gekhtman-tiozzo} on singularity of hitting measures (for random walks with superexponential moment) with respect to Gibbs measures.

In our setting, we obtain points in $\scrH_\G^{++}$ from positive continuous potentials $F:\UG \ra \R$ satisfying the \emph{Bowen condition}, see \Cref{def:bowen} and \Cref{prop.dictionaryBowen}. This condition is strictly weaker than being H\"older, and it also appears in the recent work of Carrasco and Rodr\'iguez-Hertz \cite{carasco-hertz}. For compact, locally $\CAT(-1)$ spaces, they show a correspondence between \emph{measurable} Bowen potentials on the flow space $\UG$ and \emph{quasimorphisms} of $\G$, in a similar spirit to \Cref{mainthm:dictionary_manifold}. 
Our work is disjoint from theirs in two aspects. First, we are able to construct \emph{continuous} potentials from geometric data on $\G$, which is not immediately accessible from their methods. Second, for a hyperbolic group $\G$, its space of quasimorphisms is disjoint from $\calH_\G^{++}$ (up to bounded functions) and is ``orthogonal'' to $\calD_\G$.

Beyond the continuous setting, there are also recent developments in the Patterson--Sullivan theory for \emph{coarse} cocycles associated to convergence group actions by Blayac--Canary--Zhu--Zimmer~\cite{BCZZ24:GPS1,BCZZ24:GPS2}. See also Dilsavor's thesis \cite{dilsavor} for related notions and Kim--Zimmer's work on PS systems \cite{KZ25:Rigidity}.  We wonder if the continuous cocycles obtained from our work (or a modification thereof) are instances of continuous GPS-systems in the sense of \cite{BCZZ24:GPS2}.


\subsection{Green metrics}\label{subsec:introGreen}
To obtain continuous reparameterizations in \Cref{mainthm:dictionary_manifold}, we make use of \emph{Green metrics} on hyperbolic groups. These metrics were introduced by Blach\`ere and Brofferio in \cite{blachere-brofferio} and encode probabilistic information about random walks on the group. 

A probability measure $\mu$ on the (non-necessarily hyperbolic group) $\G$ is \emph{admissible} if its support generates $\G$ as a semi-group. If the random walk $(Z_k)_k$ on $\G$ with transition probabilities given by $\mu$ is transient, the Green metric $d_{\mu}$ on $\G$ is given by
\[d_\mu(g, h) := -\log \bbP(\exists k \text{ s.t. } gZ_k = h).\]
Note that $d_\mu$ is not necessarily symmetric. For $\G$ a non-elementary hyperbolic group, it is known that $d_\mu$ belongs to $\calH_\G^{++}$ as long as $\mu$ has superexponential moment and $d_\mu\in \calD_\G$ if in addition $\mu$ is symmetric \cite{BHM.harmonic,gouezel,GGPY}. These results are based on the work of Ancona~\cite{Anc87:InequalityAncona}.

For a non-elementary hyperbolic group, we show that Green metrics induce a dense subspace of $\scrD_\G$. More precisely, by considering suitably chosen random walks supported on large spheres, we construct sequences of Green metrics whose induced metric structures converge to an arbitrary point in $\scrD_\G$.

\begin{mainthm}[Density of Green metrics]\label{mainthm.greendense}
    Let $\G$ be a non-elementary hyperbolic group and $d\in \calD_\G$. Then there exists a constant $\del_1>0$ satisfying the following. Let $\del>\del_1$ and $\mu_l$ be the uniform probability measure supported on $S_l=\{g\in \G \colon l-\del<d(o,g)\leq l\}$. Then $S_l$ generates $\G$ for $l$ large enough and $[d_{\mu_l}]$ converges to $[d]$ in $\scrD_\G$.
\end{mainthm}




When $\G$ is the fundamental group of a closed hyperbolic surface, Green metrics are also related to the \emph{singularity conjecture} \cite{kaimanovich-leprince}. This conjecture predicts that if $\mu$ is an admissible finitely supported probability measure on $\G$ and $\rho:\G \ra \PSL(2,\R)$ is a discrete and faithful representation, then the hitting measure of $\mu$ on $\mbb{S}^1$ is mutually singular to the Lebesgue measure. As proven in \cite{BHM.harmonic}, this is equivalent to $[d_\mu]\neq \vx_{\rho}\in \scrD_\G$, where $\vx_\rho$ is the point induced by the action of $\G$ on $\Hy^2$ given by $\rho$. In a sense, \Cref{mainthm.greendense} reveals that the singularity conjecture is as hard as it can be from the point of view of the length functions: for any point $\vx$ in the Teichm\"uller space of $\G$, any neighborhood of $\vx$ in $\scrD_\G$ contains a point induced by a Green metric for a finitely supported measure.

The geometry of Green metrics and \Cref{mainthm.greendense} are crucial in the proof of \Cref{mainthm:dictionary_manifold}~(2). The main property is that Green metrics have H\"older continuous Busemann cocycles, which follows from the H\"older continuity of the Martin kernel \cite[Theorem~3.3]{INO} and the identification of the Martin boundary and the Gromov boundary \cite[Theorem~1.7]{BHM.harmonic}. This H\"older regularity also follows from the fact that Green metrics are strongly hyperbolic \cite{nica-spakula}. For these Busemann cocycles, Ledrappier's correspondence applies, and then a limiting argument using \Cref{mainthm.greendense} upgrades this correspondence to continuous potentials on $\UG$ that are dual to metrics in $\calD_\G$, see \Cref{prop:D_G->C}.



\subsection*{Organization}

In \Cref{sec:preliminaries} we review background on Gromov hyperbolic spaces, pseudometrics, hyperbolic metric potentials, 
as well as quasi-conformal measures and geodesic currents. In \Cref{sec:metriccompletion} we describe points in the metric completion $\MC(\scrH_\G^{++})$ as equivalence classes of left-invariant pseudometrics on $\G$. In \Cref{sec:mineyevflow} we recall the construction of Mineyev’s flow space and develop the dynamical properties needed later. In \Cref{sec:dictionary} we assume \Cref{mainthm.greendense}
(whose proof is deferred to the end) and prove \Cref{thm:dictionary}, which is our dictionary result that specializes to \Cref{mainthm:dictionary_manifold} in the manifold case. In \Cref{sec:meandistortion_length} we prove \Cref{thm.tauvsell}, a technical result relating mean distortion to ratios of length functions evaluated on the appropriately normalized BMS current. In \Cref{sec:examples} we construct examples via the dictionary: we deduce \Cref{thmmain.H_Gnotcomplete} about non-completeness of $(\scrH_\G^{++},\Del)$, and exhibit metrics on surface and free groups that vanish precisely on simple elements, which imply  \Cref{maincor:positiveonnonsimpleintro}.
Finally, in \Cref{sec:proofgreen} we prove density of Green metrics in $\scrD_\G$, establishing \Cref{mainthm.greendense}.

\subsection*{Acknowledgments}
The authors thank Caleb Dilsavor, Ursula Hamenst\"adt, Giuseppe Martone, Eduardo Silva and Ralf Spatzier for helpful discussions and comments related to this work. Support for D. Mart\'inez-Granado was provided by the Luxembourg National Research Fund (AFR / Bilateral--ReSurface 22/17145118) joint with the National University of Singapore, the Marie Sklodowska--Curie Action CurrGeo grant (101154865), and by the Long Term Visitor fund of the Thematic Program on Randomness and Geometry at the Fields Institute 2024. E. Reyes is supported by ANID Fondecyt Iniciaci\'on grant 11260637. He also thanks the hospitality of the University of Chicago, and the hospitality and support from the Max Planck Institut für Mathematik, the Fields Institute during the Thematic Program on Randomness and Geometry in 2024, and the Simons Laufer Mathematical Sciences Institute in Berkeley, California, during the Spring 2026 semester (National Science Foundation Grant No. DMS-2424139).



\section{Preliminaries}
\label{sec:preliminaries}
\subsection{Pseudometric spaces and quasi-isometries}\label{subsec:quasiisometries} 
Let $(X,d)$ be a \emph{pseudometric} space; that is, $d \colon X \times X \to \mathbb{R}$ is a non-negative function satisfying $d(x,x)=0$ and the triangle inequality. This is weaker than being a metric: in particular, we may have $d(x,y)=0$ for distinct points $x,y \in X$, and $d$ need not be symmetric.
The \emph{double difference} in $X$ is the function $( \cdot, \cdot \mid \cdot, \cdot ) \colon X^4 \to \mathbb{R}$ defined by
\[
( a, a' \mid b, b' ) \coloneq \frac{1}{2} \left( d(a,b) - d(a',b) - d(a,b')+d(a',b') \right).
\]
This extends the notion introduced
by Paulin in~\cite{paulin96} when $d$ is a symmetric distance. The function
\[
( a \mid b )_c \coloneqq ( a,c \mid c, b) =
\frac{1}{2} (d(a,c) + d(c,b) - d(a,c))
\]
is the \emph{Gromov product} with respect to the pseudometric $d$, and we have the relation
\[
( a, b \mid x, y) = ( b \mid x )_a - ( b \mid y )_a.
\]
If we want to emphasize the dependence of $d$, we will write $\gpr{a,b}{a',b'}=\gpr{a,b}{a',b'}^d$ and $( a \mid b )_c= ( a\mid b )^d_c$.

Given an interval $I \subset \mathbb{R}$ and constants $L,C>0$, we
say that a map $\gamma \colon I \to X$ is an \emph{$(L,C,d)$-quasigeodesic}
if
\begin{equation*}
L^{-1}|s-t|-C \leq d(\gamma(s),\gamma(t)) \leq L |s-t| + C
\quad \text{for all } s,t \in I,
\label{eq:qi}
\end{equation*}
and a \emph{$(C,d)$-rough geodesic} if
\[
|s-t| - C \leq d(\gamma(s), \gamma(t)) \leq |s-t| + C
\quad \text{for all } s,t \in I .
\]
A $(0,d)$-rough geodesic is referred to as a \emph{$d$-geodesic}. Unless necessary, we will drop the dependence on $d$ and refer, respectively, to $(L,C)$-quasigeodesics,
$C$-rough geodesics, or geodesics.
The pseudometric space $(X,d)$ is called \emph{$C$-roughly geodesic} if for
each $x, y \in X$, there is a $C$-rough geodesic connecting $x$ and $y$.
We simply say $(X,d)$ is \emph{roughly geodesic} if it is $C$-roughly geodesic
for some $C\geq 0$.

Given two pseudometric spaces $(X,d)$ and $(Y,\hat d)$,
a map $F \colon X \to Y$ is a \emph{quasi-isometric embedding} if there exist
constants $\lambda_1,\lambda_2,C>0$ such that
\[
\lambda_1 d(x,y) - C \leq \hat d(F(x),F(y)) \leq \lambda_2 d(x,y) + C
\quad \text{for all } x,y \in X .
\]
If, moreover, there is a constant $C'>0$ such that every $y \in Y$ is
within $C'$ of some element of $F(X)$, we say that $F$ is a \emph{quasi-isometry}.
A quasi-isometry $F$ satisfying $\lambda_1=\lambda_2$ is called a
\emph{rough similarity}, and when $\lambda_1=\lambda_2=1$ it is called
a \emph{rough isometry}. If $d_1,d_2$ are two pseudometrics on the same set $X$, we say that $d_1,d_2$ are \emph{quasi-isometric/roughly similar/roughly isometric} if the identity map $\Id_X:(X,d_1)\ra (X,d_2)$ is a quasi-isometry/rough similarity/rough isometry.


\subsection{Gromov hyperbolic spaces}\label{subsec:hyperbolic}

A pseudometric space $(X,d)$ is $\delta$\emph{-hyperbolic} $(\del\geq 0)$ if
\begin{equation}\label{eq:gromovhyp}
\gpr{x}{z}_{w} \geq \min \{ \gpr{x}{y}_w, \gpr{y}{z}_w \} - \delta
 \quad \text{ for all }x,y,z,w \in X,\end{equation}
and $(X,d)$ is \emph{hyperbolic} if it is $\delta$-hyperbolic for some $\delta \geq 0$. Equivalently, we say that $d$ is $(\del$-)hyperbolic. A pseudometric space $(X,d)$ is \emph{strongly hyperbolic} with parameter $\ep>0$ if 
\begin{equation}\label{eq:defstronghyp}
    e^{-\ep\gpr{x}{z}_w}\leq e^{-\ep\gpr{x}{y}_w}+e^{-\ep\gpr{y}{z}_w} \quad \text{ for all }x,y,z,w\in X.
\end{equation}
Any strongly hyperbolic space with parameter $\ep$ is $\del$-hyperbolic with $\del=(\log 2)/\ep$ \cite[Theorem~4.2]{nica-spakula}.

Hyperbolic pseudometric spaces have a natural completion at infinity: the Gromov boundary. A sequence $\{x_k\}_{k\geq 0}$ of points in the hyperbolic space $X$ \emph{diverges}
if $\gpr{x_k}{x_m}^d_o$ tends to infinity as $\min\{k, m\}$ tends to infinity, for $o$ a fixed base point in $X$. Two divergent sequences $\{x_k\}_{k\geq 0}$ and $\{y_k\}_{k\geq 0}$ are said to be \emph{equivalent} if $\gpr{x_k}{y_m}^d_o$ tends to infinity as $\min\{k, m\}$ tends to infinity. The notions of divergence and equivalence are independent of the base point $o\in X$. The \emph{Gromov boundary} $\partial X$ of $X$ is the space of equivalence classes of divergent sequences in $X$. If $a \in \partial X$ is represented by the divergent sequence $\{x_k\}_k$ we say that $x_k$ \emph{converges} to $a$. We denote $\ov X \coloneqq X \cup \partial X$.

Using hyperbolicity, we form a quasi-extension of the Gromov product $\gpr{\cdot}{\cdot}^d_o:\ov X^2 \ra \R$
in such a way that for some constant $C\ge 0$ we have that
\begin{equation*}
\limsup_{k \to \infty}\gpr{x_k}{y_k}^d_o-C \le \gpr{a}{b}^d_o \le \liminf_{k \to \infty}\gpr{x_k'}{y_k'}^d_o+C
\end{equation*}
for all $a, b \in \ov{X}$ and for all $\{x_k\}_{k \geq 0}, \{x_k'\}_{k\geq 0} \in a$ and $\{y_k\}_{k\geq 0}, \{y_k'\}_{k\geq 0}^\infty \in b$ (if $a\in X$ we interpret $\{x_k\}_{k \geq 0}\in a$ as as sequence $\{x_k\}_k$ converging to $a$).


The \emph{Busemann function} $\b^d_o:X \times \partial X$ is defined according to
\begin{equation}\label{eq.defbusemann}
  \b^d_o(x, a):=\sup\big\{\limsup_{k \to \infty}(d(x, x_k)-d(o, x_k)) \ : \ \{x_k\}_{k\geq 0} \in a\big\}.
\end{equation}

When $(X,d)$ is strongly hyperbolic, the Busemann function is obtained as a genuine limit (i.e. $\b^d_o(x, a):=\lim_{k \to \infty}(d(x, x_k)-d(o, x_k))$). Similarly, the Gromov product extends uniquely to $\ov X$, and is continuous for the corresponding topology on $\ov{X}$. We have the following criterion for strong hyperbolicity.

\begin{lemma}[{\cite[Lemma~6.2]{nica-spakula}}]\label{lem:nicaspacula}
A hyperbolic pseudometric space $(X,d)$ is strongly hyperbolic with parameter $\ep$ if and only if there exist constants $C,R_0 > 0$ such that if $x,y,z,w\in \ov{X}$ satisfy
\[\gpr{x,w}{y,z}\geq R\geq R_0,\]
then
\[|\gpr{x,z}{y,w}|\leq Ce^{-\ep R/2}.\]  
\end{lemma}


\subsection{Hyperbolic groups and the space of symmetric pseudometrics}\label{subsec:hypgroups}

Recall that a non-elementary hyperbolic group is a finitely generated group $\G$ that is not virtually cyclic and such that any word metric $d_S$ on $\G$ associated to a finite, symmetric generating set $S$ is hyperbolic. Throughout this work we use $o$ to denote the identity element of $\G$ and $[\G]$ (resp. $[\G]'$) to denote the set of (resp. non-torsion) conjugacy classes of $\G$.

Given a hyperbolic group $\G$ we write $\Dc_\G$ for the collection of pseudometrics $d$ on $\G$ that are
\begin{itemize}
    \item symmetric;
    \item hyperbolic;
    \item quasi-isometric to a (equivalent any) word metric $d_S$ on $\G$; and,
    \item are (left) $\G$-invariant: $d(gx,gy) = d(x,y)$ for all $g, x, y \in \G$.
\end{itemize}

\noindent The pseudometrics in $\Dc_\G$ are roughly geodesic \cite[Theorem~2.2]{BHM.harmonic}.

We are interested in the space $\Dc_\G$ up to rough similarity (recall \Cref{subsec:quasiisometries}). Let $\scrD_\G$ denote the quotient of $\Dc_\G$ by the equivalence relation of rough similarity \cite{furman}. Given $d \in \Dc_\G$ we write $[d] \in \scrD_\G$ for its corresponding equivalence class. There is a natural metric $\D$ on $\scrD_\G$ which is the \emph{symmetrized Thurston metric}:
for $d_1,d_2 \in \Dc_\G$ let 
\[
\Dil(d_1,d_2) = \sup_{[g] \in [\G]'} \frac{\ell_{d_1}[g]}{\ell_{d_2}[g]}.
\]
Then given $[d_1], [d_2] \in \scrD_\G$, define its distance according to
\begin{equation*}
\D([d_1],[d_2]) = \log (\Dil(d_1,d_2) \, \Dil(d_2,d_1)).
\end{equation*}
Here $\ell_d$ denotes the \emph{stable translation length} of $d$:
\[\ell_d[g]:=\lim_{k\to \infty}\frac{d(o,g^k)}{k} \quad \text{ for }[g]\in [\G].\]
The metric properties of $(\scrD_\G, \D)$ were studied in \cite{oregonreyes.metric} and \cite{cantrell-reyes.manhattan}. 

For a hyperbolic group, its Gromov boundary can be defined using any metric $d\in \calD_\G$, and the resulting space $\pG$ is independent of this choice by the Morse lemma. Moreover, $\pG$ is compact and metrizable.


\subsection{Hyperbolic metric potentials}\label{sec:HMP}
As discussed in the introduction, we are interested in a natural non-symmetric extension of $\Dc_\G$. Let $\G$ be a non-elementary hyperbolic group with identity element $o \in \G$ as above.

\begin{definition}\label{def:HMP}
Let $\calH_\G$ denote the space of all functions
$\psi \colon \G \times \G \to \mathbb{R}$ satisfying:
\begin{enumerate}
\item $\G$-invariance: $\psi(gx,gy)=\psi(x,y)$ for all $x,y,g \in \G$; and,
\item for any $d_0 \in \calD_\G$ there exists $\lambda>0$ so that for every
$x,y,w \in \G$ we have
\[
\bigl|\gpr{x}{y}_w^\psi\bigr|
\leq \lambda \gpr{x}{y}_w^{d_0} + \lambda ,
\]
where
\[
\gpr{x}{y}_w^\psi \coloneqq
\frac{1}{2}\bigl(\psi(x,w)+\psi(w,y) - \psi(x,y)\bigr)
\]
denotes the $\psi$-\emph{Gromov product} of $x,y$ based at $w$.
\end{enumerate}
The elements of $\calH_\G$ are called \emph{hyperbolic metric potentials}.
We also let $\calH^{++}_\G$ be the set of proper maps $\psi \in \calH_\G$
such that for some $C>0$ we have $\psi(x,y)>-C$ for all $x,y \in \G$, which we refer to as \emph{properly positive hyperbolic metric potentials}.
\end{definition}

Hyperbolic metric potentials were introduced in \cite{CRS}.
In that work, the following characterization of hyperbolic metric potentials was provided \cite[Proposition~3.1]{CRS}.

\begin{proposition}\label{prop:criterionHMP}
    A function $\p: \Gamma \times \Gamma \to \R$ is a hyperbolic metric potential if and only if it is $\Gamma$-invariant and the following holds. Given $d\in \calD_\G$ and all large enough $C, R \ge 0$, there exists $C_0 \ge 0$ such that for all $g,h\in \G$ and any $(C,d)$-rough geodesic $\gamma$ between $g$ and $h$, we have
\begin{equation*}
|\gpr{g}{h}_w^\p| \le C_0
\end{equation*}
for all $w$ in the $R$-neighborhood of $\gamma$ for the metric $d$.
\end{proposition}

Generalizing the situation for $\calD_\G$, given $\p\in \calH_\G$ we define its \emph{stable translation length function} $\ell_\p:[\G] \ra \R$ according to
\[\ell_{\p}[g]=\lim_{k\to \infty}{\frac{\p(o,g^k)}{k}}.\]
This function is well-defined and finite \cite[Lemma~3.4]{CRS}. 

Even though $\calD_\G$ and $\calH_\G^{++}$ play a prominent role as subspaces of $\calH_\G$, there are other subspaces of interest that will appear throughout this work. These are:
\begin{itemize}
    \item $\partial \calD_\G$: the subspace of $\calH_\G$ consisting of symmetric pseudo-metrics on $\G$ that are unbounded but do not belong to $\calD_\G$.  
    \item $\ov{\calD}_\G:=\calD_\G \cup \partial \calD_\G$.
     \item $\calH_\G^{+}$: the space of all $\p\in \calH_\G$ such that $\ell_\p\geq 0$.
\end{itemize}

Up to bounded error every $\p \in \Hc_\G^+$ is a pseudometric~\cite[Lemma~3.4]{CRS}.

\begin{lemma}\label{lem:potential_pseudometric} Every $\p \in \Hc_\G^+$ is equal to $d-b$ for $d$ a (non-necessarily symmetric) $\G$-invariant Gromov hyperbolic pseudometric, and $b$ a bounded function.
\end{lemma}

The following diagram summarizes the chain of containments among some of the spaces introduced above.

\begin{table}[h]
\centering
\renewcommand{\arraystretch}{1.3}
\setlength{\tabcolsep}{6pt}
\begin{tabular}{ccccccc}
$\Dc_\G$ & $\subset$ & $\Hc_\G^{++}$ &  $\subset$ & $\Hc_\G^+$ & $\subset$ & $\Hc_\G$ \\[6pt]
{\footnotesize \begin{tabular}{@{}c@{}}proper, symmetric \\pseudometrics\end{tabular}}
& 
& {\footnotesize \begin{tabular}{@{}c@{}}proper, \\
$\ell_\p > 0$\end{tabular}}
&  
& {\footnotesize $\ell_\p \geq 0$}
& 
& {\footnotesize \begin{tabular}{@{}c@{}}hyperbolic\\
metric potentials\end{tabular}}
\end{tabular}
\end{table}

Given $\p,\varphi \in \calH_\G^{++}$, and in analogy with metrics in $\Dc_\G$  we define 

\begin{equation}\label{eq:formulaDel}
   \Dil(\varphi,\p) = \sup_{[g] \in [\G]'} \frac{\ell_{\varphi}[g]}{\ell_\p[g]} \quad \text{ and } \quad \D(\varphi,\p) = \log \left(\Dil(\varphi,\p) \, \Dil(\p,\varphi)\right). 
\end{equation}
The next lemma follows immediately from \cite[Lemma~3.4]{CRS}.

\begin{lemma}\label{lem.comp}
    Given $\varphi,\p \in \calH_\G^{++}$ there exists $C >0$ such that
    \[
    \Dil(\p,\varphi)^{-1} \p(g,h) - C \le \varphi(g,h) \le \Dil(\varphi,\p)\p(g,h) + C \quad \text{ for all }g,h\in \G.
    \]
\end{lemma}

The \emph{exponential growth rate} of $\p\in \calH_\G^{++}$ is the limit \begin{equation}\label{eq:EGR}
    v_\p = \limsup_{T\to\infty} \frac{1}{T} \log \#\{g \in \G: \p(o,g) < T\},
\end{equation}
which is finite and positive \cite[Lemma~3.16]{CRS}.

Given $\varphi\in \calH_\G$ and $\p \in \calH^{++}_\G$ we define their \emph{mean distortion} $\tau(\varphi/\p)$ to be
\begin{equation}\label{eq:meandistortion}
\tau(\varphi/\p)=\lim_{T\to\infty}\frac{1}{\#\{[g]\in [\G]':\ell_{\p}[g]<T\}}\sum_{[g]\in [\G]': \ell_{\p}[g]<T}{\frac{\ell_{\varphi}[g]}{\ell_\p[g]}}.
\end{equation}
This limit was shown to exist in \cite{cantrell-tanaka.manhattan} for pairs of metrics in $\Dc_\G$, and the same proof can be applied to pairs of elements in $\calH_\G$; see \cite[Proposition~5.9]{CRS}. Moreover, if $\varphi\in \calH_\G^{++}$ and $v_\p=v_{\varphi}=1$, then $\tau(\varphi/\p)\geq 1$.

We let $\scrH_\G$ and $\scrH_\G^{++}$ be the set of rough similarity equivalence classes in $\calH_\G$ and $\calH_\G^{++}$ respectively. That is, $\p$ and $\varphi$ belong to the same class if and only if $\sup_{g\in \G}|\p(o,g)-\lam \varphi(o,g)|<\infty$ for some $\lam>0$. As for $\scrD_\G$, we let $[\p]$ denote the equivalence class of the hyperbolic metric potential $\p\in \calH_\G$. Similarly, the function $\Del$ from \Cref{eq:formulaDel} defines a metric on $\scrH_\G^{++}$. 

By \Cref{lem:potential_pseudometric}, every properly positive metric potential admits a pseudometric representative in its rough similarity class. Accordingly, we will sometimes work with pseudometric representatives directly.


\subsection{Quasi-conformal measures}\label{subsec:QC}

Given $\p\in \calH_\G$, its \emph{Busemann function} $\beta^\p_o:\G \times \partial \G\ra \R$ can be defined as in \Cref{eq.defbusemann}. When $\p$ belongs to $\calH^{++}_\G$, a Borel probability measure
$\nu$ on $\partial \G$ is \emph{quasi-conformal} for $\p$ if there exists
a constant $C>1$ such that for every $g \in \G$ and for $\nu$-almost every
$a \in \partial \G$ we have
\begin{equation}\label{eq:defquasiconformal}
C^{-1} e^{-v_\p \beta_o^\p(g,a)}
\leq \frac{dg_*\nu}{d\nu}(a)
\leq C e^{-v_\p \beta_o^\p(g,a)}.
\end{equation}

 For $\p\in \calD_\G$, Coornaert \cite{coornaert} proved that quasi-conformal measures always exist, and any two such measures are absolutely continuous with uniformly bounded Radon--Nikodym derivatives. The same result holds for arbitrary $\p\in \calH_\G^{++}$.

\begin{proposition}[{\cite[Proposition 3.19]{CRS}}]
    Fix a reference metric $d \in \Dc_\G$. Then any $\p\in \calH_\G^{++}$ 
    admits a quasi-conformal measure $\nu$ and any two such quasi-conformal measures are absolutely continuous with respect to each other.  
\end{proposition}

If $\p\in \calH_\G^{++}$ and $\nu$ is a quasi-conformal measure for $\p$, then 
\begin{equation}\label{eq:defp_nu}
    \p_\nu(g,h):=\log \left\| \frac{dg_*\nu}{d\nu}\right\|
\end{equation}
defines a pseudometric roughly isometric to $v_\p \p$. This follows from \Cref{eq:defquasiconformal} and the fact that for any $g\in \G$ there exists a point $a\in \pG$ such that $g$ belongs to a uniform quasigeodesic ray from $o$ to $a$ (see e.g.~\cite[Equation~2.5]{CRS}), hence \Cref{prop:criterionHMP} implies that $-\p(o,g)$ is comparable to $\beta^\p_o(g,a)$.

Quasi-conformal measures can also be used to recover the mean distortion $\tau(\varphi/\p)$ for $\varphi, \p \in \Hc^{++}_\G$.
\begin{lemma}[{\cite[Corollary~5.8]{CRS}}]\label{lem:limittauQC}
    Let $\nu$ be a quasi-conformal measure associated to $\p\in \calH_\G^{++}$. Then for any $\varphi\in \calH_\G$ and $\nu$-almost every $a \in \partial \G$, we have that
    \[
\frac{\varphi(o,g_k)}{\p(o,g_k)} \to \tau(\varphi/\p)
   \quad \text{ as }k\to \infty \]
   for every quasigeodesic ray  $(g_k)_{k\geq 0}$ converging to $a$.
\end{lemma}


\subsection{Geodesic currents}\label{subsec:currents}

The \emph{double boundary} is the set
$\partial^2\G$ of \emph{ordered} pairs of distinct points of $\partial \G$.
We equip $\partial^2\G$ with the topology induced by the inclusion
$\partial^2\G \subset \partial \G\times \pG$. The diagonal action
of $\G$ on $\partial^2\G$ is continuous and cocompact.

A \emph{geodesic current} on $\G$ is a $\G$-invariant Radon
measure on $\partial^2\G$.
We let $\calC_\G$ denote the space of all geodesic currents on $\G$
equipped with the weak$^*$-topology. We do not assume geodesic currents to be \emph{symmetric} (i.e. invariant under the flip $\iota(a,b)=(b,a)$), as our metric potentials naturally yield non-symmetric currents. 

Geodesic currents were introduced by Bonahon \cite{bonahon.annals,bonahon.currentshpygroups} as a completion of the space of weighted conjugacy classes in $\G$. Indeed, if $g$ is a primitive non-torsion element, i.e.\ if $g=h^m$ for some $h\in \G$ then $m=\pm1$, let $g^-,g^+$ be its repelling and attracting fixed points
in $\partial \G$, respectively. Let $\eta_{[g]}$ be the sum of Dirac measures
supported on the $\G$-translates of $(g^-,g^+)$. If $g=h^m$ for $h$ primitive and $m\geq 1$, we let $\eta_{[g]} \coloneqq m \cdot \eta_{[h]}$.
Since $\G \cdot (g^-,g^+)$ is countable and discrete, $\eta_{[g]}$ is a locally finite Borel measure on $\partial^2\G$, hence Radon,
and it is clearly $\G$-invariant.
We call $\eta_{[g]}$ the \emph{rational current} associated to $g$, which actually only depends on the conjugacy class $[g]\in [\G]'$. Bonahon also showed that positive real multiples of rational currents
are dense in $\calC_\G$ \cite[Theorem~7]{bonahon.currentshpygroups}. 

Let
\[
\mathbb{P}\calC_\G \coloneqq (\calC_\G \setminus \{0\})/\mathbb{R}^+
\]
denote the space of \emph{projective geodesic currents} equipped
with the quotient topology. The space $\mathbb{P}\calC_\G$
is compact and metrizable \cite{bonahon.annals}.  We let $[\Lam]$ denote the projective class of the non-zero current $\Lam$.

In \cite{furman}, Furman noted that pseudometrics in $\calD_\G$ induce geodesic currents in a natural way, a fact further extended in \cite[Section~3.3]{CRS} to properly positive hyperbolic metric potentials. 
Given $\psi \in \Hc^{++}_\G$, let $\hat{\psi} \in \Hc^{++}_\G$ be defined as $\hat{\psi}(g,h) \coloneqq \psi(h,g)$. We have that $\hat\p$ also belongs to $\calH_\G^{++}$ and $v_\psi=v_{\hat{\psi}}$. Consider quasi-conformal measures $\nu, \hat{\nu}$ for $\psi$ and $\hat{\psi}$, respectively. By \cite[Equation~3.9]{CRS} there exists a geodesic current $\Lambda_\psi \in \calC_\G$ satisfying
\[
d\Lambda_\psi(a,b) = \alpha(a,b)\exp(2v_\psi\gpr{a}{b}^\psi_o)d\hat{\nu}(a)d\nu(b),
\]
where $\alpha:\ppG\ra \R$ is a measurable function essentially bounded and bounded away from zero. The geodesic current obtained above is ergodic. Also, since $|\gpr{\cdot}{\cdot}_o^\p-\gpr{\cdot}{\cdot}_o^{\hat\p}|$ is a bounded function \cite[Lemma~3.10]{CRS}, we have that $[\Lam_{\hat{\psi}}]=[\iota_*\Lambda_\psi]$. In general, if $[\psi] \neq [\hat{\psi}]$ then $[\Lambda_\psi] \neq [\Lambda_{\hat{\psi}}]$, so the resulting geodesic current is not necessarily symmetric. We call $\Lam_\p$ a \emph{Bowen--Margulis--Sullivan (BMS)} geodesic current for $\p$ or $[\p]$.
This induces a map \[\BMS:\scrH^{++}_\G \ra \bbP\calC_\G\] that sends $[\psi]$ to the projective class $[\Lam_\psi]$ of $\Lam_\psi$. We also call $\BMS([\p])$ the \emph{BMS (projective) current} of $\p$ or $[\p]$. By abuse of notation, we also denote $\BMS([\p])$ by $\BMS(\p)$.



\section{The completion $\MC(\scrH_\G^{++})$}\label{sec:metriccompletion}

In this section we describe points in the metric completion $\MC(\scrH_\G^{++})$ of $\scrH_\G^{++}$ as equivalence classes of certain left-invariant pseudometrics on $\G$. As we will see in \Cref{subsec:noncomplete}, this completion is strictly bigger than $\scrH_\G^{++}$. The following is the main result of the section.

\begin{proposition}\label{prop:MC(H)}
For any point $\vx\in \MC(\scrH_\G^{++})$ there exists a $\G$-invariant pseudometric $\p$ on $\G$ with exponential growth rate 1 (in the sense of \Cref{eq:EGR}) and satisfying the following.
\begin{enumerate}
    \item For any sequence $(\vx_m)_m$ in $\scrH_\G^{++}$ converging to $\vx$ and representatives $\p_m$ of $\vx_m$ with exponential growth rate 1, we have that 
    $$\log \ell_{\p_m} \to\log \ell_{\p} \quad  \text{ uniformly on }[\G]'.$$
    \item $\p$ is quasi-isometric to any metric in $\calD_\G$. 
\end{enumerate}  
\end{proposition}

\begin{definition}\label{def:MCHandD}
Let $\MC(\calH_\G^{++})$ be the space of left-invariant functions $\p':\G\times \G \ra \R$ such that $|\p'-\lam \p|$ is sublinear in $\G$ for $\p$ a pseudometric as in the above proposition and some $\lam\in \R^+$. That is, for some fixed $d\in \calD_\G$, we have that $|\p'(o,g)-\lam \p(o,g)|/d(o,g)$ tends to zero as $d(o,g) \to \infty$. We let $\MC(\calD_\G)\subset \MC(\calH_\G^{++})$ be the subset of symmetric pseudometrics. 
\end{definition}

\begin{remark}\label{rmk:badmetrics}
Not every left-invariant pseudometric on $\G$ quasi-isometric to a metric in $\calD_\G$ belongs to $\MC(\calH_\G^{++})$. See for instance \Cref{rmk:extensioninMC}.
\end{remark}

Before proving \Cref{prop:MC(H)}, we note as a consequence that the projective class of the stable length $\ell_{\p}$ completely determines the point $\vx=[\p] \in \MC(\scrH_\Gamma^{+})$. In addition, the definition of dilation can be extended to pairs in $\MC(\calH_\G^{++})$, and hence we can use translation lengths to compute the distance $\Del$ in $\MC(\scrH_\G^{++})$ as in \Cref{eq:formulaDel}.

\begin{corollary}\label{cor:DeltaformulaMC}
For any $\vx=[\p]$ and $\vx'=[\p']$ in $\MC(\scrH_\G^{++})$ we have 
\[\Del(\vx,\vx')=\log(\Dil(\p,\p')\cdot \Dil(\p',\p)). \qedhere\]
\end{corollary}

To prove \Cref{prop:MC(H)}, we recursively modify elements in a sequence in $\calH_\G^{++}$ to extract convergent subsequences. In order to do this, we require the next lemma and its subsequent corollary. Given a quasi-conformal measure $\nu$ on $\partial \G$, we recall the pseudometric $\p_\n\in \calH_\G^{++}$ defined in \Cref{eq:defp_nu}.

\begin{lemma}\label{lem:improvemeasure}
Let $\p_0,\p\in \calH_\G^{++}$ be such that $\p_0$ is a pseudometric and $v_\p=1$. 
    Suppose that $\p\leq \p_0+C$ for some $C\geq 0$. Then there exists a probability measure $\nu$ on $\pG$ that is quasi-conformal to $\p$, and such that  
$$\p_\nu(o,g)=\log\left\|\frac{dg_\ast\nu}{d\nu}\right\|\leq \p_0(o,g) \quad \text{ for all }g\in \G.$$
\end{lemma}

\begin{proof}
Let $\wtilde\nu$ be an arbitrary  quasi-conformal measure for $\p$. Since $\p$ has exponential growth rate 1, $\p_\nu$ is roughly isometric to $\p$ and we can find $D>1$ such that 
\begin{equation*}
    \left\|\frac{dg_\ast\wtilde\nu}{d\wtilde\nu}\right\|\leq De^{\p_0(o,g)} \quad \text{ for any }g\in \G.
\end{equation*}

Given $a\in \pG$, we define \begin{equation*}\label{eq.defdensity}
    F(a):=\inf_{g\in \G}{\frac{dg^{-1}_\ast\wtilde\nu}{d\wtilde\nu}(a)e^{\p_0(o,g)}}.
\end{equation*} This function is well-defined $\wtilde\nu$-almost everywhere, and for $\wtilde\nu$-almost every $a\in \pG$ we have $D^{-1}\leq F(a)\leq 1$. Hence, the probability measure $\nu$ with density $d\nu:=\frac{1}{\int{F}{d\wtilde\nu}}Fd\wtilde\nu$ is also quasi-conformal for $\p$.

We claim that for a given $g\in \G$ and for $\nu$-almost every $a$ we have $$\frac{dg_\ast\nu}{d\nu}(a)= \frac{F(g^{-1}a)}{F(a)}\frac{dg_\ast\wtilde\nu}{d\wtilde\nu}(a)\leq e^{\p_0(o,g)}.$$
Indeed, for any $h\in \G$ we have
\begin{align*}
    \frac{F(g^{-1}a)}{\frac{dh^{-1}_\ast\wtilde\nu}{d\wtilde\nu}(a)e^{\p_0(o,h)}}\leq \frac{\frac{d(hg)^{-1}_\ast\wtilde\nu}{d\wtilde\nu}(g^{-1}a)e^{\p_0(o,hg)}}{\frac{dh^{-1}_\ast\wtilde\nu}{d\wtilde\nu}(a)e^{\p_0(o,h)}}=\frac{e^{\p_0(o,hg)-\p_0(o,h)}}{\frac{dg_\ast\wtilde\nu}{d\wtilde\nu}(a)} \leq \frac{e^{\p_0(o,g)}}{\frac{dg_\ast\wtilde\nu}{d\wtilde\nu}(a)},
\end{align*}
where in the last inequality we use that $\p_0$ satisfies the triangle inequality.
Since the inequality above holds for every $g\in \G$ and $\wtilde\nu$-almost every $a\in \partial \G$, the lemma follows.
\end{proof}

\begin{corollary}\label{coro:improvemetric}
    Let $\p_0,\p\in \calH_\G^{++}$ with $\p_0$ a pseudometric. If $\p\leq \p_0+C$ for some constant $C>0$, then there exists a pseudometric $\p'$ that is roughly isometric to $\p$ and satisfies $\p'\leq \p_0$.
\end{corollary}

\begin{proof}
    After rescaling we can assume that $\p$ has exponential growth rate 1. Then we can take $\p'=\p_\nu$ for $\nu$ the quasi-conformal measure given by \Cref{lem:improvemeasure}.
\end{proof}

\begin{proof}[Proof of \Cref{prop:MC(H)}]
    Let $(\vx_m)_m$ be any sequence in $\scrH_\G^{++}$ converging to $\vx$, and suppose $\vx_m=[\p_m]$ with $\p_m\in \calH_\G^{++}$ with exponential growth rate 1. 
    
    First, we find a subsequence $(\vx_{r_k})_k$ such that for each $k$ we have $\Del(\vx_{r_{k+1}},\vx_{r_k})\leq 2^{-k}$ for all $k\geq 1$. Then inductively, we construct a sequence $\hat\p_{1},\hat\p_2,\dots$ such that:
    \begin{itemize}
        \item each $\hat\p_k$ is a pseudometric roughly isometric to $\p_{r_k}$; and
        \item $\Dil(\p_{r_k},\p_{r_{k+1}})^{-1}\hat\p_{k}\leq \hat\p_{k+1}\leq  \Dil(\p_{r_{k+1}},\p_{r_k})\hat\p_k$ for all $k$.
    \end{itemize}
    We set $\hat\p_{1}$ to be any pseudometric roughly isometric to $\p_{r_1}$, which exists by \Cref{lem:potential_pseudometric}. Assuming we have found $\hat\p_{1},\hat\p_{2},\dots \hat\p_{k}$, we use \Cref{coro:improvemetric} to find a pseudometric $\hat \p'_{k+1}$ that is roughly isometric to $\p_{r_{k+1}}$ and such that $\p'_{k+1}\leq \Dil(\p_{r_{k+1}},\p_{r_k})\hat\p_{k}$. We define $\hat\p_{k+1}:=\max\{\p'_{k+1},\Dil(\p_{r_{k}},\p_{r_{k+1}})^{-1}\p'_{k}\}.$

    By construction and a simple induction, our choice of subsequence $(\vx_{r_k})_k$ implies that $\hat\p_{k+1}\leq 2\hat\p_1$ for each $k$. In particular, up to taking a further subsequence and reindexing, we can assume that $(\hat\p_k)_k$ pointwise converges to a pseudometric $\p$, which we claim satisfies all the required properties.

    First, from the construction of $(\hat\p_k)_k$ we have
    \begin{equation}\label{eq:QIhatpsi}
        2^{-1/2^{m-1}}\hat\p_{m}\leq \hat\p_{m+k+1}\leq 2^{1/2^{m-1}}\hat\p_m \quad \text{ for all }m,k,
    \end{equation}
 and the same inequalities hold if we replace $\hat\p_{m+k+1}$ by $\p$. In particular, we have 
    \[2^{-1/2^{m-1}}\ell_{\p_{r_m}}\leq \ell_\p\leq 2^{1/2^{m-1}}\ell_{\p_{r_m}}\]
    for all $m$, implying that $\log\ell_{\p_m}$  uniformly converges to $\log\ell_\p$. It is not hard to see that this convergence holds for any sequence in $\scrH_\G^{++}$ converging to $\vx$, thus we have proved item (1). Item (2) then follows from \Cref{eq:QIhatpsi}. 
\end{proof}

We can also use \Cref{prop:MC(H)} to define a ``completion'' of $\calH_\G$. In virtue of \cite[Proposition~3.8]{CRS}, $\calH_\G$ is spanned as a real vector space by $\calH_\G^{++}$. Hence we make the following definitions.

\begin{definition}\label{def:MCH}
    We define $\MC(\calH_\G)$ as the set of real linear combinations of elements in $\MC(\calH_\G^{++})$. If $\p=\p^+-\p^-$ for $\p^+,\p^-\in \MC(\calH_\G^{++})$, we define its stable translation length function as $\ell_\p:=\ell_{\p^+}-\ell_{\p^-}$. We also let $\MC(\calH_\G^{+})$ be the subset of elements of $\p$ in $\MC(\calH_\G)$ such that $\ell_\p\geq 0$.
\end{definition}

It turns out that the translation function $\ell_\p$ is independent of the decomposition $\p=\p^+-\p^-$ as above. Also, in virtue of \Cref{prop:MC(H)}, we define $\MC(\scrH_\G)$ as the sets of equivalence classes in $\MC(\calH_\G)$, in which we identify $\p$ and $\p'$ if $\ell_\p=\lam\ell_{\p'}$ for some $\lam>0$. 

\section{Mineyev's flow space}\label{sec:mineyevflow}
For this section we fix a non-elementary, torsion-free hyperbolic group $\G$ and a strongly hyperbolic metric $\hat d\in \calD_\G$. Our goal of is to describe the construction of Mineyev's flow space \cite{mineyev.flow} and establish some technical properties that will be needed in the sequel. In Mineyev's original construction, the input metric $\hat{d}$ is defined in terms of a geometric action of $\G$ on a graph. However, any strongly hyperbolic metric in $\calD_\G$ suffices, see e.g. \cite[Section~3]{tanaka.topflows} and \cite[Section~5]{dilsavor}.

 As a continuous dynamical system, the flow space $\UG$ goes back to Gromov's seminal paper \cite{gromov.hypgroups} (see also \cite{matheus,champetier}) and is orbit equivalent to the geodesic flow on $\msf{T}^1 M$ when $M$ is a closed negatively curved manifold and $\G=\pi_1(M)$. For arbitrary $\G$, Mineyev \cite{mineyev.flow} was the first to introduce a metric on $\UG$ (and an appropriate parameterization of the flow) resembling hyperbolic behavior (see e.g.~\Cref{thm:mineyev}).

\subsection{Construction of the flow}\label{subsec:construction}

We consider the space $\wt\UG:=\ppG \times \R$ equipped with the product topology induced by the inclusion into $\pG \times \pG \times \R$. A \emph{line} is a set of form $[a,b]:=\{(a,b)\}\times \R$ for $a,b\in \ppG$. There are two natural actions on $\wt\UG$:
\begin{enumerate}
    \item The $\R$-action: for $t\in \R$ we define $\wt{\msf{g}}_t:\wt\UG \ra \wt\UG$ according to
    \[ \wt{\msf{g}}_t(a,b,s):=(a,b,s+t) \quad \text{ for }(a,b,s)\in \wt\UG.\]
    The action $\wt{\msf{g}}=(\wt{\msf{g}}_t)_{t\in \R}$ is the \emph{geodesic flow} on $\wt\UG$.  
    \item The $\Z/2\Z$-action: we define an involution $\wt\iota: \wt\UG \ra \wt\UG$ according to \[\wt\iota(a,b,s):=(b,a,-s) \quad \text{ for }(a,b,s)\in \wt\UG.\]
    \end{enumerate}
 Additionally, the metric $\hat d$ allows us to define a third action:
    \begin{enumerate}    
    \item[(3)] The $\G$-action: for $g\in \G$ and $(a,b,s)\in \wt\UG$ we define
    \[g(a,b,s):=(ga,gb,s-\frac{1}{2}(\hat \beta(g^{-1},a)-\hat\beta(g^{-1},b))),\]
    where $\hat\beta=\beta^{\hat d}_o$ is the Busemann function associated to $\hat d$.
    \end{enumerate}

All these actions are continuous on $\wt\UG$. Moreover, they satisfy the following properties \cite[Lemma~6]{mineyev.flow}:

\begin{itemize}
\item[(a)] The $\G$ and $\R$-actions commute.
\item[(b)] The $\G$ and $\Z/2\Z$-actions commute.
\item[(c)] The $\Z/2\Z$-action anti-commutes with the $\R$-action:
$\wt\iota (\wt{\msf{g}_t}(\wt\vv)) = \wt{\msf{g}}_{-t}(\wt\iota(\wt\vv))$ for $\wt\vv \in \wt\UG$ and $t\in \R$.
\item[(d)] The three actions map lines onto lines. 
\item[(e)] The $\G$-action is free, properly discontinuous, and cocompact.
\end{itemize}

Using the double difference $\gpr{\cdot,\cdot}{\cdot,\cdot}$ in $(\G,\hat d)$, we define a \emph{Busemann cocycle} $\beta^\times: \wt\UG \times \wt\UG\times \G\ra \R$ according to
\[\beta^\times(\wt\vv,\wt\vw;g):=\gpr{a}{b}_g+|s-\gpr{a,b}{g,o}|-\gpr{a'}{b'}_g-|s'-\gpr{a',b'}{g,o}|\]
for $g\in \G$ and $\wt\vv=(a,b,s), \wt\vw=(a',b',s')\in \wt\UG$; see \cite[Definition~10 \& Section~8.4]{mineyev.flow}. This cocycle satisfies the following properties \cite[Theorem~11]{mineyev.flow}:
\begin{itemize}
    \item[(a)] For $\wt\vv,\wt\vw\in \wt\UG$, the map $g\mapsto \beta^\times(\wt\vv,\wt\vw;g)$ is Lipschitz on $(\G,\hat d)$.
\item[(b)] $\beta^\times(\cdot,\cdot;g)$ satisfies the cocycle condition: $$\beta^\times(\wt\vvu,\wt\vw;g)=\beta^\times(\wt\vvu,\wt\vv;g)+\beta^\times(\wt\vv,\wt\vw;g)
\quad \text{ for }g\in \G \text{ and }\wt\vvu,\wt\vv,\wt\vw\in \wt\UG.$$
\item[(c)] $\beta^\times(\cdot,\cdot;g)$ is $\Z/2\Z$-invariant in each variable: $$\beta^\times(\wt\vv, \wt\vw;g) = \beta^\times(\wt\vv, \wt\iota(\wt\vw);g)=\beta^\times(\wt\iota(\wt\vv), \wt\vw;g) \quad \text{ for } \wt\vv,\wt\vw\in \wt\UG.$$
\item[(d)] $\beta^\times$ is $\G$-invariant: $$\beta^\times(g\wt\vv, g\wt\vw;gh) = \beta^\times(\wt\vv,\wt\vw;h) \quad \text{ for }g,h \in \G \text{ and }\wt\vv,\wt\vw\in \wt\UG.$$
\end{itemize}
Indeed, the cocycle $\beta^\times$ continuously extends to $\wt\UG \times \wt\UG \times \ov\G$ while still satisfying properties (a)-(d) above \cite[Theorem~55~(a')]{mineyev.flow}. For example, if $a,b,b'\in \partial \G$ are pairwise distinct and $s,s'\in \R$, then
\begin{equation}\label{eq:betahoro}
\beta^\times((a,b,s),(a,b',s');a)=s-s'+\gpr{b,b'}{a,o}.
\end{equation}
Similarly, for $a,a',b\in \partial \G$ pairwise distinct and $s,s'\in \R$, we have
\begin{equation*}
\beta^\times((a,b,s),(a',b,s');b)=s'-s+\gpr{a,a'}{b,o}.
\end{equation*}

We use the cocycle $\b^\times$ to define metrics $d^\times$ and $d_{\wt\UG}$ on $\wt\UG$ as follows. First, we define
\begin{equation*}
d^\times(\wt\vv,\wt\vw) := \sup_{g \in \G} |\beta^\times(\wt\vv,\wt\vw;g)|.
\end{equation*}

This metric satisfies that $t\mapsto \wt{\msf{g}}_t(\wt\vv)$ is a unit speed isometric embedding into $(\wt\UG,d^\times)$ for each $\wt\vv\in \wt\UG$. Next, we fix a strong hyperbolicity parameter $\lam$  for $\hat d$ (see \Cref{eq:defstronghyp}) and a number $\lam'\in (0,\lam/2)$. Then we set
\begin{equation}\label{eq:defd_UG}
   \qquad
d_{\wt\UG}(\wt\vv,\wt\vw)
:= 1/C_{\lam'}\int_{-\infty}^{\infty}
d^\times(\wt{\msf{g}}_t(\wt\vv),\wt{\msf{g}_t}(\wt\vw))e^{-\lam'|t|}dt, 
\end{equation}
where $C_{\lam'}$ is chosen so that $t\mapsto \wt{\msf{g}}_t(\wt\vv)$ is again a unit speed isometric embedding into $(\wt\UG,d_{\wt\UG})$. We call $\lam'$ the \emph{expansion/contraction} rate for $d_{\wt\UG}$.

The metric $d_{\wt\UG}$ satisfies the following properties.

\begin{theorem}[{\cite[Theorem~60]{mineyev.flow}}]\label{thm:mineyev} The pair $(\wt\UG,d_{\wt\UG})$ equipped with the geodesic flow $\wt{\msf{g}}=(\wt{\msf{g}}_t)_{t\in \R}$, the involution $\wt\iota$ and the $\G$-action satisfies the following properties.
\begin{enumerate}
\item $d_{\wt\UG}$ is $\G$-invariant, $\wt\iota$-invariant, and compatible with the topology on $\wt\UG$. 
\item Any orbit map induces a $\G$-equivariant rough isometry $(\G,\hat d)\ra(\wt\UG, d_{\wt\UG})$.
\item For each $\wt\vv \in \wt\UG$ the orbit map $\R\ra (\wt\UG, d_{\wt\UG})$ given by $t \mapsto \wt{\msf{g}}_t(\wt\vv)$ is an isometric embedding.
\item There exists a constant $C>0$ such that for each $t\in \R$, the translation map $\wt{\msf{g}}_t:(\wt\UG,d_{\wt\UG})\ra (\wt\UG,d_{\wt\UG})$ is a biLipschitz homeomorphism with constant $Ce^{\lam'|t|}$.
\item For any non-trivial $g \in \G$ and $\wt\vz=(g^-, g^+,s)$ for some $s\in \R$, we have $d_{\wt\UG}(\wt\vz, g\wt\vz) =\ell_{\hat d}[g]$.
\item There exists  $M>0$ such that for any triplet $a,b,c\in \partial \G$ of pairwise distinct points and $t\in \R$ we have
\[d_{\wt\UG}((a,c,t+\gpr{a,c}{b,o}),(b,c,t+\gpr{b,c}{a,o}))\leq Me^{-\lam' t}.\]
\end{enumerate}
\end{theorem}

Given a triplet of pairwise distinct points $a,b,c\in \pG$, we define the \emph{projection} of $a$ on the flow line $[b,c]$ according to \begin{equation}
\label{eq:proj_flowline}
p_{[b,c]}(a):=(b,c,\gpr{b,c}{a,o}),
\end{equation} see \cite[Section~2.9]{mineyev.flow}.

\begin{remark}\label{rmk:updated}
 \begin{enumerate}
     \item Our definition of $d_{\wt\UG}$ in \Cref{eq:defd_UG} differs from Mineyev's definition \cite[Section~8.6]{mineyev.flow} in the use of the auxiliary constant $\lam'$. This change is so that the bounds from items (4) and (6) above are in terms of the same exponential rates, as in \cite[Proposition~5.1.9]{dilsavor}. The rest of the properties can be proven as in Mineyev's original definition, see also \cite[Section~5]{dilsavor}.
     \item By the definition of projection $p_{[b,c]}(a)$~in \Cref{eq:proj_flowline}, item (6) translates to \begin{equation}\label{eq:approxstable}
         d_{\wt\UG}(\wt\vg_t(p_{[a,c]}(b)),\wt\vg_t(p_{[b,c]}(a)))\leq Me^{-\lam't} \quad \text{ for all }t\geq 0, 
     \end{equation}
     and similarly
\begin{equation}\label{eq:approxunstable}
         d_{\wt\UG}(\wt\vg_{-t}(p_{[a,b]}(c)),\wt\vg_{-t}(p_{[a,c]}(b)))\leq Me^{-\lam't} \quad \text{ for all }t\geq 0. 
     \end{equation}
     In particular, we have $d_{\wt\UG}(p_{[a,c]}(b),p_{[b,c]}(a))\leq M$ for all pairwise distinct $a,b,c\in \pG$, as expected by Gromov hyperbolicity.
 \end{enumerate}
\end{remark}


We define $\UG=\msf{U}_{\hat d}\G$ as the quotient of $\wt\UG$ under the $\G$-action, which is a compact space. Let $\pi:\wt\UG \ra \UG$ be the quotient map. Since the $\R$-action via $\wt{\msf{g}}$ commutes with the $\G$-action, we let $\msf{g}=(\msf{g}_t)_{t\in \G}$ be the induced flow on $\UG$. Similarly, the involution $\wt\iota$ on $\wt\UG$ descends to an involution $\iota$ on $\UG$ that anti-commutes with $\vg$. We also define the metric $d_{\UG}$ on $\UG$ by the formula
\[d_{\UG}(\vv,\vw)=\inf\{d_{\wt\UG}(\wt\vv,\wt\vw):\pi(\wt\vv)=\vv,\pi(\wt\vw)=\vw\} \quad \text{ for }\vv,\vw\in \UG.\]
This metric is compatible with the quotient topology on $\UG$.

\begin{definition}\label{def:mineyevflow}
    The \emph{Mineyev's flow space} associated to $\G$ (with respect to $\hat d$) is the tuple $\UG=(\UG,\msf{g},\iota,d_{\UG})$. We refer to $\vg$ as the \emph{geodesic flow} on $\UG$. 
\end{definition}


\subsection[Properties of $(\wt{\UG},\wt{\vg},d_{\wt{\UG}})$]
           {Properties of $(\wt{\UG},\wt{\vg},d_{\wt{\UG}})$}
In this subsection we record some further properties about the geometry of the space $(\wt\UG,d_{\wt\UG})$. These are well-known for the usual geodesic flow on the universal cover of a closed negatively curved manifold.

\begin{lemma}\label{lem.nicecompact}
    There exists a compact set $K\subset \wt\UG$ such that for any $g\in \G$ we can find $\wt\vv\in K$ and $t\in \R$ with $\wt{\msf{g}}_t(\wt\vv)\in gK$.
\end{lemma} 

\begin{proof}
For the space $(\G,\hat d)$ we use the following well-known fact: there exists $\al>0$ such that 
any two points $g,h\in \G$ belong to a bi-infinite $(\al,\hat d)$-rough geodesic path. This is true for word metrics on $\G$, and the result for an arbitrary $\hat d$ follows from the Morse lemma.

Now, let $K_0\subset \ppG$ be the set of pairs $(a,b)$ such that $\gpr{a}{b}_o\leq 3\al/2$. We claim that for any $g\in \G$ the intersection $K_0 \cap g^{-1}K_0$ is non-empty. Indeed, let $a,b$ be the endpoints of a bi-infinite $(\al,\hat d)$-rough geodesic in $\G$ containing $o$ and $g^{-1}$. Since $\hat d$ is strongly hyperbolic, we have $\gpr{a}{b}_o\leq 3\al/2$ and $\gpr{ga}{gb}_o=\gpr{a}{b}_{g^{-1}}\leq 3\al/2$, so that $(a,b)$ and $(ga,gb)$ belong to $K_0$. 
To finish the proof, let $K=K_0\times \{0\}\subset \ppG\times \R=\wt\UG$. Given $g\in \G$, by our previous claim there exists $(a,b)\in K_0\cap g^{-1}K_0$. We define $\wt\vv:=(ga,gb,0)\in K$ so for $t=-\frac{1}{2}(\hat\beta(g^{-1},a)-\hat\beta(g^{-1},b))$ we have $$\wt{\msf{g}}_t(\wt\vv)=(ga,gb,-\frac{1}{2}(\hat\beta(g^{-1},a)-\hat\beta(g^{-1},b)))=g\wt\vv\in gK. \qedhere$$
\end{proof}

Our next result states that if two flow-line segments are close at their endpoints, then they are exponentially close in their interiors.

\begin{proposition}\label{prop.geodesicsclose} Let $\lam'$ the expansion/contraction rate for $d_{\wt\UG}$. Then for any $\ep>0$ there exist $\del,T_0>0$ such that if $\wt\vv,\wt\vw\in \wt\UG$ and $T>T_0$ satisfy
\begin{equation}\label{eq:asdelT}
d_{\wt\UG}(\wt\vv,\wt\vw)<\del \quad \text{ and } \quad d_{\wt\UG}(\wtg_T(\wt\vv),\wtg_T(\wt\vw))<\del,
\end{equation}
then there exists $t_0\in (-\ep,\ep)$ such that
\begin{equation*}\label{eq:exponentialdecay}
    d_{\wt\UG}(\wtg_t(\wt\vv),\wtg_{t+t_0}(\wt\vw))< \ep e^{-\lam'\min\{t,T-t\}}\quad \text{ for all }t_0\leq t \leq T-t_0.
\end{equation*}
\end{proposition}

\begin{lemma}\label{lem:timecontrol}
Let $\lam'$ be the expansion/contraction constant for $d_{\wt\UG}$. Then for any $T>0$ there exists $\del_0$ satisfying the following. Let $a,b,c\in \pG$ be pairwise distinct, and let $\wt\vv=p_{[b,c]}(a)$ and $\wt\vw=p_{[a,c]}(b)$. Then for any $t\leq T$ we have $d_{\wt\UG}(\wt\vg_t(\wt\vv),\wt\vg_t(\wt\vw))\geq \del_0$.
\end{lemma}

\begin{proof}
    Suppose the result is not true and let $(a_k)_k,(b_k)_k,(c_k)_k$ be sequences of pairwise distinct points in $\pG$, and a sequence of real numbers $(t_k)_k$ such that $t_k\leq T$ for all $k$, but
    $d_{\wt\UG}(\wt\vg_{t_k}(\wt\vv_k),\wt\vg_{t_k}(\wt\vv_k))\to 0$ as $k\to \infty$ for $\wt\vv_k=p_{[b_k,c_k]}(a_k)$ and $\wt\vw_k=p_{[a_k,c_k]}(b_k)$. 

    Since $\G$ acts cocompactly on the space of distinct triples of $\pG$, after translating elements of this sequence by elements of $\G$ and taking a subsequence, we can assume that $a_k\to a$, $b_k\ra b$ and $c_k\ra c$ for some pairwise distinct points $a,b,c\in \pG$.  Therefore, by the definition of the projection (\Cref{eq:proj_flowline}) and the continuity of $\gpr{\cdot,\cdot}{\cdot,\cdot}$ we have that $\wt\vv_k \ra \wt\vv=p_{[b,c]}(a)$ and $\wt\vw_k \ra \wt\vw=p_{[a,c]}(b)$. 
    
    If the sequence $(t_k)_k$ is bounded, using \Cref{thm:mineyev}~(1) and (3), up to taking a subsequence we can assume that $t_k \to t\in (-\infty,T]$ and that $\wt\vg_{t}(\wt\vv)=\wt\vg_t(\wt\vw)$. This is a contradiction since $a\neq b$.

    On the other hand, if up to a subsequence we have $t_k\to -\infty$, then $\wt\vg_{t_k}(\wt\vv_k)\to b$ and $\wt\vg_{t_k}(\wt\vw_k)\to a$, again a contradiction.
\end{proof}

Given $\wt\vv=(a,b,s)\in \wt\UG$, its \emph{stable horosphere} is defined as the set 
\[H^s(\wt\vv)=\{\wt\vv'=(a',b,s-\gpr{a,a'}{b,x_0})|a'\in \partial \G\bs \{b\}\}.\] 
Similarly, its \emph{unstable horosphere} is $$H^u(\wt\vv)=\{\wt\vv'=(a,b',s+\gpr{b,b'}{a,x_0})|b'\in \partial \G \bs \{a\}\}.$$
A \emph{stable} (resp. \emph{unstable}) \emph{horosphere} is a set of form $H^s(\wt\vv)$ (resp. $H^u(\wt\vv)$) for some $\wt\vv\in \wt\UG$.

From \Cref{lem:timecontrol} we immediately deduce the following.

\begin{corollary}\label{cor:closehorospheres}
    For any $\ep > 0$ there exists $\del_0 > 0$ satisfying the following. Let $\wt\vv,\wt\vw\in \wt\UG$ be points on the same stable (resp. unstable) horosphere. If $d_{\wt\UG}(\wt\vv,\wt\vw)<\del_0$, then $d_{\wt\UG}(\wt\vg_t(\wt\vv),\wt\vg_t(\wt\vw))<\ep e^{-\lam'|t|}$ for any $t\geq 0$ (resp. $t\leq 0$).
\end{corollary}

\begin{proof}
    Let $M$ be the constant from \Cref{thm:mineyev}~(6) and $T>0$ such that $Me^{-\lam'T}\leq \ep$. Then the lemma for the stable case follows by choosing $\del_0$ as in \Cref{lem:timecontrol} for $T$ and applying \Cref{thm:mineyev}~(6) (see \Cref{eq:approxstable} and \Cref{eq:approxunstable}). The unstable case follows by the same argument.
\end{proof}

\begin{proof}[Proof of \Cref{prop.geodesicsclose}]
    For the proof we fix $\ep>0$. Let $\gpr{\cdot}{\cdot}$ be the Gromov product for $\hat  d$ based at the identity, and let $C$ and $R_0$ be the constants from \Cref{lem:nicaspacula} applied to $\hat d$ (recall that $\lam'>0$ is less than half the strong hyperbolicity parameter of $\hat d$). Let $R_1>R_0$ be large enough so that $Ce^{-\lam' R_1}\leq \ep/3$. We also let $K$ be a compact subset of $\wt\UG$ whose $\G$-translates cover $\wt\UG$. Then there exists $A>0$ such that $\gpr{a}{b}\leq A$ for any $(a,b,s)\in K$.

    By \Cref{cor:closehorospheres}, let $\del_0\in (0,3\ep)$ be such that for $\wt\vv,\wt\vw\in \wt\UG$ in the same stable horosphere, $d_{\wt\UG}(\wt\vv,\wt\vw)<\del_0$ implies $d_{\wt\UG}(\wt\vg_t(\wt\vv),\wt\vg_t(\wt\vw))<\ep/2e^{-\lam'(\ep+t)}$ for all $t\geq 0$.

    Using \Cref{thm:mineyev}~(1), let $\k_0\in (0,\del_0/9)$ be small enough so that if $(a,b,s)\in K$ and $(a',b',s')\in \wt\UG$ satisfies $|s-s'|<\k_0$ and $\gpr{a}{a'},\gpr{b}{b'}>\k_0^{-1}$, then $d_{\wt\UG}((a,b,s),(a',b',s'))<\del_0$. We also consider $R_2>\max\{R_1,\k_0^{-1}\}$ so that $Ce^{-\lam'R_2}<\k_0$.

    Finally, we let $T_0>\ep$ be so that $\ep/2 e^{-\lam'T_0}<\del_0/3$, and $\del\in (0,\del_0/3)$ so that $d_{\wt\UG}((a,b,s),(a',b',s'))<\del$ for $(a,b,s)\in K$ implies
    \begin{itemize}
        \item $|s-s'|<\k_0$.
        \item $\gpr{a}{a'}, \gpr{b}{b'}>R_2+2A$.
        \item $\gpr{a'}{b},\gpr{a}{b'},\gpr{a'}{b'}\leq 2A$.
    \end{itemize}
    For these values of $\del$ and $T_0$, let $\wt\vv=(a,b,s), \wt\vw=(a',b',s')\in \wt\UG$ satisfy \Cref{eq:asdelT}, and without loss of generality assume that $\wt\vv\in K$. Then we have 
    \[\gpr{a,b}{o,a'}=\gpr{a}{a'}-\gpr{b}{a'}\geq R_2 \quad \text{ and }\quad \gpr{b',a'}{o,b}=\gpr{b}{b'}-\gpr{a'}{b}\geq R_2. \]
    Consequently, \Cref{lem:nicaspacula} implies  \begin{equation}\label{eq:smallCR}
   \gpr{a,a'}{o,b}, \gpr{b,b'}{o,a'}\leq C^{-\lam'R_2}<\k_0<\del_0/9<\ep/3.
    \end{equation}

   Let $\wt\vv'$ be the point in $[a',b]$ in the same stable horosphere as $\wt\vv$, and $\wt\vw'$ be the point in $[a',b]$ in the same unstable horosphere as $\wt\vw$. Then we have $\wt\vv'=(a',b,s+\gpr{a,a'}{o,b})$ and $\wt\vw'=(a',b,s'+\gpr{b,b'}{o,a'})$. 
    Let $t_0$ be such that $\wt\vv'=\wt\vg_{t_0}(\wt\vw')$, and note from \Cref{eq:smallCR} that 
    \begin{equation}\label{eq:t_0}
        |t_0|=d_{\wt\UG}(\wt\vv',\wt\vw')=|s-s'+\gpr{a,a'}{o,b}-\gpr{b,b'}{o,a'}| <3\k_0<\del_0/3<\ep. 
    \end{equation}

     We are left to bound the distance $d_{\wt\UG}(\wt\vg_t(\wt\vv),\wt\vg_{t+t_0}(\wt\vw))$ for $0\leq t\leq T-t_0$. In order to do this, we note that for any such value $t$, the identity $\wt\vv'=\wt\vg_{t_0}(\wt\vw')$ implies
     \[d_{\wt\UG}(\wt\vg_t(\wt\vv),\wt\vg_{t+t_0}(\wt\vw))\leq d_{\wt\UG}(\wt\vg_t(\wt\vv),\wt\vg_{t}(\wt\vv'))+d_{\wt\UG}(\wt\vg_{t+t_0}(\wt\vw'),\wt\vg_{t+t_0}(\wt\vw)).\]
In addition, our choice of $\k_0$, \Cref{eq:smallCR}, and the fact that $R_2>\k_0^{-1}$ imply that $d_{\wt\UG}(\wt\vv,\wt\vv')<\del_0$. Therefore, by \Cref{cor:closehorospheres}
     we have \begin{equation}\label{eq:ep/2lam}
         d_{\wt\UG}(\wt\vg_t(\wt\vv),\wt\vg_t(\wt\vv'))<\ep/2e^{-\lam't} \quad \text{ for all }t\geq 0.
     \end{equation}

     Also, since $T>T_0$ and $\del< \del_0/3$, \Cref{eq:ep/2lam},  \Cref{thm:mineyev}~(3) and \Cref{eq:t_0} give us
     \begin{align*}
         d_{\wt\UG}(\wt\vg_T(\wt\vw),\wt\vg_T(\wt\vw')) &  \leq  d_{\wt\UG}(\wt\vg_T(\wt\vw),\wt\vg_T(\wt\vv))+d_{\wt\UG}(\wt\vg_T(\wt\vv),\wt\vg_T(\wt\vv'))+d_{\wt\UG}(\wt\vg_T(\wt\vv'),\wt\vg_{T-t_0}(\wt\vw'))\\
         & <\del + \ep/2e^{\lam' T_0}+|t_0|\\
         & <\del_0/3+\del_0/3+\del_0/3=\del_0.
     \end{align*}
    Hence, by \Cref{cor:closehorospheres} and \Cref{eq:t_0} we also have 
    \[d_{\wt\UG}(\wt\vg_{t+t_0}(\wt\vw'),\wt\vg_{t+t_0}(\wt\vw))<\ep/2e^{-\lam'(\ep+T-(t+t_0))}\leq \ep/2e^{-\lam'(T-t)} \quad \text{ for }t\leq T-t_0.\]
    Combining this inequality with \Cref{eq:ep/2lam}, we deduce \Cref{eq:exponentialdecay}, concluding the proof of the proposition. 
\end{proof}

By the same argument, we can also prove the following related result. We leave the details to the reader.

\begin{proposition}\label{prop.geodesicsclose2} For any $\ep_1,A>0$ there exists $T_1>0$ such that if $\wt\vv,\wt\vw\in \wt\UG$ and $T>T_1$ satisfy
\begin{equation*}
d_{\wt\UG}(\wt\vv,\wt\vw) \leq A \quad \text{ and } \quad d_{\wt\UG}(\wtg_T(\wt\vv),\wtg_T(\wt\vw)) \leq A,
\end{equation*}
then there exists $t_0\in (-T_1,T_1)$ such that
\begin{equation*}
    d_{\wt\UG}(\wtg_t(\wt\vv),\wtg_{t+t_0}(\wt\vw))< \ep_1 e^{-\lam'\min\{t,T-t\}}\quad \text{ for all }T_1\leq t \leq T-T_1. \qed
\end{equation*}
\end{proposition}


\subsection{Invariant measures and currents}\label{subsec:invariantmeasures}

We continue with the notation from the previous subsection and denote the space of real-valued continuous functions on $\UG$ by $C(\UG)$. Sometimes, we will refer to functions in $C(\UG)$ as \emph{potentials}.

A Borel measure $\frakm$ on $\UG$ is $\msf{g}$-\emph{invariant} if $(\msf{g}_t)_\ast\frakm=\frakm$ for any $t\in \R$. Similarly, we can talk of $\wt{\msf{g}}$-invariant measures on $\wt\UG$. We denote by $\calM_{\msf{g}}(\UG)$ the space of $\msf{g}$-invariant Borel probability measures on $\UG$ equipped with the weak$^\ast$ topology. Note that $\calM_{\msf{g}}(\UG)$ is compact and metrizable.

Invariant measures supported on periodic orbits correspond to conjugacy classes in $\G$. Namely, for a non-trivial conjugacy class $[g]\in [\G]'$ there is a corresponding periodic orbit $\gam_{[g]}$ on $\UG$, given by the image of the curve $t \mapsto (g^-,g^+,t)$ in $\wt \UG$. By \Cref{thm:mineyev}~(5), the length of the periodic orbit $\gam_{[g]}$ is equal to $\ell_{\hat d}[g]$. Then $[g]$ determines a $\msf{g}$-invariant Borel probability measure $\frakm_{[g]}$ on $\UG$ given by the normalized integral along $\gam_{[g]}$:
\begin{equation}\label{eq:def.m_g}
  \int F d\frakm_{[g]}=\frac{1}{\ell_{\hat d}[g]}\int_{\gam_{[g]}}{F}:=\frac{1}{\ell_{\hat d}[g]}\int_{0}^{\ell_{\hat d}[g]}F(\msf{g}_t(\vv))dt\quad \text{ for all }F\in C(\msf{U}\G),  
\end{equation} where $\vv$ is any point on $\gam_{[g]}$. Invariant measures supported on periodic orbits are dense in $\calM_{\vg}(\UG)$. This follows, for instance, from \Cref{lem:homeoCM} below and the density of rational currents on $\calC_\G$ \cite{bonahon.currentshpygroups}.

The correspondence above extends to a natural bijection between geodesic currents and Borel $\msf{g}$-invariant measures on $\UG$ \cite[Section~3.2]{tanaka.topflows}. Given a geodesic current $\Lam \in \calC_\G$, let $\wt\frakm_\Lam:=\Lam\otimes \Leb$ be the product measure on $\wt\UG=\ppG \times \R$, where $\Leb$ is the (normalized) Lebesgue measure on $\R$. Then $\wt\frakm_{\Lam}$ is a Radon  $\wt{\msf{g}}$-invariant measure. Moreover, $\wt\frakm_{\Lam}$ is also invariant under the $\G$-action on $\wt\UG$, hence it descends to a $\msf{g}$-invariant measure $\frakm_{\Lam}$ on $\UG$. We have the following.

\begin{lemma}[{\cite[Lemma~3.4]{tanaka.topflows}}]\label{lem:tanaka}
    The measure $\frakm_{\Lam}$ is the unique $\msf{g}$-invariant Borel measure on $\UG$ such that if $\wh F:\wt\UG \ra \R$ is a compactly supported continuous function and
    \begin{equation}\label{eq:relFhatF}
        F(\vv):=\sum_{\pi(\wt\vv)=\vv}{\wh F(\wt\vv)} \quad \text{ for }\vv\in \UG,
    \end{equation}
    then 
    \begin{equation}\label{eq:intFhatF}
        \int_{\wt\UG}{\wh F}{d\wt\frakm_\Lam}=\int_{\UG}{F}{d\frakm_\Lam}.
    \end{equation}
\end{lemma}

Since for any $F\in C(\UG)$ there exists a compactly supported continuous function $\wh F:\wt\UG \ra \R$ satisfying \Cref{eq:relFhatF}, the identity from \Cref{eq:intFhatF} gives us that the convergence $\Lam_k \xrightharpoonup{\ast} \Lam$
in $\calC_\G$ implies the weak$^\ast$ convergence $\frakm_{\Lam_k} \xrightharpoonup{\ast} \frakm_{\Lam}$ in the space of $\msf{g}$-invariant measures on $\UG$. 

 Given a projective geodesic current $[\Lam]\in \bbP\calC_\G$, we let $\frakm_{[\Lam]}\in \calM_{\msf{g}}(\UG)$ be the unique probability measure in the projective class of the measure $\frakm_\Lam$. It is not hard to see that for the rational current $\eta_{[g]}$ associated to $[g]\in [\G]'$, the measure $\frakm_{[\eta_{[g]}]}$ is equal to the measure $\frakm_{[g]}$ from \Cref{eq:def.m_g}.

\begin{lemma}\label{lem:homeoCM}
The map $\msf{m}: \bbP\calC_\G \ra \calM_{\msf{g}}(\UG)$ that maps $[\Lam]$ to $\frakm_{[\Lam]}$ is a homeomorphism.
\end{lemma}

\begin{proof}
    From our discussion above, it follows that $\msf{m}$ is continuous. Since both $\bbP\calC_\G$ and $\calM_{\msf{g}}(\UG)$ are compact and metrizable, it is enough to prove that $\msf{m}$ is a bijection. 

   To produce an inverse of $\msf{m}$, we follow the construction by Kaimanovich \cite[Thm.~2.2]{kaimanovich}. Let $\frakm\in \calM_{\msf{g}}(\UG)$, and lift it to the unique $\wt{\msf{g}}$-invariant Radon measure $\wt\frakm$ on $\wt\UG$ satisfying \Cref{eq:intFhatF}. If $p:\wt\UG=\ppG \times \R \ra \ppG$ denotes the projection map, then there exists a unique probability measure $\Lam$ on $\ppG$ such that for every Borel subset $\calA\subset \wt\UG$ we have 
\begin{equation*}
    \wt\frakm(\calA)=\int_{p(\calA)}{\mathrm{Leb}_{(a,b)}(p^{-1}((a,b)))}{d\Lam(a,b)},
\end{equation*}
where $\mathrm{Leb}_{(a,b)}$ is the normalized Lebesgue measure along the flow line $[a,b]=\{(a,b)\}\times \R=\R$.
$\Lam$ is clearly Radon, and $\G$-invariance of $\wt\frakm$ implies that $\Lam$ is $\G$-invariant. Hence $\Lam$ is a geodesic current, and as in the proof of \cite[Theorem~2.2]{kaimanovich}, it can be checked that $[\Lam]=\msf{m}^{-1}(\frakm)$. We leave the details to the reader.
\end{proof}


\subsection{Metric-Anosov property and Ledrappier correspondence}\label{subsec:metricanosovledrappier}

A crucial property that we will use is that Mineyev's flow $(\UG,\vg,d_{\UG})$ is \emph{metric-Anosov}---a broad extension of the notion of Anosov flow---, as was proven recently in Disalvor's thesis \cite{dilsavor}.

For $\CAT(-1)$ groups, this was first studied in \cite{CLT.strong,CLT.weak} (see also \cite{dilsavor-thompson} for non-necessarily convex cocompact $\CAT(-1)$ groups). In particular, \cite{CLT.strong} proves that Mineyev's flow for $\CAT(-1)$ groups is metric-Anosov, which was also proven for groups admitting Anosov representations in \cite{BCLS.pressure} (up to H\"older orbit equivalence). A strong coding for Mineyev's flow space was proven in \cite{cantrell-tanaka.invariant} without relying on the metric-Anosov property (see also \cite{tanaka.topflows,cantrell-tanaka.manhattan}). 
In \Cref{prop:weakspecification} we show that Mineyev's flow space for any non-elementary torsion-free  hyperbolic group satisfies the specification property.

\begin{theorem}[{\cite[Theorem~D]{dilsavor}}]\label{thm:mineyevanosov}
    For any strongly hyperbolic metric $\hat d\in \calD_\G$, then Mineyev's flow space $\UG=(\UG,\vg,d_{\UG})$ is metric-Anosov. 
\end{theorem}

The metric-Anosov property will not be used directly, so we refer the reader to \cite{pollicott},  \cite[Section~2]{sambarino.report}, \cite[Section~3]{CLT.strong}. However, this property has some remarkable consequences. For instance, it will allow us to establish versions of Ledrappier correspondence between potentials on $\UG$ and cocycles on $\G$ \cite[Th\'eor\`eme~3]{ledrappier}, as was proven by Sambarino \cite{sambarino.report}.

On the side of $\G$, a \emph{cocycle} on $\G$ is a continuous map $\msf{c}:\G \times \pG\ra \R$ such that
\[\msf{c}(gh,a)=\msf{c}(g,ha)+\msf{c}(h,a) \quad \text{ for }g,h\in \G, a\in \pG.\]
Such a cocycle is \emph{H\"older} if each function $a\mapsto \msf{c}(g,a)$ is H\"older with the same exponent for all $g\in \G$, when $\pG$ is equipped with some (hence, any) visual metric.

For a cocycle $\msf{c}$, its \emph{period map} is the function  $\msf{L}_{\msf{c}}:[\G]'\ra \R$ such that
\[\msf{L}_{\msf{c}}[g]=\msf{c}(g,g^+) \quad \text{ for }[g]\in [\G]'.\]
We say that $\msf{c}$ is \emph{properly positive} if there exists $d\in \calD_\G$ such that $\msf{L}_{\msf{c}}\geq \ell_d$.
\begin{example}\label{ex:cocyclestrongly}
Let $d\in \calD_\G$ be a Green metric with respect to a symmetric finitely supported probability measure on $\G$ whose support generates $\G$. Then the corresponding Busemann function $\beta=\beta^{d}_o: \G \times \pG \ra \R$ can be defined as a true limit and the cocycle  $\msf{c}=\msf{c}_{d}: \G \times \pG \ra \R$ given by
\[\msf{c}(g,a)=\beta(g^{-1},a)\]
is H\"older. This follows by combining \cite[Theorem~3.3]{INO} and \cite[Theorem~1.7]{BHM.harmonic}. Moreover, we also have $\msf{L}_{\msf{c}}[g]=\ell_{d}[g]$ for all $[g]\in [\G]'$, so that $\msf{c}$ is properly positive. The same is true for any $d\in \calD_\G$ that is strongly hyperbolic, which holds for Green metrics \cite[Theorem~1.3]{nica-spakula}.
\end{example}


Two cocycles $\msf{c},\msf{c}'$ on $\G$ are \emph{cohomologous} if there exists a continuous function $\al:\pG \ra \R$ such that \[\msf{c}(g,a)-\msf{c}'(g,a)=\al(a)-\al(ga) \quad \text{ for all }g\in \G,a\in \pG.\]

On the side of $\UG$, we say that two potentials $F,F'\in C(\UG)$ are \emph{cohomologous} if there exists a continuous potential $U:\UG \ra \R$ that is $C^1$ in the direction of the flow (that is, $t\mapsto U(\vg_t(\vv))$ is $C^1$ for any $\vv\in \UG$), and such that
\[F(\vv)-F'(\vv)=\left.\frac{\partial}{\partial t}\right|_{t=0}U(\vg_t(\vv)) \quad \text{ for }\vv\in \UG.\]

It follows that if $F$ and $F'$ are cohomologous then they are \emph{weakly cohomologous}, in the sense that
\[\int{F}d\frakm=\int{F'}d\frakm \quad \text{ for any }\frakm\in \calM_{\vg}(\UG).\]

On the other hand, and using that the flow $(\UG,\vg)$ is transitive (see e.g.~\cite[Proposition~5.1]{BCLS.pressure}) and metric-Anosov, we have that these two notions coincide among H\"older continuous potentials. Using the closing lemma (available in this setting by the metric-Anosov property) and \Cref{prop.geodesicsclose}, one can prove \emph{Liv\v{s}ic's theorem} in this context \cite{livsic}.

\begin{theorem}[Sambarino, {\cite[Theorem~2.2.2]{sambarino.report}}]\label{thm:livsic}
If $F,F':\UG \ra \R$ are H\"older continuous, then they are cohomologous if and only if they are weakly cohomologous.
\end{theorem}

Using the theorem above, Sambarino proved in \cite{sambarino.report} the following version of Ledrappier correspondence for metric-Anosov Mineyev's flow spaces. Combining this result with \Cref{thm:mineyevanosov}, we have the following.

\begin{proposition}[{Sambarino, \cite[Proposition~3.1.1]{sambarino.report}}]\label{prop:ledrappier}
    If $\msf{c}:\G \times \pG \ra \R$ is a H\"older cocycle then there exists a H\"older continuous function $F=F_{\msf{c}}:\UG \ra \R$ such that
    \[\msf{L}_{\msf{c}}[g]=\int_{\gam_{[g]}} F\quad \text{ for }[g]\in [\G]'.\]
    Moreover, if $\msf{c}$ and $\msf{c}'$ are cohomologous, then $F_{\msf{c}}$ and $F_{\msf{c}'}$ are cohomologous.
\end{proposition}


\section{Construction of the dictionary}\label{sec:dictionary}

In this section we let $\G$ be a non-elementary torsion-free hyperbolic group equipped with a Mineyev's flow space $\UG=(\msf{U}\G,\msf{g},d_{\UG})$ constructed from the strongly hyperbolic metric $\hat d\in \calD_\G$, as in \Cref{sec:mineyevflow}. 
Our goal is to prove \Cref{thm:dictionary}, which implies \Cref{mainthm:dictionary_manifold}  when $\G$ is the fundamental group of a closed negatively curved manifold. The proof relies on \Cref{mainthm.greendense}, whose proof is deferred to \Cref{sec:proofgreen}.

We establish a dictionary between the (appropriate equivalence classes) of the following objects:
\begin{enumerate}
    \item Properly positive cocycles $\msf{c}:\G \times \partial^2 \G\ra \R$ (\Cref{subsec:metricanosovledrappier}).
    \item Positive continuous functions $F:\UG \ra \R$.
    \item Reparameterizations of Mineyev's flow space $(\msf{U}\G,\msf{g})$ (\Cref{def:reparameterizations}).
    \item Properly positive hyperbolic metric potentials in  $\MC(\calH_\G^{++})$ (\Cref{def:MCH}).
\end{enumerate}


\subsection{The space of reparameterizations and dictionary $\Par(\UG) \Ra C(\UG)$}\label{subsec:reparameterizations}

In this section we define the space of reparameterizations of the geodesic flow.

\begin{definition}[Flow reparameterizations]\label{def:reparameterizations}
    A \emph{reparameterization} of Mineyev's flow space $(\UG,\vg)$ is a continuous flow $\sfP=(\sfP_t)_{t\in \R}$ on $\UG$ whose flow lines coincide with the flow lines for $\vg$ and respect the orientation. The reparameterization $\sfP$ is \emph{H\"older} if the map $\UG\times \R \ra \UG \times \R$ that sends $(\vv,t)$ to $(\sfP_t(\vv),t)$ is bi-H\"older for $\UG\times \R$ equipped with the metric $d_{\UG}+d_\R$ for $d_\R$ the usual distance on $\R$. We denote the pair $(\UG,\sfP)$ by $\UG^\sfP$.
\end{definition}

If $\UG^\sfP$ is a reparameterization and $[g]\in [\G]'$, then $\gam_{[g]}$ is also a periodic orbit for $\UG^\sfP$. We denote its \emph{period} by $\msf{L}_{\sfP}[g]$. Equivalently, if $[g]$ is primitive then $\msf{L}_{\sfP}[g]$ is the minimal positive number such that $\sfP_{\msf{L}_{\sfP}[g]}(\vv)=\vv$ for any $\vv\in \gam_{[g]}$. 

For any reparameterization $\UG^\sfP$ there exists a continuous function $\k=\k_\sfP:\UG \times \R$ satisfying
\begin{equation*}
    \vg_t(\vv)=\sfP_{\k(\vv,t)}(\vv) \quad \text{ for }\vv\in \UG, t\in \R.  
\end{equation*}
Such a function $\k$ is a \emph{translation cocycle} as it satisfies
\[\k(\vv,s+t)=\k(\vg_t(\vv),s)+\k(\vv,t) \quad \text{ for }\vv\in \UG \text{ and } s,t\in \R.\]
Moreover, if $\sfP$ is H\"older, then $\vv\mapsto \k(\vv,t)$ is H\"older with exponent independent of $t$ and multiplicative constant that is bounded for $t$ in a bounded subset of $\R$. We call $\k$ the \emph{translation cocycle} of $\UG^\sfP$. Using the translation cocycle and standard arguments relying on Liv\v{s}ic's \Cref{thm:livsic}, any reparameterization determines a potential in $C(\UG).$

\begin{proposition}[cf.~{\cite[Proposition~1.11]{tholozan},\cite[Corollary~2.2.3]{sambarino.report}}]\label{prop:par->C}
If $\UG^\sfP$ is a continuous reparameterization of $\UG$, then there exists a positive continuous function $F=F_\sfP:\UG \ra \R$ such that
\begin{equation*}
\int_{\gam_{[g]}}{F}=\msf{L}_{\sfP}[g] \quad \text{ for any primitive }[g]\in [\G]'.
\end{equation*}
Moreover, if $\sfP$ is H\"older then we can choose $F$ to be H\"older.
\end{proposition}

Two reparameterizations $\UG^\sfP$ and $\UG^{\sfP'}$ are \emph{conjugate} if there exists a flow-line preserving homeomorphism $\Phi:\UG^\sfP\ra \UG^{\sfP'}$ such that $\Phi \circ \sfP_t=\sfP'_t \circ \Phi$ for all $t$. Such a conjugating homeomorphism must be of the form $\vv \mapsto \Phi(\vv)=\Phi_F(\vv)=\sfP'_{F(\vv)}(\vv)$ for some $F \in C(\UG)$, and in that case $\sfP$ and $\sfP'$ are related by the identity 
\begin{equation*}
\sfP_t(\vv)=\sfP'_{t+F(\vv)-F(\sfP_t(\vv))}(\vv).
\end{equation*}

We say that $\sfP$ and $\sfP'$ are \emph{weakly conjugate} if $\sfP'$ is the uniform limit of reparameterizations conjugate to $\sfP$. That is, there exists a sequence $(\sfP_m)_m$ of flow reparameterizations conjugate to $\sfP$ such that the maps $(\vv,t)\mapsto (\sfP_m)_t(\vv)$ converge to $(\vv,t)\mapsto \sfP'_t(\vv)$ uniformly on compact subsets of $\UG \times \R$. It turns out that the correspondence from \Cref{prop:par->C} maps weak conjugacy classes of reparameterizations to weak conjugacy classes of potentials.  

Given a reparameterization $\UG^\sfP$, its \emph{rescaling} by $\lam>0$ is the reparameterization $\UG^{\lam\sfP}$ such that $\lam\sfP_{t}(\vv)=\sfP_{\lam t}(\vv)$ for all $\vv\in \UG$. Two reparameterizations are (\emph{weakly}) \emph{equivalent} if they have (\emph{weak}) conjugate rescalings. 

A reparameterization $\UG^{\sfP}$ is \emph{reversible} if it anti-commutes with the involution $\iota$ on $\UG$ in the sense that $\sfP_t\circ \iota=\iota \circ \sfP_{-t}$ for all $t$. 

\begin{definition}\label{def:PAR}
    We define $\Par(\UG)=\Par(\UG,\msf{g})$ as the set of continuous reparameterizations of $(\UG,\vg)$, modulo weak equivalence. This is the \emph{space of reparameterizations} of Mineyev's flow space $\UG$. The equivalence class of $\UG^\sfP$ in $\Par(\UG)$ is denoted by $[\sfP]$.
    
    We let $\Par^s(\UG)$ be the subspace of $\Par(\UG)$ of points represented by reversible reparameterizations, and $\Par^h(\UG)\subset \Par(\UG)$ be the subspace of points represented by H\"older reparameterizations.
\end{definition}

It turns out that points in $\Par(\UG)$ are completely determined by the homothety classes of their periods. The next lemma follows from the fact that two reparameterizations are conjugate if and only if their corresponding translation cocycles are cohomologous \cite[Proposition~1.17]{tholozan}.

\begin{lemma}[{cf.~\cite[Corollary~1.19]{tholozan}}]\label{lemm:weak.conj}
    Two reparameterizations $\UG^\sfP$ and $\UG^{\sfP'}$ are weakly conjugate if and only if $\msf{L}_{\sfP}[g]=\msf{L}_{\sfP'}[g]$ for all primitive $[g]\in [\G]'$. \qed 
\end{lemma}

The following is one of the main results of this paper.

\begin{theorem}[Dictionary metrics-reparameterizations]\label{thm:dictionary}
    Let $\G$ be the a non-elementary torsion-free hyperbolic group and consider a Mineyev's flow space $(\UG,\msf{g})$. Then there exists an injection \begin{equation*}
        \Thet:\Par(\UG) \ra \MC({\scrH^{++}_\G})
    \end{equation*} satisfying the following:
    \begin{enumerate}
    \item If $[\sfP] \in \Par(\UG)$ is represented by a reparameterization $\UG^\sfP$, then $\Thet([\sfP])=[\psi]$ for some $\psi\in \MC(\calH^{++}_\G)$ such that 
    \[\msf{L}_{\sfP}[g]=\ell_{\p}[g] \quad \text{ for any primitive }[g]\in [\G]'.\]
    \item $\Thet(\Par^s(\UG))=\MC(\scrD_\G)$.
    \item $\Thet(\Par^h(\UG))\subset \scrH^{++}_\G$. 
    \end{enumerate}
\end{theorem}

The rest of the section is devoted to prove this result.


\subsection{Dictionary ($C(\UG) \Ra \calH_\G$)}\label{subsec:dictionaryC->H}

We begin by relating the vector spaces $\MC(\calH_\G)$ and $C(\msf{U}\G)$. Recall that for $[g]\in [\G]'$, $\frakm_{[g]}\in \calM_{\msf{g}}(\UG)$ is the $\msf{g}$-invariant probability measure supported on $\gam_{[g]}$ as in \Cref{eq:def.m_g}.

\begin{definition}\label{def:dualC-H}
    Let $F\in C(\msf{U}\G)$ and $\psi\in \MC(\calH_\G)$. We say that $F$ and $\psi$ are \emph{dual} if
\begin{equation}\label{eq:def.dualC-H}
    \int{F}{d\frakm_{[g]}}=\frac{\ell_{\psi}[g]}{\ell_{\hat d}[g]}  \quad \text{ for all } [g]\in [\G]'.
    \end{equation}
\end{definition}
Equivalently, $\p$ and $F$ are dual if $\int_{\gam_{[g]}}F=\ell_{\p}[g]$ for any primitive conjugacy class $[g]\in [\G]'$. By \cite[Corollary~3.5]{CRS}, it follows that for two dual pairs $(F,\p), (F',\p')\in C(\UG)\times \calH_\G$, we have that $\p$ and $\p'$ are roughly isometric if and only if $F$ and $F'$ are weakly cohomologous.

We can find potentials in $\calH_\G$ dual to functions $F\in C(\UG)$ satisfying the following dynamical property.

\begin{definition}\label{def:bowen}
    A continuous function $F\in C(\UG)$ has the \emph{Bowen property} if there exist $C_0,\ep_0>0$ such that for any $\vv,\vw\in \UG$ and $T>0$:
    \begin{equation*}
        d_{\UG}(\msf{g}_t(\vv),\msf{g}_t(\vw)) \leq \ep_0 \text{ for }0\leq t \leq T \ \text{ implies } \ \left|\int_{0}^{T}{(F(\msf{g}_t(\vv))-F(\msf{g}_t(\vw)))}{dt}\right|\leq C_0. 
    \end{equation*}
We let $C^B(\UG)$ denote the vector space of continuous potentials with the Bowen property.
\end{definition}

From a generic point of view, this condition is not that restrictive. By a standard argument using \Cref{prop.geodesicsclose} (see e.g. the proof of \cite[Proposition~4.3]{CLT.weak}), it follows that functions with the Bowen property are dense.

\begin{lemma}\label{lem:bowendense}
Any H\"older potential $F:(\UG,d_{\UG})\ra \R$ has the Bowen property. In particular, $C^B(\UG)$ is dense in $C(\UG)$ for the uniform topology. \qed
\end{lemma}

The following is the main result of this subsection.
Recall the subspaces $\ov\calD_\G,\calH_\G^+$ of $\calH_\G$ from \Cref{sec:HMP}.

\begin{proposition}\label{prop.dictionaryBowen}
For any $F\in C^B(\calF,\R)$ there exists $\psi_F\in \calH_\G$ dual to $F$. Moreover, we have:
\begin{enumerate}
    \item if $F\geq 0$ (resp. $F>0$) then $\psi_F\in \calH_\G^{+}$ (resp. $\psi_F\in \calH_\G^{++}$); and,
    \item if $F \circ \iota=F \geq 0$ then we can choose $\psi_F \in \ov\calD_\G$.
\end{enumerate}
\end{proposition}

Before proceeding with the proof, we note the following consequence, which also relies on \Cref{lem:convF} below.

\begin{corollary}\label{coro:F->H}
Any $F\in C(\UG)$ is dual to some $\p\in \MC(\calH_\G)$, so that if $F$ is positive then $\p\in \MC(\calH_\G^{++})$.
\end{corollary}

\begin{proof}
By \Cref{thm:mineyev}~(2), if $\psi_F\in \MC(\calH_\G)$ is dual to $F$ then $\psi_F+a\hat d$ is dual $F+a \mathsf{1}_{\UG}$ for all $a\in \R$. Therefore, without loss of generality we can assume that $F>0$. By \Cref{lem:bowendense}, let $(F_m)_m$ be a sequence in $C^B(\UG)$ converging uniformly to $F$. We can assume that each $F_m$ is positive and dual to $\p_m\in \calH_\G^{++}$ by \Cref{prop.dictionaryBowen}. Then the lemma follows by applying \Cref{lem:convF} below. 
\end{proof}

\begin{lemma}\label{lem:convF}
    Let $(F_m)_m$ be a sequence in $C(\UG)$ uniformly converging to $F$. Suppose that each $F_m$ is positive and dual to $\p_m\in \MC(\calH_\G^{++})$. Then $([\p_m])_m$ converges in $\MC(\scrH_\G^{++})$ to some $[\p]$ so that $\p\in \MC(\calH_\G^{++})$ is dual to $F$.
\end{lemma}

\begin{proof}
    Let $c>0$ be such that $c^{-1}\leq F_m\leq c$ for all $m$, so that $c^{-1}\leq F\leq c$ as well. Then for all $m,k$ and $\frakm\in \calM_{\vg}(\UG)$ we have
    \[\left|\log\left( \int{F_m}d\frakm\right)-\log\left( \int{F_k}d\frakm\right)\right|\leq c^{-1}\left|\int{F_m}{d\frakm}-\int{F_k}{\frakm}\right|\leq c\|F_m-F_k\|.\]
    In particular, after applying this to $\frakm=\frakm_{[g]}$ for each $[g]\in [\G]$ and the fact that each $\p_m$ is dual to $F_m$, we obtain
    \begin{align*}
        \Del([\p_m],[\p_k])& =\log\left(\sup_{[g]\in [\G]'}\frac{\ell_{\p_m}[g]}{\ell_{\p_k}[g]}\cdot \sup_{[g]\in [\G]'}\frac{\ell_{\p_k}[g]}{\ell_{\p_m}[g]}\right)\\
        & \leq \sup_{[g]\in [\G]'}\left(\log\ell_{\p_m}[g]-\log\ell_{\p_k}[g]\right)^2\\
        & \leq \sup_{[g]\in [\G]'}\left|\log\left( \int{F_m}d\frakm_{[g]}\right)-\log\left( \int{F_k}d\frakm_{[g]}\right)\right|^2\leq c^{2}\|F_m-F_k\|^2.
    \end{align*}
As this last term tends to zero as $m,k\to \infty$, the sequence $([\p_m])_m$ is Cauchy, and hence convergent to some $[\p]$ in $\MC(\scrH_\G^{++})$. Since $c^{-1}\leq F_m\leq c$ for each $m$, after rescaling we can find a representative $\p$ of $[\p]$ that is dual to $F$. We leave the details to the reader.
\end{proof}


We continue with the proof of \Cref{prop.dictionaryBowen}, for which we require the following slight upgrade to the Bowen property.

\begin{lemma}\label{lem.liftBowen}
    Given $F\in C^B(\UG)$, let $\wt F:\wt\UG \ra \R$ be the $\G$-invariant lift of $F$. Then for any $A>0$ there exists $B>0$ such that if $\wt\vv, \wt\vw \in \wt\UG$ and $T,T'>0$ satisfy 
    \[ \max \left(d_{\wt\UG}(\wt\vv, \wt\vw), d_{\wt\UG}(\wt{\msf{g}}_T(\wt\vv), \wt{\msf{g}}_{T'}(\wt\vw)) \right)\leq A \]
    then \[\left|\int_{0}^{T}{\wt F(\wt{\msf{g}}_t(\wt \vv))}{dt}-\int_{0}^{T'}{\wt F(\wt{\msf{g}}_t(\wt \vw))}{dt}\right|\leq B.\]
\end{lemma}

\begin{proof}
    Let $\ep,C>0$ be the Bowen constants for $F$, and $T_1>0$ the constants given by \Cref{prop.geodesicsclose2} for $\del=3A$ and $\ep_1=\ep$.
    Assume that $T\leq T'$, so from \Cref{thm:mineyev}~(3) we get $T'\leq T+2A$. 
    
    Suppose first that $T\leq 4T_1$.
    Then 
    \begin{align*}
        \left|\int_{0}^{T}{\wt F(\wt{\msf{g}}_t(\wt \vv))}{dt}-\int_{0}^{T'}{\wt F(\wt{\msf{g}}_t(\wt \vw))}{dt}\right|\leq (T+T')\|\wt F\|\leq 2(8T_1+A)\|F\|.
    \end{align*}
   If $T>4T_1$, note that \[ \max \left(d_{\wt\UG}(\wt\vv, \wt\vw), d_{\wt\UG}(\wt{\msf{g}}_T(\wt\vv), \wt{\msf{g}}_{T}(\wt\vw)) \right)\leq 3A=\del, \]
   and from \Cref{prop.geodesicsclose2} we find $t_0\in (-T_1,T_1)$ such that \[d_{\wt\UG}(\wt\vg_t(\wt\vv),\wt\vg_{t+t_0}(\wt\vw))\leq \ep \quad \text{  for all } T_1\leq t\leq T-T_1.\] 
   In consequence, from the Bowen property we get the bounds
     \begin{align*}
        & \left|\int_{0}^{T}{\wt F(\wt{\msf{g}}_t(\wt \vv))}{dt}-\int_{0}^{T'}{\wt F(\wt{\msf{g}}_t(\wt \vw))}{dt}\right| \\
        & \leq \left|\int_{T_1}^{T-T_1}{\wt F(\wt{\msf{g}}_t(\wt \vv))}{dt}-\int_{T_1}^{T-T_1}{\wt F(\wt{\msf{g}}_{t+t_0}(\wt \vw))}{dt}\right|+(2T_1+(|T'-T|+2T_1+2|t_0|))\|\wt F\|\\
        & \leq C +2(3T_1+A)\|F\|.
    \end{align*}
    The lemma then follows by setting $B=C+2(3T_1+A)\|F\|$.
\end{proof}

\begin{proof}[Proof of \Cref{prop.dictionaryBowen}]
Our argument is similar to that of \cite[Section~3.1]{connell-muchnik.gibbs} (see also \cite[Lemma~4.2.7]{dilsavor}). As in the proof of \Cref{coro:F->H}, we can assume that $F \geq 0$. Let $\wt F: \wt\UG \ra \R$ be the $\G$-equivariant lift of $F$ and $K\subset \wt\UG$ a compact subset from \Cref{lem.nicecompact}, with diameter $D$ for the metric $d_{\wt\UG}$. After possibly enlarging $K$, we can assume that $\G\cdot K=\wt\UG$.
    
    We define $\p=\p_F: \G \times \G \ra \R$ according to
    \[\p(g,h):=\sup \left\{\int_{0}^{T}{\wt F(\wt{\msf{g}}_t (\wt\vv))dt} : \wt\vv \in gK, \wt{\msf{g}}_T(\wt\vv)\in hK, T \geq 0\right\}.\]

    We first note that $\p$ is well-defined and finite. Given $g,h\in \G$, by \Cref{lem.nicecompact} we know that we can find $\wt\vv\in gK$ and $T>0$ such that $\wt{\msf{g}}_T(\wt\vv)\in hK$. For any other $\wt\vv'\in gK$ and $T'>0$ satisfying the same property, \Cref{lem.liftBowen} implies that 
    \[\left|\int_{0}^{T}{\wt F(\wt{\msf{g}}_t(\wt\vv))}{dt}-\int_{0}^{T'}{\wt F(\wt{\msf{g}}_t(\wt\vv'))}{dt}\right|\leq B\]
    for $B=B(D)$, and hence $\p(g,h)$ is finite. 

    Clearly $\p$ is nonnegative, and the $\G$-invariance of $\wt F$ and $\G$-equivariance of $\wt{\msf{g}}$ imply that $\p$ is $\G$-invariant. To show that $\p\in \calH_\G$, 
    we fix a word metric $d\in \calD_\G$. By \Cref{prop:criterionHMP} it is enough to show the following: 
    
    \textbf{Claim:} There exists $k_1>0$ such that if $g,h\in \G$ and $g$ lies on a $d$-geodesic $\gam$ from $o$ to $gh$, then
    \begin{equation}\label{eq.BBTd'}
        |\p(o,gh)-\p(o,g)-\p(o,h)| \leq k_1.
\end{equation}
To show this claim, let $\wt\vv_g,\wt\vv_h,\wt\vv_{gh}\in K$ and $T_g,T_h,T_{gh}>0$ satisfy
\begin{equation}\label{eq.conditionu_ghgh}
    \wt{\msf{g}}_{T_g}(\wt\vv_g)\in gK, \qquad  \wt{\msf{g}}_{T_h}(\wt\vv_h)\in hK, \qquad  \wt{\msf{g}}_{T_{gh}}(\wt\vv_{gh})\in ghK,
\end{equation}
which exist by our choice of $K$. For any fixed $\wt\vw_0\in K$, by \Cref{thm:mineyev}~(2) we get that the orbit $\gam \cdot \wt\vw_0$ is within bounded distance from a uniform quasigeodesic from $\wt\vw_0$ to $gh \wt\vw_0$. But $d_{\wt\UG}(\wt\vw_0,\wt\vv_{gh})\leq D$ and $d_{\wt\UG}(gh\wt\vw_0,\wt{\msf{g}}_{T_{gh}}(\wt\vv_{gh}))\leq D$, and hence \Cref{thm:mineyev}~(2) and the Morse lemma give us a constant $R>0$ independent of $g,h$, and a time $0\leq T'\leq T_{gh}$ such that 
$d_{\wt\UG}(g\wt\vw_0,\wt{\msf{g}}_{T'}(\wt\vv_{gh}))\leq R$. This implies that 
\[\max\left(d_{\wt\UG}(\wt\vv_g,\wt\vv_{gh}), d_{\wt\UG}(\wt{\msf{g}}_{T_g}(\wt\vv_g),\wt{\msf{g}}_{T'}(\wt\vv_{gh}))\right)\leq D+R\]
and 
\[\max\left(d_{\wt\UG}(g\wt\vv_h,\wt{\msf{g}}_{T'}(\wt\vv_{gh})), d_{\wt\UG}(\wt{\msf{g}}_{T_h}(g\wt\vv_h),\wt{\msf{g}}_{T_{gh}}(\wt\vv_{gh}))\right)\leq D+R.\]
Hence, from \Cref{lem.liftBowen} we get
\[ \left| \int_{0}^{T_g}{\wt F(\wt{\msf{g}}_t(\wt\vv_g))}{dt}-\int_{0}^{T'}{\wt F(\wt{\msf{g}}_t(\wt\vv_{gh}))}{dt} \right| \leq B\]
and 
\begin{align*} &\left| \int_{0}^{T_h}{\wt F(\wt{\msf{g}}_t(\wt\vv_h))}{dt}-\int_{T'}^{T_{gh}}{\wt F(\wt{\msf{g}}_t(\wt\vv_{gh}))}{dt} \right| \\&=\left| \int_{0}^{T_h}{\wt F(\wt{\msf{g}}_t(g\wt\vv_h))}{dt}-\int_{0}^{T_{gh}-T'}{\wt F(\wt{\msf{g}}_t(\wt{\msf{g}}_{T'}(\wt\vv_{gh})))}{dt} \right| \leq B\end{align*}
for $B=B(D+R)$. Adding these two inequalities we get
\[\left|\int_{0}^{T_{gh}}{\wt F(\wt{\msf{g}}_t(\wt\vv_{gh}))}{dt} -\int_{0}^{T_g}{\wt F(\wt{\msf{g}}_t(\wt\vv_{g}))}{dt}-\int_{0}^{T_{h}}{\wt F(\wt{\msf{g}}_t(\wt\vv_{h}))}{dt}\right| \leq 2B,\]
and after taking supremum over all $\wt\vv_g,\wt\vv_h,\wt\vv_{gh}$ and $T_g,T_h,T_{gh}$ satisfying \Cref{eq.conditionu_ghgh}, we deduce \Cref{eq.BBTd'} for $k_1=2B(D+R)$. This concludes the proof of the claim and hence implies $\p \in \calH_\G$.

Next, we prove that $\psi$ is dual to $F$. Given $g\in \G$ we consider the point
$\wt\vv=(g^{-},g^{+},0)\in \wt\UG$ and take $h\in \G$ such that $ \wt\vv\in hK$. Given $m>0$ we consider $\wt\vv_m\in K$ and $T_m>0$ such that $\wt{\msf{g}}_{T_m}(\wt\vv_m)\in g^{m}K$. By definition of the action of $\G$ on $\wt\UG$ we have $g^m \tilde v=\wt{\msf{g}}_{m\ell_{\hat d}[g]}(\wt \vv) \in g^mhK$, and hence 
\[\max \{d_{\wt\UG}(\wt\vv,\wt\vv_m),d_{\wt\UG}(\wt{\msf{g}}_{m\ell_{\hat d}[g]}(\wt\vv),\wt{\msf{g}}_{T_m}(\wt\vv_m))\} \leq A:=\max_{u,u'\in K}{d_{\wt\UG}(\wt\vvu,h\wt\vvu')}.\]
Then \Cref{lem.liftBowen} implies that 
\[\left| \int_{0}^{m\ell_{\hat d}[g]}{\wt F(\wt{\msf{g}}_t(\wt\vv))}{dt}  -\int_{0}^{T_m}{\wt F(\wt{\msf{g}}_t(\wt\vv_m))}{dt} \right|\leq B(A),\]
and after taking supremum over all possible $\wt\vv_m$ and $T_m$ (but fixing $m$) we deduce that 
\[\left| \int_{0}^{m\ell_{\hat d}[g]}{\wt F(\wt{\msf{g}}_t(\wt\vv))}{dt}- \p(o,g^m)\right|\leq B(A)\]
for all $m$. If $\vv$ is the image of $\wt\vv$ under the projection $\pi:\wt\UG \ra \UG$, then
\begin{align*}
    \int{F}{d\frakm_{[g]}}& =\frac{1}{\ell_{\hat d}[g]}\int_{0}^{\ell_{\hat d}[g]}{F(\msf{g}_t(\vv))}{dt}\\
    & =\lim_{m\to \infty}{\frac{1}{m\ell_{\hat d}[g]}\int_{0}^{m\ell_{\hat d}[g]}{F(\msf{g}_t(\vv))}{dt}} \\
    & = \frac{1}{\ell_{\hat d}[g]}\lim_{m\to \infty}{\frac{1}{m}\int_{0}^{m\ell_{\hat d}[g]}{\wt F(\wt{\msf{g}}_t(\wt\vv))}{dt}} \\
    & = \frac{1}{\ell_{\hat d}[g]}\lim_{m\to \infty}{\frac{\p(o,g^m)}{m}}\\
    & =\frac{\ell_{\psi}[g]}{\ell_{\hat d}[g]}.
\end{align*}
This proves that $\p$ and $F$ are dual. 

We already proved the first claim of item (1) by showing that we can choose $\p_F\geq 0$ (hence in $\calH_\G^+$) for $F\geq 0$. For the second item, assume that $F>0$. Then continuity of $F$ and compactness of $\UG$ gives us $F\geq c$ for some $c>0$. This implies $\ell_{\psi_F}[g]\geq c \ell_{\hat d}[g]$ for all $g\in \G$, and hence $\p_F\in \calH_\G^{++}$.
This completes the proof of item (1), and item (2) follows after replacing $\psi_F$ by $\frac{1}{2}(\psi_F+\psi_{F \circ \iota})$ and using \cite[Lemma~3.6]{CRS}. 
\end{proof}


\subsection{Dictionary ($\MC(\calD_\G)\Ra C(\UG)$)}\label{subsec:DCoc}

In this subsection we generalize \Cref{ex:cocyclestrongly} and relate positive functions in $C(\UG)$ to metrics in $\MC(\calD_\G)$. This provides a partial converse to \Cref{prop.dictionaryBowen} in the symmetric case. 

\begin{proposition}\label{prop:D_G->C}
Any $d\in \MC(\calD_\G)$ is dual to some continuous function $F\in C(\UG)$. Such an $F$ can be chosen to be positive and $\iota$-invariant: $F\circ \iota=F$. 
\end{proposition}

\begin{proof}
    Without loss of generality assume that $d$ has exponential growth rate 1. Then by \Cref{mainthm.greendense}, we can find a sequence $(d_m)$ of Green metrics in $\calD_\G$ such that $[d_m]\to [d]$ in $\MC(\scrD_\G)$. Each $d_m$ is dual to a H\"older cocycle $\vc_m=\vc_{d_m}$ as in \Cref{ex:cocyclestrongly}, and hence \Cref{prop:ledrappier} gives us a H\"older continuous function $F_m:\UG \ra \R$ dual to $d_m$. 

    We claim that after taking a subsequence and changing each $F_m$ by a cohomologous H\"older continuous function $F_m'$, we can guarantee that $F_m'$ uniformly converges to a positive continuous function $F$ as $m$ tends to infinity. Then after possibly replacing $F$ by $(F+F\circ \iota)/2$, \Cref{lem:convF} tells us that $F$ is positive, $\iota$-invariant, and dual to $d$. 

    To prove the claim, let $c>1$ be such that $c^{-1}\leq \ell_{\p_m}[g]/\ell_{\hat d}[g]\leq c$ for all $m$ and $[g]\in [\G]'$, which exists since $[d_m]\to [d]$ and each $d_m$ has exponential growth rate 1. From this, and using \Cref{eq:def.dualC-H} for all $m,k$ and the fact that measures supported on periodic orbits are dense in $\calM_{\vg}(\UG)$, we have 
    \begin{equation}
    \begin{aligned}
\sup_{\frakm\in \calM_\vg(\UG)}\left|\int{F_m}{d\frakm}-\int{F_k}{d\frakm}\right| &=\sup_{[g]\in [\G]'}\frac{|\ell_{\p_m}[g]-\ell_{\p_k}[g]|}{\ell_{\hat d}[g]}\\
            &\leq c\sup_{[g]\in [\G]'}\left| \log \left(\frac{\ell_{\p_m}[g]}{\ell_{\p_k}[g]}\right)\right|\\
        &\leq c\Del([d_m],[d_k]). 
    \end{aligned}
    \label{eq:ineq.D}
    \end{equation}

Since $[d_m]_m\to [d]$, we extract a subsequence and reindex so that $\Del([d_{m+1}],[d_{m}])\leq c^{-1}/2^{m+2}$ for all $m\geq 1$. Inductively, we define a sequence $F_1',\dots,F_m',\dots$ in $C(\UG)$ such that
\begin{itemize}
    \item each $F'_m$ is cohomologous to $F_m$;
    \item $F'_1\geq 3c^{-1}/4$; and,
    \item $\|F'_{m+1}-F'_m\|\leq c^{-1}/2^{m+1}$ for all $m\geq 1$.
\end{itemize} 

First, we note that $\inf_{\frakm\in \calM_{\vg}(\UG)}{\int{F}d\frakm}\geq c^{-1}$. Then we apply the approximate Ma\~{n}\'e lemma (see e.g.~\cite[Lemma~3.1]{pavez-molina}), to find $F'_1\in C(\UG)$ cohomologous to $F_1$ and such that $F_1\geq 3c^{-1}/4$. Assuming we have found $F_1',\dots,F_m'$ satisfying the above, we define $G_m=F_{m+1}-F_{m}'$. By \Cref{eq:ineq.D} and our choice of subsequence $(F_m)_m$, we have that $\sup_{\frakm\in \calM_{\vg}{\UG}}{\int |G_m|d\frakm}\leq c^{-1}/2^{m+2}$. Then again applying the approximate Ma\~{n}\'e lemma, we find $G'_m$ cohomologous to $G_m$ and such that $\|G'_m\|\leq c^{-1}/2^{m+1}$. We set $F'_{m+1}:=F_m'+G_m'$, which by construction also satisfies the desired properties.  

To end the proposition we check that $(F_m')_m$ is Cauchy and converges to a positive continuous function. Indeed, for $k,m\geq 1$ we have
\begin{align*}
    \|F'_{m+k+1}-F'_{m}\|\leq \|F'_{m+k+1}-F'_{m+k}\|+\cdots +\|F'_{m+1}-F_m'\|\leq c^{-1}/2^{m},
\end{align*}
which tends to 0 as $m\to \infty$. Hence $(F'_m)_m$ is Cauchy, and its limit $F$ must be positive since for all $m\geq 1$ we have
\[F_{m+1}'=F'_1+\cdots+(F_{m+1}'-F_m')\geq 3c^{-1}/4-c^{-1}(1/2^3+1/2^4+\cdots)\geq c^{-1}/4. \qedhere\]   
\end{proof}


\subsection{Proof of \Cref{thm:dictionary}}

We have all the ingredients to prove \Cref{thm:dictionary}.

\begin{proof}
    We define $\Thet :\Par(\UG)\ra \MC(\scrH_\G^{++})$ as follows. Given a reparameterization $\sfP$ of $\vg$, let $F_{\sfP}\in C(\UG)$ be the positive continuous function given by \Cref{prop:par->C}. Then we apply \Cref{coro:F->H} to find $\p_\sfP=\psi_{F_{\sfP}}\in \MC(\calH_\G^{++})$ dual to $F_\sfP$. We define $\Thet([\sfP])=[\p_\sfP]$, which is well-defined and injective by \Cref{lemm:weak.conj} and \Cref{cor:DeltaformulaMC}. The construction of $\Thet$ using \Cref{prop:par->C} and  \Cref{coro:F->H} implies (1).

    To prove (2), suppose first that $\sfP$ is a reversible reparameterization. Then it follows that $\msf{L}_{\sfP}[g]=\msf{L}_{\sfP}[g^{-1}]$ for all primitive $[g]\in [\G]'$. As a consequence, if $\Theta([\sfP])=[\p]$, then $[\p]=[(\p+\hat\p)/2]$ and $(\p+\hat\p)/2\in \MC(\calD_\G)$. This gives $\Thet(\Par^s(\UG))\subset \MC(\calD_\G)$. To prove the reverse inclusion, we let $d\in \MC(\calD_\G)$ and use \Cref{prop:D_G->C} to find a continuous, positive and $\iota$-invariant function $F\in C(\UG)$ that is dual to $d$. From this we construct the reparameterization $\sfP$ of $\msf{g}$ according to 
    \[\vg_{t}(\vv)=\sfP_{\k(\vv,t)}(\vv),\]
    for $\k(\vv,t)$ equal to $\int_0^t{F(\vg_s(\vv))ds}$ if $t\geq 0$ and equal to $-\int_{-t}^0{F(\vg_s(\vv))ds}$ if $t<0$.
    This reparameterization is well-defined since $F$ is positive and is reversible since $F$ is $\iota$-invariant. We conclude the proof of (2) by checking that $\Thet([\sfP])=[d]$: for $[g]\in [\G]'$ primitive and $\vv\in \gam_{[g]}$ we have
    \begin{align*}
    \msf{L}_{\sfP}[g]=\k(\vv,\ell_{\hat d}[d])=\int_{0}^{\ell_{\hat d}[g]}{F(\vg_s(\vv))ds}=\int_{\gam_{[g]}}F=\ell_{d}[g].    
    \end{align*}

    Finally, item (3) follows from the moreover statement in \Cref{prop:par->C}, \Cref{prop.dictionaryBowen}, and \Cref{lem:bowendense}.
\end{proof}


\section{Mean distortion from length functions}
\label{sec:meandistortion_length}
In this section we assume $\G$ is a non-elementary torsion-free hyperbolic group. 
Here we prove the following theorem, which describes the mean distortion $\tau(\varphi/\p)$ for $\varphi\in \calH_\G$ and $\p\in \calH_\G^{++}$ (see \Cref{eq:meandistortion}) as the length function for $\p$ evaluated at the appropriately normalized BMS current for $\varphi$. This result, and in particular its \Cref{cor:dualintegral}, is used in the proof of \Cref{thmmain.H_Gnotcomplete} in the next section.

\begin{theorem}\label{thm.tauvsell}
    Let $\varphi \in \calH_\G$ and $\p\in \calH^{++}_\G$. If $\Lam_\p\in \calC_\G$ is a BMS current for $\p$ then 
    \begin{equation*}
\tau(\varphi/\p)=\frac{\ell_{\varphi}(\Lam_\p)}{\ell_\p(\Lam_\p)}.
    \end{equation*}
\end{theorem}

Throughout the section we fix $\varphi \in \calH_\G$ and $\p\in \calH^{++}_\G$. We also fix a Mineyev's flow space $\UG$ for $\G$ associated to the strongly hyperbolic metric $\hat d$ in $\calD_\G$, as in \Cref{sec:mineyevflow}. We keep the notation from that section, so that $\pi:\wt\UG \ra \UG$ denotes the associated covering map, and $\msf{g}=(\msf{g}_t)_{t\in \R}$ and $\wt{\msf{g}}=(\wt{\msf{g}}_t)_{t\in \R}$ denote the flows on $\UG$ and $\wt\UG$ respectively. We also let $d_{\UG}$ and $d_{\wt\UG}$ denote the corresponding metrics on $\UG$ and $\wt\UG$.

The core of the proof is to describe $\ell_{\varphi}$ as a drift function for an almost cocycle defined on $\UG$. This is inspired by the construction of the continuous extension of $\ell_d$ to $\calC_\G$ when $d$ is a word metric by Erlandsson--Parlier--Souto \cite[Theorem~1.5]{EPS}. Incidentally, this reproves the fact that $\ell_{\varphi}$ admits a (unique) continuous extension to $\calC_\G$, recovering the result of Kapovich and Martinez-Granado \cite{kapovich-martinezgranado}.

To begin the description of $\ell_{\varphi}$, we consider a precompact Borel fundamental domain $\calB \subset \wt\UG$ for the action of $\G$, and for $\vv\in \UG$ we let $\wh\vv$ denote its unique lift in $\calB$. For each $t\in \R$ and $\vv\in \UG$, we let $g(\vv,t)=g_\calB(\vv,t)\in \G$ be the unique group element such that $\wt{\msf{g}}_t(\wh \vv)\in g_\calB(\vv,t)\calB$. It follows from its definition that $g_\calB$ satisfies the cocycle relation
\begin{equation}\label{eq.cocycle}
g_\calB(\vv,t+s)=g_\calB(\vv,t)g_\calB(\msf{g}_t(\vv),s) \quad \text{ for }\vv\in \UG \text{ and }s,t>0.
\end{equation}

\begin{definition}\label{def:Psi_B}
    We define $\P=\P_\calB:\UG\times \R \ra \R$ according to $\P(\vv,t)=\varphi(o,g_\calB(\vv,t))$.
\end{definition}

Since $\calB$ is Borel, $\P$ is a Borel function, and bounded backtracking of $\varphi$ implies that $\P$ is an almost cocycle: 
\begin{lemma}\label{lem.almostcocycle}
    There exists $C>0$ depending only on $\calB$ such that for $\P=\P_\calB$ we have
    \[|\P(\vv,t+s)-\P(\vv,t)-\P(\msf{g}_t(\vv),s)|\leq C \quad \text{ for }\vv\in \UG\text{ and }s,t>0.\]
   
\end{lemma}

\begin{proof}
    Let $C_1>0$ be the diameter of $\calB$ in the metric $d_{\wt\UG}$. Fixed a point $\vp\in \UG$, the metric $d^{\wh\vp}_{\wt \UG}$ on $\G$ given by $(g,h)\mapsto d_{\wt\UG}(g\wh\vp,h\wh\vp)$ belongs to $\calD_\G$ by \Cref{thm:mineyev}~(3). Also, let $R>0$ be such that 
    \begin{equation}\label{eq.BBTd_F}
   |\gpr{g}{h}^{\varphi}_{o}|\leq R \gpr{g}{h}_o^{d_{\wt\UG}^{\wh\vp}}+R      
    \end{equation}for all $g,h\in \G$. 

Now let $\vv\in \UG$, and for $s,t>0$ note that the distances $$d_{\wt \UG}(\wh\vp,\wh \vv),\qquad d_{\wt \UG}(g_\calB(\vv,t)\wh\vp,\wt{\msf{g}}_t(\wh \vv)), \qquad  d_{\wt \UG}(g_\calB(\vv,t+s)\wh\vp, \wt{\msf{g}}_{t+s}(\wh \vv))$$ are bounded above by $C_1$. Since $r \mapsto \wt{\msf{g}}_{r}(\wh \vv)$ is a geodesic in $(\wt\UG,d_{\wt\UG})$ by \Cref{thm:mineyev}~(3), \Cref{eq.BBTd_F} implies that
\[|\varphi(o,g_\calB(\vv,t))+\varphi(g_\calB(\vv,t),g_\calB(\vv,t+s))-\varphi(o,g_\calB(\vv,t+s))|\leq C\]
for $C=3C_1R+2R$. The conclusion then follows from the cocycle relation in \Cref{eq.cocycle} and the definition of $\P$. 
\end{proof}

\begin{definition}
    Given a measure $\frakm\in \calM_{\msf{g}}(\UG)$, we define
    \[L_\calB(\frakm)=L_{\varphi,\calB}(\frakm):=\lim_{t\to +\infty}\frac{1}{t}\int{\P_\calB(\vv,t)}{d\frakm(\vv)}.\]
\end{definition}
    The limit above exists by \Cref{lem.almostcocycle} and Kingman's subadditive ergodic theorem. In fact, from the almost cocycle relation we have the following stronger property. If $\frakm\in \calM_{\msf{g}}(\UG)$, then $\msf{g}$-invariance implies that 
    \[\left|\int{\P_\calB(\vv,s+t)}{d\frakm}-\int{\P_\calB(\vv,t)}{d\frakm}-\int{\P_\calB(\vv,s)}{d\frakm}\right|\leq C\frakm(\UG)\] for all $s,t>0$,
where $C>0$ is the constant from \Cref{lem.almostcocycle}. In particular, Fekete's lemma implies that
    \begin{equation}\label{eq.L_B}
    \left|L_\calB(\frakm)-\frac{1}{n}\int\P_\calB(\vv,n)d\frakm \right|\leq \frac{C\frakm(\UG)}{n}
    \end{equation}
    for all natural numbers $n\geq 1$.

In addition, the quantity $L_{\calB}(\frakm)$ is independent of the precompact Borel fundamental domain $\calB$. To show this, let $\calB'$ be another precompact Borel fundamental domains, so that there exists a finite set $A\subset \G$ such that $\calB\subset A\calB'$ and $\calB'\subset A\calB$. Then for all $\vv\in \UG$ and $t\in \R$ we have $g_{\calB'}(\vv,t)=ag_{\calB}(\vv,t)a'$ for $a,a'\in A$ (depending on $t$ and $\vv$). From \cite[Lemma~3.3(2)]{CRS}, this implies that 
$|\P_{\calB'}(\vv,t)-\P_{\calB}(\vv,t)|\leq C'$, for $C'$ independent of $\vv$ and $t$, which implies that $L_{\calB'}(\frakm)=L_{\calB}(\frakm)$ for all $\frakm\in \calM_{\msf{g}}(\UG)$.
    
\begin{definition}
    Let $L=L_{\varphi}:\calM_{\msf{g}}(\UG)\ra \R$ be the function such that $L(\frakm)=L_\calB(\frakm)$ for any precompact Borel fundamental domain $\calB\subset \wt\UG$. 
\end{definition}

Our next goal is to prove the continuity of $L$. To do this, we fix $\frakm\in \calM_{\msf{g}}(\UG)$, for which we want to construct the appropriate fundamental domain $\calB$ so that $L_{\calB}$ is continuous at $\frakm$. This amounts to understanding the discontinuity points of $\P_\calB$. 

To find the convenient domain $\calB$ we require the notion of \emph{transversals} from \cite[Section~3.2]{EPS}. For the rest of the section we let $\Phi:\wt\UG \times \R \ra\wt\UG$ denote the flow map $\Phi(\wt\vv,t)=\wt{\msf{g}}_t(\wt\vv)$.

\begin{definition}
   A \emph{transversal} is a compact set $\tau\subset \wt\UG$ satisfying:
    \begin{itemize}
        \item The restriction $\Phi:\tau \times [-\ep,\ep] \ra \wt\UG$ is an embedding for some $\ep>0$.
        \item The restriction of the covering map $\wt\UG \ra \UG$ to $\tau$ is an embedding.
        \item For every $\wt\vv\in \wt\UG$ there exist $t>0$ and $g\in \G$ with $\wt{\msf{g}}_t(\wt\vv)\in g\tau$. 
    \end{itemize} 
\end{definition}

As discussed in \cite[Section~3.2]{EPS}, transversals always exist. For such a transversal $\tau$, its \emph{interior} (denoted $\dot\tau$) is defined as the set of all $\wt\vw \in \tau$ such that, for all $\ep>0$ sufficiently small, $\wt\vw$ is an interior point of $\Phi(\tau \times [-\ep,\ep])$. Non-interior points of $\tau$ are called \emph{boundary points}, and the set of all such points is denoted by $\partial\tau$. 

The \emph{first return time} of $\tau$ is the function $T_\tau:\tau \ra \R^+$ such that $\wt{\msf{g}}_{T(\wt\vv)}(\wt\vv)\in g\tau$ for some $g\in \G$ (necessarily non-trivial) and $\wt{\msf{g}}_{t}(\wt\vv)\notin  \G\tau$ for any $0<t<T_\tau(\wt\vv)$. It follows that $T_\tau$ is Borel, and in fact continuous inside the set $T_\tau^{-1}(\G\dot \tau)$. We also set $R_\tau(\wt\vv):=\wt{\msf{g}}_{T_\tau(\wt\vv)}(\wt\vv)$ for $\wt\vv\in \tau$.

Given a transversal $\tau$ and $\ep>0$ small enough, $\tau_\ep:=\wt{\msf{g}}_\ep(\tau)$ is also a transversal and $\partial \tau_\ep=\wt {\msf{g}}_\ep(\partial\tau)$. Moreover, if we choose $\epsilon$ small enough so that the image of $\Phi(\tau\times [-\ep,\ep])$ under  $\pi:\wt\UG \ra \UG$ is an embedding, we have that \begin{equation}\label{eq.tauemptyintersection}
    (\tau_{\ep_1} \cup R_{\tau_{\ep_1}}(\tau_{\ep_1}))\cap (\tau_{\ep_2} \cup R_{\tau_{\ep_2}}(\tau_{\ep_2})) =\emptyset \quad \text{ for }\ep_1\neq \ep_2 \text{ in }[0,\ep].
\end{equation}

The advantage of using transversals is justified by the next definition.

\begin{definition}
 For a transversal $\tau$ we define
\[\calB_\tau:=\{\wt{\msf{g}}_t(\wt\vv):\wt\vv\in \tau,0\leq t<T_\tau(\wt\vv)\}\subset \wt\UG.\]   
\end{definition}

The definition of transversal easily implies:
\begin{lemma}
   For any transversal $\tau$, the set $\calB_\tau$ is a Borel fundamental domain for the action of $\G$ on $\wt\UG$.
\end{lemma}

In addition, the interior of $\calB_\tau$ in $\wt\UG$ contains the points of form $\wt{\msf{g}}_t(\wt\vv)$ for $\wt\vv$ in $T_{\tau}^{-1}(\G\dot\tau)$ and $0<t<T_\tau(\wt\vv)$. This is equivalent to saying that the boundary $\partial\calB_\tau$ is contained in the union
\[\tau \cup R_\tau(\tau)\cup \{\wt{\msf{g}}_t(\wt\vv):\wt\vv\in T^{-1}_\tau(\G\cdot\partial\tau),t\in \R\}.\]

For our fixed measure $\frakm\in \calM_{\msf{g}}(\UG)$, we let $\wt\frakm$ denote its $\G$-invariant lift to $\wt\UG$. The conclusion of \cite[Lemma~3.1]{EPS} implies that there exists a transversal $\tau$ such that $\wt\frakm(\G\cdot\{\wt{\msf{g}}_t(\wt\vv):\wt\vv\in \partial\tau,t\in \R\})=0$. 
For such a $\tau$, \Cref{eq.tauemptyintersection} implies that the (uncountably many) sets $(\tau_{\ep} \cup R_{\tau_{\ep}}(\tau_{\ep}))$ are pairwise disjoint among all $\ep\geq 0$ small enough. Therefore, we can find such $\ep$ satisfying $\wt\frakm((\tau_{\ep} \cup R_{\tau_{\ep}}(\tau_{\ep})))=0$. Since $$\{\wt{\msf{g}}_t(\wt\vv):\wt\vv\in T^{-1}_\tau(\G\cdot\partial\tau),t\in \R\} \subset \G\cdot\{\wt{\msf{g}}_t(\wt\vv):\wt\vv\in \partial\tau,t\in \R\},$$
from this discussion we deduce the following lemma.

\begin{lemma}\label{lem.0measure}
    There exists a transversal $\tau$ such that $\frakm(\pi(\partial \calB_\tau))=0$.
\end{lemma}

Our efforts in the construction of $\calB_\tau$ are justified by the next lemma.

\begin{lemma}\label{lem.continuityfrakm}
 Let $\calB$ be a precompact Borel fundamental domain for $\wt\UG$. If $\frakm(\pi(\partial \calB))=0$ then $L_\calB$ is continuous at $\frakm$. 
\end{lemma}

\begin{proof}
First, we note that for each $t$, the function $\vv \mapsto g_\calB(\vv,t)$ is continuous (in fact, locally constant) on the open set $\UG \bs \msf{g}_{-t}(\pi(\partial \calB))$. Hence, $\msf{g}$-invariance of $\frakm$ and the assumption $\frakm(\pi(\partial \calB))=0$ imply that 
\[\frakn \mapsto \int{\P_\calB(\vv,t)}{d\frakn(\vv)}\]
is continuous at $\frakm$. This fact combined with \Cref{eq.L_B} gives us the continuity of $L_\calB$ at $\frakm$.
\end{proof}

As a consequence of \Cref{lem.0measure} and \Cref{lem.continuityfrakm}, we deduce the following.

\begin{corollary}\label{cor.Lcontinuous}
    $L:\calM_{\msf{g}}(\UG)\ra \R$ is continuous.
\end{corollary}

Our next goal is to show that $L$ is a continuous extension of $\ell_{\varphi}$.

\begin{lemma}\label{lem.Lrational}
    Let $[g]\in [\G]'$ be a non-trivial conjugacy class and $\frakm_{[g]}\in \calM_{\msf{g}}(\UG)$ supported on $\gam_{[g]}$ as in \Cref{eq:def.m_g}. Then 
    \begin{equation*}
L(\frakm_{[g]})=\frac{\ell_{\varphi}[g]}{\ell_{\hat d}[g]}.
    \end{equation*}
\end{lemma}

\begin{proof}
Recall that $\hat d$ is the strongly hyperbolic metric used to define $\UG$. Without loss of generality suppose that $g$ is primitive, and consider a precompact Borel fundamental domain $\calB$ containing $\{(g^-,g^+)\}\times [0,\ell_{\hat d}[g])$. Then for $k\in \Z$ and $\vv\in \UG$ with lift $\wh \vv= (g^-,g^+,s)$ and $0\leq s<\ell_{\hat d}[g]$, we get $g_\calB(\vv,k\ell_{\hat d}[g])=g^k$. This observation and \Cref{eq:def.m_g} imply 
$$\int{\P_\calB(\vv,k\ell_{\hat d}[g])}{d\frakm_{[g]}}(\vv)=\frac{1}{\ell_{\hat d}[g]}\int_{\gam_{[g]}}{\varphi(o,g^k)dt}=\varphi(o,g^k),$$ and hence
\[L(\frakm_{[g]})=\lim_{k\to \infty}{\frac{1}{k\ell_{\hat d}[g]}}{\int{\P_\calB(\vv,k\ell_{\hat d}[g])}{d\frakm_{[g]}(\vv)}}=\lim_{k\to \infty}{\frac{\varphi(o,g^k)}{k\ell_{\hat d}[g]}}=\frac{\ell_{\varphi}[g]}{\ell_{\hat d}[g]}. \qedhere\]
\end{proof}

Since real multiples of rational currents are dense in $\calC_\G$, $\ell_{\varphi}$ has at most one continuous extension from conjugacy classes in $\G$ to $\calC_\G$. Therefore, from \Cref{cor.Lcontinuous} and \Cref{lem.Lrational} we get the following, reproving a result of Kapovich and Martinez-Granado \cite{kapovich-martinezgranado}.

\begin{corollary}\label{coro.ellL}
For any $\varphi\in \calH_\G$ there exists a unique continuous extension of $\ell_\p$ from $[\G]'$ to $\calC_\G$. Such an extension satisfies
    \[\ell_{\varphi}(\Lam)=L(\frakm_\Lam)\]
for any $\Lam \in \calC_\G$.
\end{corollary}

\begin{proof}
    From \Cref{eq:intFhatF} it follows that $\frakm_{\eta_{[g]}}=\ell_{\hat d}[g]\frakm_{[g]}$ for any $[g]\in [\G]'$. Then the result follows immediately from \Cref{lem.Lrational} and \Cref{cor.Lcontinuous}. 
\end{proof}

\begin{remark}\label{rmk:extensioninMC} In virtue of \Cref{prop:MC(H)}, the conclusion of the above corollary also holds for $\varphi \in \MC(\calH_\G)$, so that $\ell_\varphi$ also extends continuously (and $\R^+$-linearly) to $\calC_\G$. This is not necessarily true for a left-invariant pseudometric on $\G$ that is quasi-isometric to metrics in $\calD_\G$. For example, take $\p=d+\hat d$ for $d$ a metric in $\calD_\G$ and $\hat d$ a left-invariant metric on $\G$ whose length function $\ell_{\hat d}$ does not admit a continuous extension to $\calC_\G$, see \cite[Proposition~11]{bonahon.currentshpygroups}.
\end{remark}

Recall that we have fixed $\p\in \calH_\G^{++}$. We now recover the mean distortion $\tau(\varphi/\p)$ from $L_{\varphi}$.

\begin{lemma}\label{lem.tauLBM}
      Let $\Lam_{\p}\in \calC_\G$ be a BMS current for $\p$, normalized so that $\frakm_{\p}:=\frakm_{\Lam_{\p}} \in \calM_{\msf{g}}(\UG)$. Then 
    \begin{equation}\label{eq:lemtau}
        L_{\varphi}(\frakm_\p)=\tau(\varphi/\p)/\tau(\hat d/\p).
    \end{equation}
\end{lemma}

\begin{proof}
    Let $\calB\subset \wt\UG$ be any precompact fundamental domain and fix $\vp\in \UG$ with lift $\wh\vp\in \calB$. Then by \Cref{thm:mineyev}~(3), there is a constant $C'>0$ such that for any $\vv\in \UG$ we have 
    $|d_{\wt\UG}(\wh\vp,g_{\calB}(\vv,t)\wt\vp)-t|\leq C'$. The same happens if we replace $d_{\wt\UG}(\wh\vp,g_{\calB}(\vv,t)\wt\vp)$ with $\hat d(o,g_\calB(\vv,t))$ by \Cref{thm:mineyev}~(3). This implies that for $\wt \vv=(a,b,s)$ with $a\neq b$ in $\partial \G$, the map $t \mapsto g_\calB(\vv,t)$ is a uniform quasigeodesic ray converging in $\G$ to $b$. 

    In addition, $\Lam_\p$ is in the measure class of a product of quasi-conformal measures, and hence by \Cref{lem:limittauQC}, for $\Lam_\p$-almost every $(a,b)\in \partial^2\G$ and $\vv\in \UG$ with lift $\wh\vv=(a,b,s)$ we have 
    \[\frac{\varphi(o,g_{\calB}(\vv,t))}{\p(o,g_\calB(\vv,t))} \to \tau(\varphi/\p) \ \ \text{ and } \ \ \frac{\p(o,g_\calB(\vv,t))}{t} \to \tau(\hat d/\p)^{-1} \quad \text{ as }t\to \infty.\]
    In consequence, we have that $$\P_\calB(\vv,t)/t \to \tau(\varphi/\p)/\tau(\hat d/\p) \quad \text{as }t\to \infty$$ for $\frakm_\p$-almost every $\vv$. Then by \Cref{lem.almostcocycle}, Kingman's subadditive ergodic theorem and the fact that $\frakm_\p$ is a probability measure, we deduce \Cref{eq:lemtau}, as desired.
\end{proof}

\begin{proof}[Proof of \Cref{thm.tauvsell}]
    After rescaling $\Lam_\p$, we can assume without loss of generality that $\frakm_\p:=\frakm_{\Lam_\p}$ is a probability measure on $\UG$. Then, applying \Cref{lem.tauLBM} to both $L_{\varphi}$ and $L_\p$ we get
    \begin{equation}\label{eq.proofthmtauvsell}
        \tau(\varphi/\p)=\frac{\tau(\varphi/\p)/\tau(\hat d/\p)}{\tau(\p/\p)/\tau(\hat d/\p)}=\frac{L_{\varphi}(\frakm_\p)}{L_\p(\frakm_\p)}.
    \end{equation}
    By \Cref{coro.ellL}, the right hand side of \Cref{eq.proofthmtauvsell} equals $\ell_{\varphi}(\Lam_\p)/\ell_\p(\Lam_\p)$, and the conclusion follows.
\end{proof}

For the proof of \Cref{thmmain.H_Gnotcomplete} in the next section we require the following consequence of \Cref{thm.tauvsell}, for which we recall the notion of duality from \Cref{def:dualC-H}.

\begin{corollary}\label{cor:dualintegral}
    Let $\p\in \calH_\G$ be dual to $F\in C(\UG)$. Then for any geodesic current $\Lam\in \calC_\G$ we have
\begin{equation*}
    \int{F}{d\frakm_{[\Lam]}=\frac{\ell_{\p}(\Lam)}{\ell_{\hat d}(\Lam)}}.
\end{equation*}
In particular, if $\BMS(\varphi)\in \bbP\calC_\G$ is the BMS current for $\varphi\in \calH_\G$, then we have
\begin{equation*}
    \int{F}{d\frakm_{\BMS(\varphi)}}=\frac{\tau(\p/\varphi)}{\tau(\hat{d}/\varphi)}.
\end{equation*}
\end{corollary}

\begin{proof}
    Since $\p$ and $F$ are dual, the first identity follows immediately from \Cref{eq:def.dualC-H}, \Cref{coro.ellL}, and the density of rational currents in $\calC_\G$. The second identity is a particular case of the first identity, since by \Cref{thm.tauvsell} we have
\begin{align*}
    \int{F}{d\frakm_{\BMS(\varphi)}}& =\frac{\ell_\p(\BMS(\varphi))}{\ell_{\hat d}(\BMS(\varphi))}
     = \frac{\tau(\p/\varphi)}{\tau(\hat d/\varphi)}/\frac{\tau(\hat d/\varphi)}{\tau(\hat d/\varphi)}=\frac{\tau(\p/\varphi)}{\tau(\hat d/\varphi)}. \qedhere
\end{align*}
\end{proof}


\section{Examples}\label{sec:examples}

In this section we use the dictionary developed in previous sections to construct examples of points in $\MC(\scrH_\G)$ satisfying exotic properties, proving \Cref{thmmain.H_Gnotcomplete} and \Cref{maincor:positiveonnonsimpleintro} from the introduction. Through this section we fix a Mineyev's flow space $(\UG,\msf{g})$ associated to the strongly hyperbolic metric $\hat d\in \calD_\G$.

\subsection{Non-completeness of $\scrH_\G^{++}$}\label{subsec:noncomplete} 

We begin by proving \Cref{thmmain.H_Gnotcomplete}. As an obstruction for a point $[\p]\in \MC(\scrH_\G^{++})$ to belong to $\scrH_\G^{++}$, we study the thermodynamic formalism of potentials dual to $\p$.
\begin{definition}
  For a potential $F\in C(\UG)$, its \emph{pressure} is the quantity
 \[\msf{P}_{\msf{g}}(F):=\sup_{\frakm\in \calM_{\msf{g}}(\UG)}\left(h(\frakm)+\int_{\UG}Fd\frakm\right),\] 
  where $h(\frakm)$ denotes the \emph{entropy} of the measure $\frakm$. If $\frakm\in \calM_\vg(\UG)$ satisfies $\msf{P}_{\vg}(F)=h(\frakm)+\int_{\UG}{Fd\frakm}$, we call it an \emph{equilibrium state} for $F$.
\end{definition}

It follows immediately that an equilibrium state for $F$ is also an equilibrium state for $F+s\mathrm{1}_{\UG}$ for any $s\in \R$. The next proposition describes a necessary condition for a potential to be dual to an element in $\calH_\G$.  This is essentially due to Dilsavor \cite[Theorem~F]{dilsavor}, when $F$ is replaced by an almost additive family $(f_t:\UG \ra \R)_{t\in \R^+}$ of measurable potentials satisfying the Bowen property \cite[Definition~4.2.2]{dilsavor}.

\begin{proposition}\label{prop:Gibbs=BMS} Suppose $F\in C(\UG)$ is dual to a potential in $\calH_\G$. Then $F$ has a unique equilibrium state.
\end{proposition}

\begin{proof}
   Let $F$ be dual to $\p\in \calH_\G$. In the language of \cite[Definition~3.1.1]{dilsavor}, $\p$ is an almost $\G$-invariant and roughly geodesic potential, which is locally bounded near the diagonal. Hence by \cite[Theorem~B]{dilsavor}, $\p$ has associated a family $f=(f_t:\UG \ra \R)_{t\geq 0}$ of bounded, almost additive measurable potentials satisfying the coarse Bowen property.

For this family $f$, its pressure is defined as
    \begin{equation}\label{eq:pressuref}
        \msf{P}_{\msf{g}}(f):=\sup_{\frakm\in \calM_{\msf{g}}(\UG)}\left(h(\frakm)+\int_{\UG}fd\frakm\right),
    \end{equation}
    where we have \[\int_{\UG}fd\frakm:=\lim_{t\to \infty}{\frac{1}{t}\int_{\UG}}{f_td\frakm}.\]
    In addition, \cite[Theorem~F]{dilsavor} tells us that $f$ has a unique equilibrium state. Therefore, to show that $F$ has a unique equilibrium state it is enough to prove that
    \[\int{F}{d\frakm}=\int{f}{d\frakm} \quad \text{ for any }\frakm\in \calM_{\vg}(\UG).
    \]

    Moreover, if we follow the proof of \cite[Theorem~B]{dilsavor}, it turns out that we can find a constant $C>0$ and a precompact Borel fundamental domain $\calB\subset \wt\UG$ for the action of $\G$ on $\wt\UG$ such that
    \[|f_t(\vv)-\P_\calB(\vv,t)|\leq C \quad \text{ for all }\vv\in \UG \text{ and }t\geq 0,\]
    for $\P_\calB$ as in \Cref{def:Psi_B}. From this and \Cref{coro.ellL} it follows that 
    \[\int{f}{d\frakm_\Lam}=\ell_{\psi}(\Lam)\]
    for any $\Lam\in \calC_\G$. But \Cref{cor:dualintegral} and the fact that $F$ is dual to $\p$ imply that
    \[\int{F}{d\frakm_\Lam}=\ell_{\psi}(\Lam)\]
    for all $\Lam\in \calC_\G$. Then \Cref{lem:homeoCM} implies that the integrals of $F$ and $f$ coincide for any $\frakm\in \calM_{\vg}(\UG)$, as desired.
\end{proof}

\begin{remark}\label{rmk:BMSalternative}
For $F$ and $\p$ as in the lemma above, it follows from the work in \cite{dilsavor} that if $F>0$ satisfies $\msf{P}_{\vg}(-F)=0$, then the equilibrium state of $-F$ is exactly $\frakm_{\BMS(\p)}$.
\end{remark}

Our second ingredient is the existence of continuous potentials with several equilibrium states. In the case of continuous potentials on transitive one-sided subshifts of finite type, this result goes back to the work of Israel \cite[Theorem~V.2.2]{israel}.

\begin{proposition}\label{prop:israel}
The set of continuous functions $F:\UG \ra \R$ with uncountably many ergodic equilibrium states is dense in $C(\UG)$.
\end{proposition}

As noted in \cite[Remark~3.9]{iommi-velozo}, the same result above holds for any continuous flow $\vg=(\vg_t)_{t\in \R}$ on a compact metric space $X$ satisfying:
\begin{enumerate}
    \item \emph{Finite entropy}: $h(\vg)<\infty$;
    \item \emph{Entropy density}: for any $\vg$-invariant probability measure $\frakm$ on $X$, there exists a sequence of $\vg$-invariant ergodic probability measures $(\frakm_k)_k$ such that $\frakm_k  \xrightharpoonup{\ast} \frakm$ and $h(\frakm_k)\to h(\frakm)$; and,
    \item \emph{Upper semi-continuity of the entropy}: if $(\frakm_k)_k$ is a sequence of $\vg$-invariant probability measures on $X$ weak$^\ast$ converging to $\frakm$, then $\limsup_{k\to \infty}h(\frakm_k)\leq h(\frakm)$.
\end{enumerate}

For the Mineyev's flow space $(X,\vg)=(\UG,\vg)$, we know that the entropy $h(\vg)$ is finite. In addition, upper semi-continuity of the entropy follows from \emph{expansivity}, which always holds for metric-Anosov flows \cite[Theorem~3.2]{CLT.strong}. Thus, in order to prove \Cref{prop:israel}, we are left to prove the entropy density property for the Mineyev's flow space. Before verifying this, we first prove \Cref{thmmain.H_Gnotcomplete} assuming \Cref{prop:israel}.

\begin{proof}[Proof of \Cref{thmmain.H_Gnotcomplete}]
By \Cref{prop:israel}, let $F\in C(\UG)$ be a positive continuous function such that $-F$ has more than one equilibrium state.

Using \Cref{lem:bowendense}, let $(F_k)_k\in C^B(\UG)$ be a sequence such that $\|F_k-F\| \to 0$ as $k\to \infty$. Up to ignoring the first terms of this sequence, we can assume that $F_k>0$ for each $k$, and by \Cref{prop.dictionaryBowen} we let $\p_k\in \calH_\G^{++}$ be dual to $F_k$ for each $k$. Defining $\vx_k:=[\p_k]$ for each $k$, \Cref{lem:convF} tells us that this sequence converges in $\MC(\scrH_\G^{++})$ to $\vx=[\p]$ for some $\p\in \MC(\calH_\G^{++})$ dual to $F$. 
Our starting assumption on $F$ together with \Cref{prop:Gibbs=BMS} imply that $\p$ cannot belong to $\calH_\G^{++}$, and hence $\vx\notin \scrH_\G^{++}$. Thus the sequence $(\vx_k)_k$ is Cauchy in $\scrH_\G^{++}$ with no limit in $\scrH_\G^{++}$.
\end{proof}

\begin{remark}
Indeed, \Cref{prop:israel} implies a stronger conclusion: the complement $\MC(\scrH_\G^{++}) \bs \scrH_\G^{++}$ is dense in $\MC(\scrH_\G^{++})$. 
We leave the details to the reader.
\end{remark}

To complete the proof of \Cref{prop:israel}, we are left to prove the entropy density of the Mineyev's flow space. This property is known to hold for the geodesic flow of a compact locally $\CAT(-1)$ space with non-elementary fundamental group \cite[Theorem~B]{CLT.weak}.
More generally, Constantine--Lafont--Thompson proved that entropy density holds for any continuous flow on a compact metric space that is both expansive and satisfies the weak specification property \cite[Theorem~5.1]{CLT.weak}.

\begin{definition}
Let $\vg=(\vg_t)_{t\in \R}$ be a continuous flow on the compact metric space $(X,d)$. This flow  has \emph{weak specification at scale} $\del>0$ if there exists $\tau > 0$ such that for
every collection of points $x_1,\dots,x_k\in X$ and times $t_1,\dots,t_k>0$ there exists a point $y$ and a
sequence of transition times $\tau_1,\dots,\tau_{k-1}\in [0, \tau]$ such that for $s_j = \sum_{i=1}^j{t_i} + \sum_{i=1}^{j-1}{\tau_i}$ and $s_0 = \tau_0 = 0$, we have
\begin{equation*}
   d(\vg_{s_{j-1}+\tau_{j-1}+s}(y), \vg_s(x_j)) < \del \quad \text{ for every } 1\leq j\leq k \text{ and }0\leq s\leq t_j.  
\end{equation*}
In that case, we say that $\tau$ is a \emph{maximum transition time}. We say that $\vg$ has the \emph{weak specification property} if it has weak specification at every scale $\del > 0$.
\end{definition}

\begin{proposition}\label{prop:weakspecification}
    The Mineyev's flow space $(\UG,\vg)$ has the weak specification property. 
\end{proposition}

We require the following lemma regarding time reparameterizations of flow lines.

\begin{lemma}\label{lem:timechange}
    There exists a constant $c_0>0$ satisfying the following. Given $0<c\leq c_0$, let $\vv,\vw\in \UG$ and $T>0$ be  such that 
    \[d_{\UG}(\vg_t(\vv),\vg_{\k(t)}(\vy))<c \quad \text{ for all }0\leq t\leq T, \]
    for $\k:[0,T]\ra \R$ a continuous and strictly increasing function with $\k(0)=0$. Then \[|\k(t)-t|\leq 2c \quad \text{ and }\quad d_{\UG}(\vg_t(\vv),\vg_t(\vw))< 3c \quad  \text{ for all } 0\leq t \leq T.\]
\end{lemma}

\begin{proof}
    By compactness of $\UG$, let $c_0/2>0$ be such that $d_{\wt\UG}(\wt\vv,g\wt\vv)\geq c_0$ for any $\wt\vv\in \wt\UG$ and $g\neq o$ in $\G$. Now consider $\vv,\vw\in \UG$ and $\k:[0,T]\ra \R$ as in the statement, and consider lifts $\wt\vv,\wt\vw\in \wt\UG$ of $\vv$ and $\vw$ respectively and such that $d_{\wt\UG}(\wt\vv,\wt\vv)< c$. By connectedness of $[0,T]$ and our choice of $c_0$, we must have $d_{\wt\UG}(\wt\vg_{t}(\wt\vv),\wt\vg_{\k(t)}(\wt\vw))<c$ for all $0\leq t\leq T$. But then \Cref{thm:mineyev}~(3) implies that $|t-\k(t)|<2c$ for all $0\leq t\leq T$, and that $d_{\wt\UG}(\wt\vg_{t}(\wt\vv),\wt\vg_{t}(\wt\vw))\leq 3c$ for any $0\leq t\leq T$. Projecting these flow lines to $\UG$ implies the lemma.
\end{proof}

Now we prove \Cref{prop:weakspecification}. The argument is very similar to that of \cite[Theorem~3.2]{CLT.weak}. However, we do not rely on $\CAT(-1)$ geometry, so we provide a proof for completeness.

\begin{proof}
By \cite[Theorem~1.1]{cantrell-tanaka.invariant}, there exists a suspension flow $(Y,f=(f_t)_{t\geq 0})$ with positive H\"older continuous roof function over a transitive subshift of finite type and a continuous and surjective semi-orbit preserving map $h:(Y,f)\ra (\UG,\vg)$. For the appropriate metric on $Y$, we have that $(Y,f)$ has the weak specification property (this follows from transitivity of the base subshift, as in \cite[Proposition~2.2]{CLT.weak}).

Since the flow $\vg$ has no fixed-points in $\UG$, by \cite[Proposition~2.4]{CLT.weak} the function $\k: Y \times \R^+\ra \R^+$ such that $h(f_t(y))=\vg_{\k(y,t)}(h(y))$ is continuous.

Given a scale $\del>0$, we fix the following quantities.
\begin{itemize}
\item A constant $c_0>0$ given by \Cref{lem:timechange}.
\item Using that $h$ is uniformly continuous, we find $\wh\del>0$ such that $$d_{\UG}(h(y_1),h(y_2))<\min\{\del/3,c_0/2\}=:\ep_0$$ for $y_1,y_2\in Y$ such that $d(y_1,y_2)<\wh\del$.

    \item A maximum transition time $\wh\tau$ at scale $\wh\del$ in $(Y,f)$.
    \item A number $R>0$ such that for any $y\in Y$ any $0\leq s \leq \wh\tau$, we have $\k(y,s)\leq R$. Such an $R$ exists by \cite[Corollary~2.5]{CLT.weak}.
\end{itemize}

Now we fix $\vv_1,\dots,\vv_k\in \UG$ and $t_1,\dots,t_k>0$. We consider points $y_1,\dots,y_k$ in $Y$ and times $s_1,\dots,s_k>0$ such that $h(y_i)=\vg_{-\ep_0}(x_i)$ and $h(f_{s_i}(y_i))=\vg_{t_i+\ep_0}(x_i)$ for all $i$. That is, we have $\k(y_i,s_i)=t_i+2\ep_0$ for all $i$. By weak specification at scale $\wh\del$ on $(Y,f)$, there exists $z\in Y$ and transition times $\wh\tau_1,\dots,\wh\tau_k\in [0,\wh\tau]$ such that for $\wh{r}_j=\sum_{i=1}^j{s_j}+\sum_{i=1}^{j-1}{\wh\tau_i}$ and $r_0=\wh\tau_0=0$, we have 
\begin{equation}\label{eq:defWEPproof}
   d(f_{\wh{r}_{j-1}+\wh\tau_{j-1}+s}(z), f_s(y_j)) < \wh\del \quad \text{ for every } 1\leq j\leq k \text{ and }0\leq s\leq s_j.  
\end{equation}

We set $\k(t)=\k(z,t)$ and $\vw=\vg_{\ep_0}(h(z))$. We claim that the partial orbits of $\vv_1,\dots,\vv_k$ stay close to the orbit $\vw$ in the expected way. 

We set $\tau:=R+4\ep_0$, $t_0=\tau_0=0$, and inductively we construct a sequence $\tau_1,\dots,\tau_k\in [0,\tau]$ such that for $r_j:=\sum_{i=1}^{j}t_j+\sum_{i=1}^{j-1}\tau_i$, we have for $1\leq j\leq k$:
\begin{itemize}
\item[(a)]$|\k(\wh{r}_{j})-r_{j}-2\ep_0|\leq 2\ep_0$; and,
    \item[(b)] $d_{\UG}(\vg_{
    r_{j-1}+\tau_{j-1}+t}
    (\vw),\vg_t(\vv_j))<3\ep_0$ for all $0\leq t\leq t_j$.
\end{itemize}

For the base case $j=1$, note that $\wh{r}_1=s_1$ and $r_1=t_1$. We use \Cref{eq:defWEPproof} for $j=1$ and our choice of $\hat\del$ to deduce
$$d_{\UG}(\vg_{\k(s)-\ep_1}(\vw),\vg_{\k(y_1,s)-\ep_1}(\vv_1))<\ep_0<c_0$$ for $0\leq s\leq s_1$. By \Cref{lem:timechange}, we deduce that $|\k(\wh{r}_1)-t_1-2\ep_0|=|\k(s_1)-\k(y_1,s_1)|\leq 2\ep_0$ and $d_{\UG}(\vg_t(\vw),\vg_t(\vv_1))<3\ep_0$ for all $0\leq t\leq t_1$. 

Assume now that we have found $\tau_1,\dots,\tau_{j-1}\in [0,\tau]$ satisfying (a) and (b) above for $1\leq k\leq j$. We define $\tau_j$ as the solution to 
\[\k(\wh{r}_{j}+\wh\tau_j)=r_j+\tau_j.\]
By our inductive assumption, note that $\k(\wt{r}_j+\wh{\tau}_j)\geq \k(\wh{r}_j)\geq r_j$, so that $\tau_j\geq 0$. The inductive assumption and the definition of $R$ also imply
$\tau_j=\k(\wh{r}_j+\wh\tau_j)-r_j\leq \k(\wh{r}_j)-r_j+R\leq R+4\ep_0=\tau$, so that $\tau_j\in [0,\tau]$.  As before, we use \Cref{eq:defWEPproof} and our choice of $\hat\del$ to obtain
$$d_{\UG}(\vg_{\k(\wh{r}_{j}+\wh\tau_{j}+s)}(\vw),\vg_{\k(y_{j+1},s)}(\vv_{j+1}))<\ep_0<c_0$$
for $0\leq s\leq s_{j+1}$. Since $\k(\wh{r}_j+\wh\tau_j)=r_j+\tau_j$ and $\k(y_{j+1},s_{j+1})=t_{j+1}+2\ep_0$, from \Cref{lem:timechange} we deduce
$|\k(\wh{r}_{j+1})-(r_j+\tau_j)-t_{j+1}-2\ep_0|=|\k(\wh{r}_{j+1})-r_{j+1}-2\ep_0|\leq 2\ep_0$ 
and $$d_{\UG}(\vg_{t+r_j+\tau_j}(\vw),\vg_t(\vv_{j+1}))<3\ep_0.$$ These inequalities imply (a) and (b), and hence weak specification at scale $\del$ since we chose $3\ep_0\leq \del$.  
\end{proof}


\subsection{Metrics vanishing on simple geodesics on a surface and on simple elements of free groups}\label{subsec:exotic}

In this subsection we construct two exotic metrics: a metric on a surface group with zero stable translation length precisely on simple closed curves (\Cref{prop:positiveonnonsimple}, implying \Cref{maincor:positiveonnonsimpleintro}), and an analogous metric on a free group of rank $r \geq 2$ (\Cref{prop:positiveonnonsimple_free}). We recall the subset $\partial \calD_\G \subset \calH_\G$ from \Cref{sec:HMP}.

\begin{proposition}\label{prop:positiveonnonsimple}
    Let $\G$ be the fundamental group of the closed hyperbolic surface $\Sig$. Then there exists a symmetric pseudometric $d\in \partial \calD_\G$ such that:
    \begin{enumerate}
        \item $\ell_{d}[g]=0$ for any $[g]\in [\G]'$ representing a simple closed geodesic.
        \item $\ell_{d}[g]>0$ for any $[g]\in [\G]'$ representing a non-simple closed geodesic.
\end{enumerate}
Moreover, $d$ also satisfies:
\begin{enumerate}
 \item[(3)] $d$ is not roughly similar to the a pseudometric induced by an action of $\G$ on a real tree; and,
\item[(4)] $\ell_{d}$ cannot be recovered as  the intersection number with a geodesic current.
\end{enumerate}
\end{proposition}

This proposition implies \Cref{maincor:positiveonnonsimpleintro}, since there is a $\G$-equivariant isometric embedding of (the metric identification of) the pseudometric space $(\G,d)$ into a geodesic Gromov hyperbolic symmetric metric space $X$ \cite{lang}. Moreover, since $d$ is roughly geodesic \cite[Lemma~6.2]{cantrell-reyes.manhattan}, the action of $\G$ on $X$ is cobounded, and hence it has bounded backtracking by \cite[Lemma~6.3]{cantrell-reyes.manhattan}.

Indeed, we will first prove a general result for arbitrary hyperbolic groups. This relies in the notion of lamination.

\begin{definition}
    An \emph{algebraic lamination} is a closed $\G$-invariant subset of $\ppG$. Such a lamination  $\calL$ is \emph{Mather} if $\calL=\bigcup_{\Lam}\supp(\Lam)$, where the union runs among all the geodesic currents $\Lam \in \calC_\G$ satisfying $\supp(\Lam) \subset \calL$.
\end{definition}

For free groups, (flip-invariant) algebraic laminations were discussed in \cite{CHL.I,CHL.II,CHL.III}. The notion of Mather lamination resembles that of the Mather set for a continuous potential on a continuous dynamical system, see e.g. \cite{morris.mather}. The relation between $\calH_\G$ and Mather laminations is given by the following result.

\begin{proposition}\label{prop:mather}
If $\calL\subset \ppG$ is an algebraic lamination, then there exists $\p\in \partial\calH_\G^{+}$ such that for any geodesic current $\Lam \in \calC_\G$, we have $\ell_\p(\Lam)=0$ if and only if $\supp(\Lam)\subset \calL$. If $\calL$ is Mather, we also have
$$\calL=\bigcup\{\supp(\Lam): \ell_\p(\Lam)=0\}.$$
Moreover, if $\calL$ is flip-invariant  then we can choose $\p\in \partial \calD_\G$.
\end{proposition}

If $\calL$ and $\p$ satisfy the above conditions, we say that the lamination $\calL$ is \emph{dual} to $\p$. For small actions of free groups on real trees, dual algebraic laminations were constructed in \cite{CHL.II}. 

\begin{proof}
    Let $\calL$ be an algebraic lamination and consider the set $\calK=\pi(\calL\times \R) \subset \UG$ (recall that $\pi:\wt\UG \ra \UG$ denotes the covering projection). Since $\calL$ is closed and $\G$-invariant, $\calK$ is compact and $\vg$-invariant (and $\iota$-invariant if $\calL$ is flip-invariant). Let $F:\UG \ra \R$ be the distance function to $\calK$ for the metric $d_{\UG}$, and note that $F$ is Lipschitz and nonnegative (and $\iota$-invariant if $\calK$ is $\iota$-invariant since $d_{\UG}$ is $\iota$-invariant).
    
    Note that $\calK$ is a $(-F)$\emph{-maximizing set} \cite{morris.sufficient}. That is, a measure $\frakm\in \calM_{\vg}(\UG)$ satisfies $\int{F}{d\frakm}=0$ if and only if $\supp(\frakm)\subset \calK$. Applying \Cref{lem:homeoCM} and the duality from \Cref{eq:def.dualC-H}, we have that $\Lam\in \calC_\G$ satisfies $\ell_\p(\Lam)=0$
    if and only if $\supp(\Lam)\subset \calK$. Furthermore, if  $\calL$ is Mather, then 
    \begin{align*}
        \calL & =\bigcup\{\supp(\Lam):\Lam\in \calC_\G \text{ and }\supp(\Lam)\subset \calL\} \\
        & =\bigcup\{\supp(\Lam):\Lam\in \calC_\G \text{ and }\ell_\p(\Lam)=0\}.
    \end{align*}
    Finally, if $\calL$ is flip-invariant, then we can choose $\p\in \partial \calD_\G$ by \Cref{prop.dictionaryBowen}.
\end{proof}

\begin{proof}[Proof of \Cref{prop:positiveonnonsimple}]
    We let $\calL\subset \ppG$ be the set of pairs of endpoints of geodesics in the universal cover of the surface $\Sig$ that project to simple geodesics in $\Sig$. This set is a flip-invariant, closed, and $\G$-invariant, hence an algebraic lamination. Then \Cref{prop:mather} gives us a pseudometric $d\in \partial \calD_\G$ such that a geodesic current $\Lam\in \calC_\G$ satisfies $\ell_\p(\Lam)=0$ if and only if $\supp(\Lam)\subset \calL$. In particular, $\ell_\p[g]=0$ if and only if $[g]$ represents a simple closed geodesic in $\Sig$, which settles the first two items.

    For the last items of the statement, consider an embedded pair of pants $P \subset \Sig$ with fundamental group generated by standard generators $g,h$,
whose conjugacy classes correspond to simple closed curves $[g], [h]$ homotopic to two of the boundary components of $P$. Let
$gh$ represent a self-intersecting figure eight curve in $P$, and $gh^{-1}$ represent a simple closed curve corresponding to the third
boundary component of $P$. 
In~\cite[Section~1.11]{CM87:GroupActions} it is shown that the translation length $$ \G \ni k\mapsto\|k\| \coloneqq \inf_{p \in T} d_T(p,k\cdot p)$$ of an isometric action of $\G$ on a real tree $T$ must satisfy that for every
pair $k_1,k_2 \in \G$, that $\|k_1k_2\| = \|k_1k_2^{-1}\|$, or
$\|k_1\| + \|k_2 \| \geq \max \{\|k_1k_2 \|, \|k_1k_2^{-1}\|\}$.
We also note that the hypothesis can be preserved under equivariant rough isometry, by~\cite[Theorem~3.7]{CM87:GroupActions} and from the fact that for a real tree the translation length and the stable length are equal.
We see that, for this pair $(g,h)$, $\ell_d$ violates Culler-Morgan's axioms. Hence, since for a real tree translation length and stable length are equal, it follows that $\ell_d$ cannot be the stable length of a real tree, proving the third item.

In~\cite[Theorem~A']{MGT26:Intersections} it is shown that for a stable length to be an intersection number with a geodesic current,
a necessary condition is that for any pair of elements $g, h \in \G$ whose hyperbolic axes cross on the hyperbolic disk, one must have
$\ell_d[g] + \ell_d[h] \geq \ell_d[gh]$. This, again, is violated by $\ell_d$, so it cannot arise as an intersection number with a geodesic current.
\end{proof}

The analog of \Cref{prop:positiveonnonsimple} for free groups is the following. An element in a free group $\G$ is \emph{simple} if it is not contained in a proper free factor of $\G$.

\begin{proposition}\label{prop:positiveonnonsimple_free}
    Let $\G$ be a free group of rank at least $2$. Then there exists a pseudometric $d\in \partial \calD_\G$ such that:
    \begin{enumerate}
        \item $\ell_{d}[g]=0$ for any $[g]\in [\G]'$ representing a simple element.
        \item $\ell_{d}[g]>0$ for any $[g]\in [\G]'$ representing a non-simple element.
\end{enumerate}
\end{proposition}

\begin{proof}
   Let $\calL$ be the closure in $\ppG$ of the set $\{(g^-,g^+) \in \partial^2 \G : g \text{ is primitive}\}$, so that $\calL$ is a  flip-invariant algebraic lamination. 
   Then \Cref{prop:mather} gives us a pseudometric $d\in \partial \calD_\G$ such that a geodesic current $\Lam\in \calC_\G$ satisfies $\ell_\p(\Lam)=0$ if and only if $\supp(\Lam)\subset \calL$. In addition, By \cite[Lemma~5.4]{GuirardelHorbez2019} we have that $(g^-,g^+)\notin \calL$ for $g$ non-simple. In particular, $\ell_\p[g]=0$ if and only if $[g]$ represents a simple element in $\G$.
    \end{proof}



\section{Density of Green metrics}\label{sec:proofgreen}

In this section we prove \Cref{mainthm.greendense}.  For the reader's convenience, we restate the result here. Recall that for a probability measure $\mu$ on the group $\G$, $d_\mu$ denotes the associated Green metric on $\G$.

\begin{theorem}[Density of Green metrics]\label{thm.greendense}
    Let $\G$ be a non-elementary hyperbolic group and $d\in \calD_\G$. Then there exists a constant $\del_1>0$ satisfying the following. Let $\del>\del_1$ and $\mu_l$ be the uniform probability measure supported on $S_l=\{g\in \G \colon l-\del<d(o,g)\leq l\}$. Then $S_l$ generates $\G$ for $l$ large enough and $[d_{\mu_l}]$ converges to $[d]$ in $\scrD_\G$.
\end{theorem}

Each set $S_l$ as in the above statement is a ``thickened sphere'' of radius $l$ around $o$ in $(\G,d)$, and $\m_l$ is the uniform probability measure on $l$. A sketch of this  approximation result is the following. 
\begin{enumerate}
    \item We give an upper bound of $d_{\mu_l}$ in terms of $d_{S_l}$, the word metric with respect to $S_l$ (\Cref{lem.greeneasyineq}). This only requires mild assumptions on $\G$.
    \item We give a lower bound of $d_{\mu_l}$ in terms of $d$ and $d_{S_l}$ (\Cref{prop.greenmetricineqHYP}). This uses hyperbolicity of $\G$ in a crucial way, as we rely on counting results from \cite{CDS}. 
    \item For $l$ large enough, we show that $d$ is well-approximated (in $\scrD_\G$) by $d_{S_l}$ (\Cref{lem.worddensesphere}). Again, this result does not require hyperbolicity of $\G$, but uses that $d$ is roughly geodesic. This approximation upgrades the bounds in (1) and (2) to bounds of $d_{\mu_l}$ in terms of $d$ only. 
    \item As $l$ tends to infinity, these bounds become more precise (up to additive errors), implying the convergence $[d_{\mu_l}]\to [d]$ in $\scrD_\G$. 
\end{enumerate}

We begin by showing that word metrics on thickened spheres approximate roughly geodesic metrics. This is analogous to \cite[Lemma~2.5~(2)]{cantrell-reyes.approx}, where it is shown that word metrics on large balls approximate roughly geodesic metrics.

\begin{lemma}\label{lem.worddensesphere}
    Let $\G$ be any group and let $d$ be an unbounded left-invariant $\al$-roughly geodesic (non-necessarily symmetric) metric on $\G$. Then for any $l>\del>2(\lceil\al \rceil+\al+2)$ the set $S_l=\{g\in \G: l-\del<d(o,g)\leq l\}$ generates $\G$ as a semi-group and there exists $M=M_{l,\al}>0$ such that
\begin{equation}\label{eq.WMdensesphere}
        (l-\lceil{\al \rceil})|g|_{S_l}-M\leq d(o,g)\leq l|x|_{S_l} \quad \text{ for all }g\in \G.
    \end{equation}
\end{lemma}
\begin{proof}
  Since $d$ is $\al$-rough geodesic, the ball $B=\{g\in \G \colon d(o,g)\leq  2+\al\}$ generates $\G$ as a semi-group and satisfies \begin{equation}\label{eq.WMdenseball}
      (\lceil{\al \rceil}-\al+1)|g|_B\leq d(o,g)+\lceil{\al \rceil}+1
  \end{equation} for all $g\in \G$, see \cite[Lemma~2.5~(2)]{cantrell-reyes.approx}. Therefore, since $d$ is unbounded and $\al$-roughly geodesic and $l>\del>2(\al+\lceil \al \rceil+1)$, we can find $h\in \G$ with $|d(o,h)-(l-2\lceil{\al \rceil}-2)|\leq \al$. In particular, if $b\in B$, then both $h,hb$ belong to $S_l$ and $|b|_{S_l}\leq 2$, implying
  \begin{equation}\label{eq.ineqS-B}
      |g|_{S_l}\leq 2|g|_B \quad \text{ for all }g\in \G.
  \end{equation}

  Now we prove \Cref{eq.WMdensesphere}, for which the second inequality holds easily by the triangle inequality. To prove the first inequality in \Cref{eq.WMdensesphere}, we let $o=g_0,\dots,g_m=g$ be an $\al$-rough geodesic joining $o$ and $g\in \G$, and write $m=q(l-\lceil{\al \rceil})+r$ with $q,r$ integers and $0\leq r<l-\lceil{\al \rceil}$. We set $z_1=g_{l-\lceil{\al \rceil}}, z_j=g_{(j-1)(l-\lceil{\al \rceil})}^{-1}g_{j(l-\lceil{\al \rceil})}$ for $2\leq j\leq q$, and $\hat z=z_q^{-1}g=g_{q(l-\lceil{\al \rceil})}^{-1}g_{q(l-\lceil{\al \rceil})+r}$. The definition of the $g_i's$ implies that $$d(o,z_j )\in [l-\lceil{\al \rceil}-\al,l-\lceil{\al \rceil}+\al]\subset (l-\del,l]$$ for all $j$, where we used $l-\al>2(\lceil \al\rceil+\al+2)$. This yields 
  \begin{align*}
      |g|_{S_l} \leq q+|\hat z|_{S_l}\leq (l-\lceil{\al \rceil})^{-1}m+|\hat z|_{S_l}
       \leq (l-\lceil{\al \rceil})^{-1}(d(o,g)+\al)+|\hat z|_{S_l}.
  \end{align*}  
  In addition, we also have $d(o,\hat z)=d(g_{q(l-\lceil{\al \rceil})},g_m)\leq l$, and so from \Cref{eq.WMdenseball} and \Cref{eq.ineqS-B}  we deduce
  \[|\hat z|_{S_l}\leq 2|\hat z|_B\leq (l+\lceil{\al \rceil}+1)(\lceil{\al \rceil}-\al+1)^{-1}.\]
  Combining the last two inequalities we obtain the first inequality in \Cref{eq.WMdensesphere} with $M=(l+\lceil{\al \rceil}+1)(l-\lceil{\al \rceil})(\lceil{\al \rceil}-\al+1)^{-1}+\al$.
\end{proof}

For $\G$ hyperbolic and $d\in \calD_\G$, the lemma above has the following corollary.

\begin{corollary}\label{coro.convEGR}
    Let $\G$ be non-elementary hyperbolic and $d\in \calD_\G$ be $\al$-roughly geodesic. Then there exists a constant $\beta>0$ depending only on $d$ such that for any $l>\del >2(\lceil \al \rceil+\al+2)$ we have
    \[l^{-1}(\log(\# S_l)-\beta)\leq v_d\leq (l-\lceil \al\rceil)^{-1}\log(\# S_l).\]
\end{corollary}

\begin{proof}
  Given a (symmetric) finite generating set $S\subset \G$, let $v_S$ denote the exponential growth rate of the word metric $d_S$. Then we have the immediate bound \begin{equation}\label{eq.1boundeasy}
      v_S\leq \log(\# S). 
  \end{equation}
  On the other hand, by \cite[Theorem~1]{arzhantseva-lysenok} there exists a constant $\beta$ depending only on $\G$, such that for any such $S$ we have \begin{equation}\label{eq.AL}
      v_S \geq \log(\# S)-\beta.
  \end{equation} 
  Now, given $l>\del >2(\lceil \al \rceil+\al+2)$, 
  from \Cref{lem.worddensesphere} we obtain the inequalities
\[l^{-1}v_{S_l}\leq v_d\leq (l-\lceil \al\rceil)^{-1}v_{S_l}.\]
These inequalities together with \Cref{eq.1boundeasy} and \Cref{eq.AL} imply the corollary.
\end{proof}

We continue by giving an upper bound for Green metrics in terms of the word metrics defined by their supports. If $d_\mu$ is the Green metric on $\G$ associated to the probability measure $\mu$, we have the formula
\[d_\mu(g,h)=-\log\left(\frac{G_\mu(g,h)}{G_\mu(o,o)}\right) \quad \text{ for }g,h\in \G,\]
where $G_\mu(g,h)=\sum_{k\geq 0}{\mu^{\ast k}(g^{-1}h)}$ is the \emph{Green function} and $\mu^{\ast k}$ denotes the $k$th-convolution of $\mu$ \cite[p.~686]{BHM.harmonic}.

\begin{lemma}\label{lem.greeneasyineq}
Let $\G$ be an infinite finitely generated group that is not virtually $\Z$ or $\Z^2$ and let $S\subset \G$ be a finite set generating $\G$ as a semi-group. If $\m=\m_S$ denotes the uniform probability measure supported on $S$, then for all $g\in \G$ we have
\begin{equation*}
    d_\mu(o,g)\leq \log\left(\#S\right)|g|_S+\log\left(G_\m(o,o)(1-(\#S)^{-1})\right).
\end{equation*}
\end{lemma}

\begin{proof}
    The assumptions on $\G$ imply that $G_\m(o,o)$ is finite and $\#S\geq 2$. Then for an arbitrary $g\in \G$ we have
    \begin{align*}
        e^{-d_\m(o,g)} & =G_\m(o,o)^{-1}G_\m(o,g)=G_\m(o,o)^{-1}\cdot \sum_{k\geq 0}{\m^{\ast k}(g)} \\
        & =G_\m(o,o)^{-1} \cdot \sum_{k\geq |g|_S}{\m^{\ast k}(g)}\\
        & \geq G_\m(o,o)^{-1}\cdot \sum_{k\geq |g|_S}{\frac{1}{(\#S)^k}} =G_\m(o,o)^{-1}\frac{(\#S)^{-|g|_S}}{1-(\#S)^{-1}}.    \end{align*}
        This last inequality implies the lemma.
\end{proof}

Our next proposition is the key inequality in the proof of \Cref{thm.greendense}, and relies crucially on the hyperbolicity of $\G$.

\begin{proposition}\label{prop.greenmetricineqHYP}
    Let $\G$ be a non-elementary hyperbolic group and consider $d\in \calD_\G$. Then there exists $\del_0$ satisfying the following. Let $\del>\del_0$ and for $l$ large enough define $S_l=S_l^\del:=\{g\in \G \colon l-\del< d(o,g)\leq l\}$. Then $S_l$ generates $\G$ (as a group) and if $\m_l$ is the uniform probability measure supported on $S_l$, then there exists $\k_l>0$ such that for all $g\in \G$ we have \begin{equation}\label{eq.propdGreen}
            d_{\m_l}(o,g)\geq \frac{v_dd(o,g)}{2}+|g|_{S_l}\cdot \left(\log\left(\frac{\#S_l}{D_0\left(\frac{l}{\del}+1\right)}\right)-\frac{v_d l}{2}\right)-\k_l.
        \end{equation}
\end{proposition}

\begin{proof}
    Given $l,n,\del>0$ and $g\in \G$, we consider the set 
    $$\calO_l^\del(g,n)=\{(g_1,\dots,g_n)\in (S_l^\del)^{n}\colon g_1\cdots g_n=g\}.$$
     From \cite[Lemma~B.3]{CDS} we can find $\del_0'>0$ such that if $l>\del>\del_0'$, then there exists $C=C_\del>0$ satisfying 
    \begin{equation}\label{eq.ineqO}
        \#\calO_l(g,n)\leq C^n\left(\frac{l}{\del}+1\right)^ne^{v_d\left(\frac{nl-d(o,g)}{2}\right)}
    \end{equation}
    for all $g\in \G$. The proof of \cite[Lemma~B.3]{CDS} assumes that $d$ is induced by a geometric action of $\G$ on a geodesic metric space $X$.  However, cocompactness can be relaxed to coboundedness and the space $X$ only requires to be roughly geodesic. Therefore, we can apply that lemma to the left action of $\G$ on $X=(\G,d)$.

    Now we set $\del_0:=\max\{2(\lceil{\al\rceil}+\al+2),\del_0'\}+1$, where $\al$ is a rough geodesic constant for $d$. Hence \Cref{lem.worddensesphere} implies that $S_l=S_l^\del$ generates $\G$ for $\del>\del_0$. Moreover, for any $l,n>0$ and $g\in \G$ we have
    $$\m_{l}^{\ast n}(g)=\sum_{(g_1,\dots,g_n)\in \calO_{l}(g,n)}{\m_l(g_1)\cdots \m_l(g_n)}=\#\calO_{l}(g,n)(\#S_l)^{-n}.$$
    Let $C=C_\del$ be the constant from \Cref{eq.ineqO}. By \Cref{coro.convEGR} we have that $\log (\# S_l) /l$ converges to $v_d$ as $l$ tends to infinity, so that for all $l$ large enough we have $C(\frac{l}{\del}+1)e^{\frac{v_dl}{2}}(\# S_l)^{-1}<1$. Therefore, there exists $\k=\k_l>0$ such that for all $g\in \G$ we have
    \begin{align*}
         e^{-d_{\m_l}(o,g)} & =G_{\m_l}(o,o)^{-1}\cdot \sum_{n\geq |g|_{S_l}}{\m_l^{\ast n}(g)}\\
         & \leq G_{\m_l}(o,o)^{-1}\cdot 
         \sum_{n\geq |g|_{S_l}}{C^n\left(\frac{l}{\del}+1\right)^ne^{v_d\left(\frac{nl-d(o,g)}{2}\right)}(\# S_l)^{-n}}\\
         & \leq e^{\k_l}e^{\frac{v_d d(o,g)}{2}}{\left(C\left(\frac{l}{\del}+1\right)e^{\frac{v_dl}{2}}(\# S_l)^{-1} \right)^{|g|_{S_l}} }.
    \end{align*}
    After rearranging, this inequality becomes \Cref{eq.propdGreen}. 
\end{proof}

\begin{proof}[Proof of \Cref{thm.greendense}]
    Let $d\in \calD_\G$, which we assume to be $\al$-roughly geodesic. We fix $\del>\max(\del_0,2(\lceil \al\rceil+\al+2)),$ so that both \Cref{lem.worddensesphere} and  \Cref{prop.greenmetricineqHYP} apply. For $l>0$ large enough and $g\in \G$ this gives us 
    \begin{align*}
                d_{\m_l}(o,g) &\geq \frac{v_dd(o,g)}{2}+|g|_{S_l}\cdot \left(\log\left(\frac{\#S_l}{C\left(\frac{l}{\del}+1\right)}\right)-\frac{v_d l}{2}\right)-\k_l\\
                & \geq  \frac{v_dd(o,g)}{2}+d(o,g)\cdot \left(\frac{\log(\# S_l)}{l}-\frac{\log(C\left(\frac{l}{\del}+1\right))}{l}-\frac{v_d}{2}\right)-\k_l\\
                & = d(o,g)\cdot \left(\frac{\log(\# S_l)}{l}-\frac{\log(C\left(\frac{l}{\del}+1\right))}{l}\right)-\k_l.
    \end{align*}    
In particular, we obtain
    \begin{equation}\label{eq.ineqDil1}
        \Dil(d,d_{\m_l})\leq \left(\frac{\log{\# S_l}}{l}-\frac{\log(C(\frac{l}{\del}+1))}{l}\right)^{-1},
    \end{equation}
 and the right hand side of \Cref{eq.ineqDil1} tends to $v_d^{-1}$ as $l$ tends to  $\infty$ by \Cref{coro.convEGR}.

 Also, applying \Cref{lem.greeneasyineq} to $S=S_l$ and \Cref{lem.worddensesphere} we get
\begin{align*}
    d_{\m_l}(o,g)& \leq \log(\# S_l)|g|_{S_l}+\log(G_{\m_l}(o,o)^{-1}(1-(\#S_l)^{-1})) \\
    & \leq \frac{\log(\# S_l)}{(l-\lceil \al\rceil)}\cdot d(o,g)+\log(G_{\m_l}(o,o)^{-1}(1-(\#S_l)^{-1}))+\frac{M}{(l-\lceil \al \rceil)}, 
\end{align*}
and hence
\begin{equation}\label{eq.greenDil2}
    \Dil(d_{\m_l},d)\leq  \frac{\log(\# S_l)}{(l-\lceil \al\rceil)}.
\end{equation}
 The right hand side of \Cref{eq.greenDil2} tends to $v_d$ as $l$ tends to infinity by \Cref{coro.convEGR}, implying that  $\Del([d_{\lam_l}],[d])=\log(\Dil(d_{\lam_l},d)\Dil(d,d_{\lam_l}))$ tends to $0$ as $l$ tends to infinity, as desired.
\end{proof}

\medskip

\bibliographystyle{alpha}
\bibliography{references.bib}

\end{document}